\documentclass[11pt,a4paper,reqno]{amsart}
\usepackage[english]{babel}
\usepackage[T1]{fontenc}
\usepackage{verbatim}
\usepackage{palatino}
\usepackage{amsmath}
\usepackage{mathabx}
\usepackage{amssymb}
\usepackage{amsthm}
\usepackage{amsfonts}
\usepackage{graphicx}
\usepackage{esint}
\usepackage{color}
\usepackage{mathtools}
\usepackage{overpic}

\usepackage[colorlinks = true, citecolor = black]{hyperref}
\author{Amir Algom, Federico Rodriguez Hertz, and Zhiren Wang}
	
	\address{Department of Mathematics, University of Haifa at Oranim, Tivon 36006, Israel}

\email{\href{mailto:amir.algom@math.haifa.ac.il}{amir.algom@math.haifa.ac.il}}

\address{Department of Mathematics\\
	The Pennsylvania State University\\
	107 McAllister Building\\
	University Park, PA 16802\\
	USA}

\email{\href{mailto:fjr11@psu.edu}{fjr11@psu.edu}}

\address{Department of Mathematics\\
	Johns Hopkins University\\
	3400 N. Charles Street\\
	Baltimore, MD 21218\\
	USA}
	
	\email{\href{mailto:zhirenw@jhu.edu}{zhirenw@jhu.edu}}

\title{Smooth projections of self-similar measures}
\date{\today}
\subjclass[20200]{28A80 (primary) 42B10 (secondary)}
\keywords{Projections, self-similar measures, Fourier decay, spectral gap, spherical averages.}
\thanks{A.A. is supported by  the Israel Science Foundation (Grant No. 392/25),   NSF-BSF Grant No. 2024692, and  Grant No. 2022034 from the United States-Israel Binational Science Foundation (BSF), Jerusalem, Israel. \newline
F. RH. is partially supported by NSF Grant No.~2453688 and by the Anatole Katok Chair in Mathematics. \newline
Z.W is partially supported by NSF grant No. 2453689.}

\newcommand{\supp}{\operatorname{supp}}

\newcommand{\E}{\mathbb{E}}

\def\Barint_#1{\mathchoice
          {\mathop{\vrule width 6pt height 3 pt depth -2.5pt
                  \kern -8pt \intop}\nolimits_{#1}}%
          {\mathop{\vrule width 5pt height 3 pt depth -2.6pt
                  \kern -6pt \intop}\nolimits_{#1}}%
          {\mathop{\vrule width 5pt height 3 pt depth -2.6pt
                  \kern -6pt \intop}\nolimits_{#1}}%
          {\mathop{\vrule width 5pt height 3 pt depth -2.6pt
                  \kern -6pt \intop}\nolimits_{#1}}}

\numberwithin{equation}{section}

\theoremstyle{plain}
\newtheorem{thm}{Theorem}
\numberwithin{thm}{section}
\newtheorem*{"thm"}{"Theorem"}

\newtheorem{lemma}[thm]{Lemma}

\newtheorem{cor}[thm]{Corollary}
\newtheorem{proposition}[thm]{Proposition}
\newtheorem{"proposition"}[thm]{"Proposition"}
\newtheorem{"lemma"}[thm]{"Lemma"}

\newtheorem{claim}[thm]{Claim}

\theoremstyle{definition}

\theoremstyle{remark}
\newtheorem{remark}[thm]{Remark}

\newcommand{\nref}[1]{(\hyperref[#1]{#1})}

\DeclareMathSymbol{\intop}  {\mathop}{mathx}{"B3}

\begin{document}
\begin{abstract}
We prove a  Furstenberg-type criterion for  a given orthogonal projection of a self-similar measure to be absolutely continuous, with quantified  regularity. It requires exponential mixing of the rotational part at a rate that is sufficiently fast compared with an orbit relative analogue of  its dimension. Using Ramanujan sets of irrational rotations in \(\mathrm{SO}(3)\) constructed by Lubotzky, Phillips and Sarnak (1986, 1987), we obtain explicit applications. In particular, we exhibit singular self-similar measures whose every line projection is absolutely continuous, measures of arbitrarily small Fourier dimension with smooth projections in all but a fully explicit exceptional set of directions, and a non-trivial example of a  self-similar measure that is Salem with a $C^2 _0$ density.
\end{abstract}

\maketitle

\section{Introduction}
\subsection{Background} A fundamental result in geometric measure theory is the
Marstrand-Mattila projection theorem. Informally, it says that almost
every orthogonal projection of a Borel probability   measure on $\mathbb{R}^d$ is well behaved:  its dimension is preserved in an appropriate sense and, above
the critical dimension, the projected measure is absolutely continuous
with quantitative regularity. Marstrand originally proved the theorem
in the plane; it was subsequently extended to higher dimensions by
Mattila, building in particular on Kaufman's Fourier-analytic approach.
Falconer obtained  estimates for the exceptional set of
directions, while the Sobolev formulation below was first noted by
Peres and Schlag.

Let us formulate it precisely. For $0\leq k \leq d$ let $G_{d,k}$ denote the Grassmannian of $k$-dimensional planes in $\mathbb{R}^d$. Recall that it supports a natural $\text{SO}(d)$-invariant probability measure. For $V\in G_{d,k}$ let $\pi_V :\mathbb{R}^d \rightarrow V$ denote the corresponding orthogonal projection, and for $t\geq 0$, $H^t (V)$ stands for the corresponding order-$t$ Sobolev space.
\begin{thm} \cite[Theorems 9.7 and 9.9]{mattila1999geometry} 
 \label{Theorem Marstrand}
Let $\nu \in \mathcal{P}(\mathbb{R}^d)$ be a Borel probability measure. 
\begin{enumerate}
\item Dimension preservation: For every $0\leq k\leq d$,  for almost every $V \in \text{G}_{d,k}$,
$$\dim (\pi_V \nu) = \min \lbrace \dim \nu,\, k\rbrace,\text{ where } \dim \nu :=\inf \lbrace \dim_H A:\, \nu(A)>0\rbrace.$$

\item Smoothness: For every $1\leq k< d$ if $\nu$ has finite $k$-energy  then for almost every $V \in G_{d,k}$,
$$\pi_V \nu \ll \text{Leb}_V,\text{ and furthermore } \pi_V \nu \in L^2 (\text{Leb}_V). $$
 If furthermore $\nu$ has finite $s$-energy for $k<s$ then $\pi_V \nu \in H^{\frac{s-k}{2}} (V).$ 
\end{enumerate}
\end{thm}
In Theorem \ref{Theorem Marstrand} and throughout the paper, $\mathcal{P}(X)$ denotes the space of Borel probability measures on a  metric space $X$, and the $s$-energy of $\nu \in \mathcal{P}(\mathbb{R}^d)$ is defined as 
$$I_s(\nu):=\int \int \frac{1}{|x-y|^s}\,d\nu(x)\,d\nu(y).$$

We are interested in a research direction suggested by Furstenberg: For dynamically
defined measures, can one determine which individual projections satisfy the conclusions of Theorem \ref{Theorem Marstrand}? \newline
There has been significant progress on this problem in the context of
Theorem~\ref{Theorem Marstrand} Part (1).  However, much less is known
about Furstenberg-type analogues of the smoothness statement in
Theorem~\ref{Theorem Marstrand} Part (2).  The purpose of this paper is to make
a contribution in this direction.  Roughly speaking, we show that if the
 dynamics is exponentially mixing on the orbit of a given projection,
and if the rate of mixing is sufficiently fast relative to the dimension of
\(\nu\) along that orbit, then the projection satisfies the smoothness
conclusion of Theorem~\ref{Theorem Marstrand} Part (2).

We will work with a class of stationary measures called self-similar measures:
 A  measure $\nu \in \mathcal{P}(\mathbb{R}^d)$ is called \textit{self-similar} if there exists a finite family (called an IFS - iterated function system) of strictly contracting similitudes $\Phi = \lbrace f_i(x)=r_i\cdot O_i\cdot x+t_i\rbrace_{i=1} ^n$ where $0<|r_i|<1, O_i \in \text{SO}(\mathbb{R}^d), t_i\in \mathbb{R}^d$, and a strictly positive probability vector $\mathbf{p}$ such that
$$\nu= \nu_\mathbf{p} =\sum_{i=1} ^n p_i\cdot f_i \nu,\text{ where } f_i \nu \text{ is the corresponding push-forward}.$$
\subsection{Main result: A baby case}
Our main result is given in Theorem \ref{thm:main}  below. However, to illustrate it, let us first state and discuss a special case. 
\begin{thm} \label{thm: baby case} Let
$\Phi = \lbrace f_i(x)=r\cdot O_i\cdot x+t_i\rbrace_{i=1} ^n$ be a self-similar IFS on $\mathbb{R}^3$ with a common contraction ratio $0<|r|<1$. Let $\mathbf{p}$ be a probability vector and let $\nu = \nu_\mathbf{p}$ be the corresponding self-similar measure. Let $P$ be the $L^2  (\text{SO}(3))$-operator defined by 
$P F(g)=\sum_i p_i F(O_i^{-1}g)$, for $F\in L^2(\mathrm{SO}(3))$, and write $\Lambda:=\frac{\log \left \lVert P \right \rVert_{L^2 _0 (\text{SO}(3))}}{\log |r|} $. Let $k\in \lbrace 1,2\rbrace$ and fix $V \in G_{3,k}$.
\begin{enumerate}
\item   If $\dim \nu >k$ and $\Lambda > \frac{k}{\dim \nu -k}$ then, $ \pi_V \nu \ll \text{Leb}_V.$
\item Suppose  $\dim_2 \nu>k$ and let $k\leq s < \dim_2 \nu$. Assume that
\begin{equation} \label{eq: baby condition}
\Lambda > \frac{s(\dim_2 \nu +1)}{\dim_2 \nu -s},
\end{equation}
Then, $ \pi_V \nu \ll \text{Leb}_V \text{ and furthermore } \pi_V \nu \in H^{ \frac{s-k}{2}} (V). $
\end{enumerate}

\end{thm}
For \(k=1\), if \eqref{eq: baby condition} holds with \(s=1\), then
\(\pi_V\nu\in L^2(V)\).  If it holds for some \(2<s<\dim_2\nu\), then
\(\pi_V\nu\in C_0(V)\), and hence
\(\operatorname{Int}_V\pi_V(\operatorname{spt}\nu)\neq\emptyset\)  (this is a consequence of the Sobolev embedding Theorem, see e.g. \cite[Chapter 4, Proposition 1.3]{Taylor1996partial}). 
Note that both parts require the induced operator $P$ on $L^2 _0 (\text{SO}(3))$, the orthogonal complement to the subspace of constant functions,  has spectral gap; that is, $\lVert P \rVert_{L^2 _0 (\text{SO}(3))}<1$ (see \eqref{def: L two zero}). The notation $\dim_2 \nu$ stands for correlation dimension or $L^2$-dimension, and is defined as 
\begin{equation} \label{eq def of l2 dim}
\dim_2 \nu:= \sup \lbrace \alpha: I_\alpha(\nu)<\infty \rbrace.
\end{equation}
Our most general result, Theorem~\ref{thm:main}, works in arbitrary
dimensions and requires neither a uniform contraction ratio nor global
spectral gap.  For the deduction of Theorem \ref{thm: baby case} from Theorem \ref{thm:main}, see the discussion following  the statement of
Theorem~\ref{thm:main}.

To the best of our knowledge, this is the first general and fully explicit
criterion ensuring smoothness of a prescribed projection of a self-similar
measure. Nonetheless, there are related works of Lindenstrauss and Varj\'u
\cite{Lind2016var}, and Kittle and Kogler \cite{kittle2024absolute}, that
concern smoothness of the ambient measure $\nu$ itself; we compare our
results and methods with theirs in Section~\ref{SEction prior}.
\subsection{Discussion of \eqref{eq: baby condition} and explicit examples}
Fix an IFS $\Phi = \lbrace f_i(x)=r\cdot O_i\cdot x+t_i\rbrace_{i=1} ^n$ and a self-similar  measure $\nu$ on $\mathbb{R}^3$; note, however, that much of our discussion and further results apply in dimension $\geq 3$. Via \eqref{eq: baby condition},  Theorem \ref{thm: baby case} relies on two delicate quantities: the spectral gap of the  operator $P$ on $L^2(\text{SO}(3))$, and $\dim_2 \nu$. Let us first discuss the spectral gap assumption. We are not aware of any known example when  $\lbrace O_i \rbrace_{i=1} ^n$ generates a dense subgroup of $\text{SO}(3)$ but $P$ does not have spectral gap.  Bourgain and Gamburd \cite[Theorem 1]{Bourgain2008Gamburd} proved that if $\lbrace O_i \rbrace_{i=1} ^n$ generates a dense subgroup and has only algebraic entries, then $\lVert P \rVert_{L^2 _0 (\text{SO}(3))}<1$ does hold true. This has been extended to $\text{SU}(d),d\geq 2$ in a subsequent paper \cite{Bourgain2012Gamburd}, and has been further extended by Benoist and de Saxc\'e \cite{Benoist2016Saxce} to general simple compact Lie groups, and in particular for $\text{SO}(d),d\geq 3$. All cited papers retain the algebraicity assumption as in \cite{Bourgain2008Gamburd}.

Precise estimates are given in the works of Lubotzky, Phillips, and Sarnak \cite{LPS1986first, LPS1987second}. These authors consider Hecke Operators on $\text{SO}(3)$, which means that $\mathbf{p}$ is taken to be uniform, and $\lbrace O_i \rbrace_{i=1} ^{n}$ is symmetric (so $n$ is  even). Denoting by $P^G F(g)=\frac1n\sum_{i=1}^n F(O_i^{-1}g)$ for $F\in L^2(\mathrm{SO}(3)),$  they  note that $\lVert P^G \rVert_{L^2 _0 (\text{SO}(3))}\geq \frac{2 \sqrt{n-1}}{n}$  is  the best possible spectral bound in this setup. 
Then, they find explicit examples of such $\lbrace O_i \rbrace_{i=1} ^{n}$ for which this  bound is \textit{achieved}: For every prime $p \equiv 1 \mod 4$ they construct a \textit{Ramanujan set} $\mathcal{R}_p \subseteq \text{SO}(3)$, which is a symmetric set  of $p+1$ elements, for which we have precisely
\begin{equation} \label{eq: ramanujan}
\lVert P^G \rVert_{L^2 _0 (\text{SO}(3))} =  \frac{2 \sqrt{p}}{p+1}.
\end{equation}
Higher-dimensional analogues of this construction were obtained by Clozel
\cite{Clozel2002automorphic} and Oh
\cite{Oh2005Ruziewicz}. However, unlike the  \(\mathrm{SO}(3)\) case,
they involve auxiliary arithmetic constants which are not given in an explicit
numerical form. While we believe such a result may be attainable, we do not know of a fully explicit
Ramanujan-type estimate for \(\mathrm{SO}(4)\).

Next, we discuss $\dim_2 \nu$, recall \eqref{eq def of l2 dim}.  There is a natural symbolic candidate for $\dim_2 \nu$ that satisfies, for measures with common contraction $0<r<1$ as in Theorem \ref{thm: baby case},
$$D_2 ^{\text{sym}}:=\min \lbrace3,\, \frac{\log \sum_i p_i ^2}{\log r} \rbrace. $$
Cawley and Mauldin \cite{Cawley1992Mauldin} proved $\dim_2 \nu =  D_2 ^{\text{sym}}$ holds if the IFS has the SSC - the  strong separation condition ($f_i\neq f_j$ implies $f_i(\text{supp}(\nu)) \cap f_j(\text{supp}(\nu)) =\emptyset$). Note, however, that elementary considerations show that in general \eqref{eq: baby condition} cannot hold assuming the IFS has the SSC. Falconer \cite{Falconer1999dim} proved  $\dim_2 \nu =  D_2 ^{\text{sym}}$ if we fix the linear part $\lbrace r\cdot O_i \rbrace_{i=1} ^n$,  for Lebesgue-a.e. choice of the translates $t_i$, assuming $r<\frac{1}{2}$. Consequently, by combining Falconer's theorem with sufficiently large
Ramanujan sets $\mathcal R_p$, one can obtain parametric families of self-similar
measures on $\mathbb R^3$ for which, for Lebesgue-a.e. choice of the
translations, every line projection belongs to $L^2$; more generally,
one can obtain any prescribed Sobolev regularity $H^t$ with $t<1$. Related recent Falconer-type projection
results for typical self-affine systems  were obtained by Morris-Sert
\cite{MorrisSert2025} and Feng-Xie \cite{FengXie2025}.
Our emphasis below, however, will be on fully explicit (i.e. non-parametric) applications of the
criterion.

A substantially weaker separation condition is the exponential separation
condition (ESC); It asks that  for some $c>0$ and infinitely many $m\in \mathbb{N}$
$$\min \lbrace d( f_I, \, f_J):\,I,J\in \lbrace 1,\dots,n \rbrace^m,\,I\neq J\rbrace\geq c^m,  $$
 where  $d(rO +t,\, r'O'+t')=|\log \frac{r}{r'}|+\lVert O-O'\rVert_{\text{op}} +\lVert t- t' \rVert$. It  always holds if all entries of $\Phi$ are algebraic, in the absence of exact overlaps (i.e. if the semi-group generated by $\Phi$ is freely generated);  see the works of Hochman \cite{hochman2015Rd, hochman2014self}. Now, in \cite[Theorem 1.4]{hochman2015Rd}, Hochman proved that if the ESC holds, the rotations of the IFS do not fix any subspace $V<\mathbb{R}^d$, and the semigroup generated by $\Phi$ is free, then the \textit{Hausdorff dimension}  $\dim \nu$   equals its symbolic counterpart $D ^{\text{sym}}:=\min \lbrace3,\, \frac{\sum_i p_i\log  p_i }{\log r} \rbrace$; this result will be used together with Theorem \ref{thm: baby case} Part (1) to derive Corollary \ref{cor:explicit-singular-smooth-projections} below. Recently, Corso and Shmerkin \cite{corso2024dynamical}, extending the work of Shmerkin \cite{shmerkin2016furstenberg},  proved that  if  all $O_i=O_j$ and $r_i=r_j$,  then the ESC further implies that $\dim_2 \nu = D_2 ^{\text{sym}}$.  It is not currently known, however, if this remains true for more general $\Phi$ as in Theorem \ref{thm: baby case} or Theorem \ref{thm:main}. 
 
Nonetheless, using  Ramanujan sets of rotations, one can  give  examples of self-similar measures for which our results apply. First, we exhibit  a non-trivial fully explicit self-similar measure on $\mathbb{R}^3$ that is singular, yet \emph{all} of its line projections are absolutely continuous:
\begin{cor}
\label{cor:explicit-singular-smooth-projections}
Let 
$$\Phi:=
\left\{
f_{q,b}(x)=
18^{-2/5}O_{q,b}x+q+2b\mathbf{1}
\right\}_{
 q\in\{0,1\}^{3}:1\leq \lVert q \rVert_1 \leq 2, \, b\in\{-1,0,1\} }$$ where
$\mathbf{1}:=(1,1,1),$ and
$(q,b)\mapsto O_{q,b},$
is 
 some bijection onto the Ramanujan set
\(\mathcal{R}_{17}\subset \operatorname{SO}(3)\).
Let \(\nu \in \mathcal{P}(\mathbb{R}^3)\) be the uniform self-similar measure with respect to $\Phi$. 

Then
$\dim\nu=\frac{5}{2}$;
In particular,
$\nu\perp \operatorname{Leb}_{\mathbb{R}^{3}}.$
Nevertheless, for every \(V\in G_{3,1}\),
$\pi_{V}\nu\ll \operatorname{Leb}_{V}.$
\end{cor}
All the parameters of the IFS above are algebraic, and  we will verify that it has no exact overlaps. Since $\mathcal{R}_{17}$ acts irreducibly on \(\mathbb{R}^{3}\), Hochman's theorem \cite[Theorem~1.4]{hochman2015Rd}  yields \( \dim\nu=\frac{5}{2}. \) The absolute continuity of every line projection subsequently follows from Part~\textup{(1)} of Theorem~\ref{thm: baby case}. In fact, using similar ideas, one can construct a self-similar measure of this form with analogous properties  with any prescribed dimension $D\in (2,3)\cap \mathbb{Q}$, varying the Ramanujan set $\mathcal{R}_{p}$ with $D$, and adjusting the common contraction ratio. See Section \ref{subsection: cor explicit singular smooth} for the proof, and also for  explicit formulae and further discussion of the matrices in $\mathcal{R}_{17}$.

Next, we give a fully explicit family of singular self-similar measures on
\(\mathbb R^4\) whose \emph{every} line projection has an \(L^2\) density. Let \(p\equiv 1\mod 4\) be  prime, and let
 $\Omega_p$ be the standard set of \(p+1\) lifts of the Ramanujan set \(\mathcal R_p\subset SO(3)\) under the adjoint map \(SU(2)\to SO(3)\). Identifying \(\mathbb{R}^{4}\) with the quaternion algebra \(\mathbb H\), for \(\omega\in\Omega_p\) let \( A_\omega x:=\omega x, \, x\in\mathbb H \); so, \(A_\omega\in SO(4)\). The explicit arithmetic description \eqref{eq: omega p} of \(\Omega_p\) and  the matrix representation \eqref{eq: A omega} of \(A_\omega\) are given in Section~\ref{subsec:proof-singular-ssc-all-line-projections}. Enumerate \(\Omega_p=\{\omega_1,\ldots,\omega_{p+1}\}\), and let
\(M=\lceil (p+1)^{1/4}\rceil\). We choose the translations \(t_j\) to be
the "first" \(p+1\) points in the \(M^4\) integer grid. Namely, 
$$\text{for } 1\leq j\leq p+1,\, \, j-1=n_1(j)+M n_2(j)+M^2 n_3(j)+M^3 n_4(j),
\, 0\leq n_i(j)\leq M-1,$$ and
put
$t_j=(n_1(j),n_2(j),n_3(j),n_4(j)).$ For $D>1$ define
\[
\Phi_{p,D}
=
\left\{
x\mapsto (p+1)^{-1/D}\cdot A_{\omega_j}x+t_j
\right\}_{j=1}^{p+1}.
\]
\begin{cor}
\label{cor:singular-ssc-all-line-projections}
Let  $\frac{3+\sqrt{21}}{2}<D<4$  and choose a prime
\(p\equiv 1 \text{ mod } 4\) such that  
$$p>
\max\left\{
5^{\frac{4D}{4-D}},\,
2^{\frac{4D(D-1)}{D^2-3D-3}}
\right\}.$$ 
Then the uniform self-similar measure $\nu \in \mathcal{P}(\mathbb{R}^4)$ with respect to $\Phi_{p,D}$ satisfies:
\begin{enumerate}

\item  $\dim_H\nu=\dim_2\nu=D<4,$ in particular, $\nu\perp\operatorname{Leb}_{\mathbb R^4}.$

\item For every $V\in G_{4,1}$,
$\pi_V\nu\ll\operatorname{Leb}_V,$ and
$\pi_V\nu\in L^2(V).$
\end{enumerate}
\end{cor}
Since there are infinitely many primes \(p\equiv1\pmod4\), by Dirichlet's
theorem on primes in arithmetic progressions, such a prime can always be chosen.
We remark that the first lower bound on $p$ ensures that $\Phi_{p,D}$ has the SSC, and the second ensures that we can apply Theorem \ref{thm:main}. 
In fact,  given $0\leq t<1/30$, we can pick \(D<4\) sufficiently close to \(4\) and then a sufficiently large
prime \(p\equiv1\mod4\), so that the corresponding self-similar measure $\nu\in\mathcal P(\mathbb R^4)$
has
$\pi_V\nu\in H^t(V)
\text{ for every }V\in G_{4,1}.$

It is natural to ask whether examples such as Corollary \ref{cor:singular-ssc-all-line-projections} are actually a consequence of the measure having large Fourier dimension.  Recall that the Fourier dimension of \(\nu\in\mathcal P(\mathbb R^d)\) is
defined by
\begin{equation}\label{eq: fourier dim}
\dim_F \nu
=
\sup\left\{
0\leq s\leq d:
|\mathcal{F}_\xi (\nu)|\lesssim |\xi|^{-s/2}
\right\},
\end{equation}
where we use the notation 
$$ \mathcal{F}_\xi ( \nu):= \widehat{ \nu}( \xi ).$$
This is not the case: our next application shows that our results still apply even when $\dim_F \nu$ is arbitrarily small.

Consider, for \(0<r<1\), the self-similar IFS on \(\mathbb R^4\)
defined by
\[
\Phi_r
=
\left\{
f_{O,b,c}(x,y)=\bigl(rOx+be_1,\, ry+c\bigr)
\right\}_{O\in \mathcal R_5,\ b,c\in\{0,1\}},\, x\in \mathbb{R}^3,\, y\in \mathbb{R},
\]
where \(\mathcal R_5\subset \mathrm{SO}(3)\) is the Ramanujan set associated to
\(p=5\).  

\begin{cor}
\label{cor:main}
For every \(\eta>0\) and every \(\tau>0\), there exist
\(0<r<1\) and a self-similar measure $\nu \in \mathcal{P}(\mathbb{R}^4)$ with respect to $\Phi_r$, such that
$\nu \perp \operatorname{Leb}_{\mathbb R^4}$ and
$\dim_F\nu <\eta.$

However, for every \(1\leq k\leq3\) and every \(W\in G_{4,k}\),
\[
\pi_W\nu\ll\operatorname{Leb}_W
\quad\Longleftrightarrow\quad
e_4\notin W.
\]
In fact, if \(e_4\notin W\), then
$\pi_W\nu \in H^t(W)
\text{ for every }t<\frac{4-k}{2}-\tau.$
\end{cor}

The measures $\nu$ produced in the Corollary are  non-product couplings of two
opposing components:  Their $\mathbb R^3$-marginals are 
self-similar measures whose Ramanujan rotations yield strong decay of
spherical averages of Fourier modes, whereas their fourth-coordinate marginals are highly
biased Bernoulli convolutions.  The latter marginals are singular and 
have arbitrarily small Fourier dimension. This forces the corresponding
properties of $\nu$, and accounts for the exceptional direction
$e_4$.  On the other hand, projections onto subspaces not containing
$e_4$ necessarily see a non-trivial $\mathbb{R}^3$ component.  Their Sobolev regularity is then obtained by applying the orbit-relative \(L^2\)-dimension criterion of Theorem~\ref{thm:main}, together with its comparison with the corresponding Sobolev dimension.

Similarly to the discussion of Fourier dimension, one may ask whether our conclusions are simply a consequence of the ambient measure having large \(L^2\)-dimension. Our next example shows that this is not the case: the \(L^2\)-dimension of the measure may be arbitrarily close to \(1\), while all of its line projections are absolutely continuous with an $L^2$ density. For $0<r<1$ 
consider the IFS on $\mathbb{R}^3$
$$\widetilde \Phi_{r} = \lbrace
f_{O,0}(x)=rOx, 
f_{O,1}(x)=rOx+e_1
\rbrace_{O\in\mathcal R_5}.$$

\begin{cor}
\label{cor:rough-all-line-projections}
For every \(0<\delta<1\) and every \(0\leq t<\delta/2\), there exist
\(0<r<1\) and a self-similar measure
\(\nu\in\mathcal P(\mathbb R^3)\) with respect to $\widetilde \Phi_r$ that satisfies:
\begin{enumerate}
\item
$\dim_2\nu <1+\delta. $ 
In particular, \(\nu \notin L^2(\mathbb R^3)\).

\item For every \(V\in G_{3,1}\),
$\pi_V\nu \ll\operatorname{Leb}_V$ and $\pi_V\nu \in H^t(V).$ 
\end{enumerate}
\end{cor}

Finally, although our main aim is to study the smoothness of projections, our analysis also yields  explicit examples in which the self-similar measure itself can be shown to have a smooth density:
\begin{cor}
\label{cor:explicit-ambient-ac}
Let $\nu$ be the uniform self-similar measure with respect to the IFS \newline
$\lbrace f_{O,0}(x)= 0.9885 O x,
f_{O,1}(x)= 0.9885 O x+e_1 \rbrace_{O\in \mathcal R_5}.$
Then 
\[ \nu \in  H^t(\mathbb R^3) \text{ for every } 0\leq t<3.9,\,  
\nu\ll \operatorname{Leb}_{\mathbb R^3},
\,
\nu
\in C_0 ^2(\mathbb R^3),\, \operatorname{Int}_{\mathbb R^3}(\operatorname{spt}\nu)\neq\emptyset.
\]
Furthermore, $\nu$ is a Salem measure.
\end{cor}
Recall that $\nu$ being Salem means that $\dim \nu = \dim_F \nu$. 
See Section \ref{Section proof of coro}  for the proof of these Corollaries. Closely related results to Corollary \ref{cor:explicit-ambient-ac} have been obtained by  Lindenstrauss and Varj\'u \cite{Lind2016var} and Kittle and Kogler
\cite{kittle2024absolute}. We discuss this in detail in the following Section.

Finally, we remark that  Solomyak
\cite{solomyak2025fourier} recently proved polynomial Fourier decay for typical, in an appropriate sense, 
 homogeneous self-similar measures on $\mathbb{R}^d,d\geq 3$; combined with the results of Corso and Shmerkin
\cite{corso2024dynamical}, this yields absolute continuity in the
supercritical regime; see also \cite{solomyak2023absolute}.   De-Jun Feng and Zhou Feng
\cite{feng2024typical} obtained sufficient conditions for typical
self-affine sets in arbitrary dimension to have non-empty interior,
assuming $r<\frac{1}{2}$ in the setting of Corollary \ref{cor:explicit-ambient-ac}.  

\subsection{Prior results} \label{SEction prior} There is a rich literature on Furstenberg-type classification of dimension
dropping projections, that is, projections failing the dimension conclusion of
Theorem~\ref{Theorem Marstrand} part (1).  Our discussion is selective and
oriented toward Theorem~\ref{thm: baby case}; broader surveys  can be found in Shmerkin~\cite{shmerkin2015projections}
, Falconer, Fraser, and Jin~\cite{falconer2015sixty},  Solomyak \cite{solomyak2023notes}, and Falconer \cite{falconer2026seventy}.

A particularly relevant parametric result is due to Peres and
Schlag~\cite{Peres200Schlag}. They developed a  framework for
parametrized families of so-called generalized projections satisfying certain transversality
assumptions. They obtain Sobolev regularity for typical parameters and give  
bounds on the dimension for the exceptional set.  In the self-similar setting, Shmerkin and Solomyak \cite{Shmerkin2016Solomyak} proved that, for a planar homogeneous self-similar measure satisfying the SSC, in the supercritical regime all line projections outside a zero-dimensional exceptional set of directions are absolutely continuous, with densities in \(L^q\) for some \(q>1\). These results  establish regularity outside a small exceptional set  parameters. In contrast, we impose the stronger  hypothesis of spectral gap, but in return obtain regularity for a prescribed projection of a fixed self-similar measure. More recently, Banaji and Yu~\cite{Banaji2025yu} used modern methods in the
Fourier analysis of self-similar measures to prove that radial projections of
certain Cartesian products of self-similar sets have positive Lebesgue
measure, and non-empty interior. Their results do not concern orthogonal projections, which are the
focus of the present paper.

A very general Furstenberg type projection Theorem  for self similar measures was given by Hochman and Shmerkin in \cite{hochman2009local}: Let $\nu \in \mathcal{P}(\mathbb{R}^d)$ be a self-similar measure with the SSC, and assume that for some (equivalently, any) $V \in G_{d,k}$, 
\begin{equation} \label{eq HS condition}
\big\{  O_{i_1}\cdots O_{i_m}V:\, i_1,....,i_m \in \lbrace1,...,n\rbrace,\, m\in \mathbb{N} \big\} \text{ is dense in } G_{d,k},
\end{equation}
where $\lbrace O_i \rbrace_{i=1} ^n$ are the rotations of the underlying IFS. 
Then, for any $V \in G_{d,k}$ we have $\dim \pi_V \nu = \min \lbrace k,\, \dim \nu \rbrace.$ That is, Theorem \ref{Theorem Marstrand} Part (1) holds for \emph{all} projections. This extended earlier work by Peres and Shmerkin \cite{Peres2009Shmerkin}. Falconer and Jin \cite{Falconer2014Jin} subsequently removed the
separation condition, using a compact-group extension and symbolic
arguments, building on the local entropy averages framework
developed in~\cite{hochman2009local}..   Algom and Shmerkin \cite{Algom2025shermkin} recently replaced the
density condition \eqref{eq HS condition} by a  weaker
criterion on the action of the orthogonal parts on the projected
subspace. Within their framework, Corso and Shmerkin \cite{corso2024dynamical} further obtained results regarding the $L^q$ dimensions of these projections.  Finally, in a slightly different direction, a recent breakthrough of Wu \cite{wu2025projection} shows that under mild conditions on a planar measure, the exceptional set to a packing dimension version of Theorem \ref{Theorem Marstrand} Part (1) is at most countable. For more recent related results regarding Furtenberg type projection Theorems  for  dynamically defined measures, under various assumptions, see e.g. \cite{Bruce2019jin, bruce2022furstenberg, Aleksi2024reso, jordan2019dimension, barany2023scaling, BaranyKaenmakiKolossvary2026}. 

Most of the results cited above follow to some extent from an entropy based approach; either the local entropy averages method of Hochman Shmerkin \cite{hochman2009local}, or Hochman's inverse Theorems for entropy \cite{hochman2014self}. It is not known if these ideas can yield information about the smoothness of the projected measure.  In fact, even if \eqref{eq HS condition} holds, smoothness may fail in general: Indeed, building on ideas of Nazarov, Peres, and Shmerkin \cite{nazarov2012peres}, Rapaport \cite{Rapaort2017exp} exhibited a  self-similar measure $\nu \in \mathcal{P}(\mathbb{R}^2)$ with dimension larger than $1$, satisfying the Hochman Shmerkin density condition \eqref{eq HS condition}, that admits a dense $G_\delta$ set  $V\in G_{2,1}$ such that $\pi_V \nu \perp \text{Leb}$. On the positive side, using the results of Shmerkin and Solomyak
\cite{Shmerkin2016Solomyak}, Rapaport \cite{Rapaort2020proj} proved that,
outside a zero-dimensional exceptional set of parameters, every line
projection of a planar homogeneous self-similar measure with strong
separation, irrational rotation, and dimension larger than $1$ belongs to
$L^q$ for some $q>1$. Again, these results are both parametric and planar; by contrast,
our criterion applies to a given projection of a fixed self-similar
measure, and apply in dimensions $d\geq3$.

The work of Rapaport \cite{Rapaort2017exp} shows that for Furstenberg-type analogues of the
smoothness part of Theorem \ref{Theorem Marstrand}, density of the rotations
 is not sufficient. This motivates the use of a quantitative mixing assumption, namely spectral
gap. Such an $L^2$ spectral gap is unavailable on $\operatorname{SO}(2)$,
so our mechanism is intrinsically higher-dimensional. In this direction, the closest
existing results do not concern projections directly, but rather absolute
continuity of the ambient self-similar measure itself.

Lindenstrauss and Varj\'u \cite{Lind2016var} proved that, in dimensions
\(d\geq 3\), if  the orthogonal parts have  spectral
gap then the associated self-similar measure is absolutely
continuous, with a smooth density, provided the contraction ratios are
sufficiently close to one. Boutonnet and Ioana \cite{boutonnet2020local} subsequently gave a simpler approach to the main twisted-operator estimate underlying the argument of Lindenstrauss and Varj\'u. Recently, Kittle and Kogler
\cite{kittle2024absolute} extended  this circle of
ideas. Among other things, they give absolute-continuity criteria in arbitrary
dimensions, construct explicit absolutely continuous inhomogeneous examples also in
dimensions one and two, and in dimensions \(d\geq 3\) improve the
Lindenstrauss Varj\'u criterion for absolute continuity (without always assuming spectral gap; their methods, however, do not give \(L^2\) or Sobolev regularity of the density, as far as we are aware). 

It is thus interesting to compare our results to 
\cite{Lind2016var, boutonnet2020local, kittle2024absolute}. First,  the aims are different. The results of
\cite{Lind2016var,kittle2024absolute} prove absolute continuity of the ambient
self-similar measure \(\nu\) itself. By contrast, we study its
projections. In particular, several of our main examples are singular measures whose projections are nevertheless absolutely continuous in all, or in a prescribed class of, directions. Thus, it does not seem that Corollaries \ref{cor:explicit-singular-smooth-projections},  \ref{cor:singular-ssc-all-line-projections}, \ref{cor:main}, \ref{cor:rough-all-line-projections},   can be recovered from \cite{Lind2016var,kittle2024absolute}.

There is also an important quantitative distinction. The criteria in
\cite{Lind2016var,kittle2024absolute}, as well as the alternative approach to \cite{Lind2016var} found in \cite{boutonnet2020local}, involve constants whose numerical values
are not made explicit.  In contrast, the conditions in Theorems \ref{thm: baby case} and \ref{thm:main} are completely
explicit. We are able to achieve this by working directly with the Fourier transform of $\pi_V \nu$, and avoiding any additive combinatorial machinery, which often leads to such constants that could be difficult to estimate in practice, and could be quite large in general. Thus, although the methods of \cite{Lind2016var,boutonnet2020local} can be made effective, the quantitative analysis we carried out did not yield estimates strong enough to recover Corollary~\ref{cor:explicit-ambient-ac}. We do not rule out, however, that a sharper optimized analysis of their methods could do so.

\subsection{Main result} We first state an \(L^2\)-version of our main result. It  already implies many of the applications above, while avoiding the additional notation needed for the general \(L^q\)-statement that follows it. 

First, we require a notion of $L^2$ dimension of a measure relative to certain subspaces.  Let $G\leq SO(d)$ be a closed subgroup.
For \(V\in G_{d,k}\), write $\mathcal O_V:=G\cdot V\subseteq G_{d,k}$ 
for its compact \(G\)-orbit, and let \(m_V\) denote the normalized
\(G\)-invariant probability measure on \(\mathcal O_V\).  For \(W\in G_{d,k}\),
let \(S(W)=\{\theta\in W:|\theta|=1\}\) be its unit sphere, and let \(\sigma_W\) be the normalized
surface measure on \(S(W)\). For a measure $\nu \in \mathcal{P}(\mathbb{R}^d)$ and $R>0$,  we define the spherical average relative to the $G$-orbit
of \(V\) by
\[
\mathcal{S}_{\nu,V}(R)
:=
\int_{\mathcal O_V}\int_{S(W)}
|\mathcal{F}_{R\theta}(\nu)|^2\,d\sigma_W(\theta)\,dm_V(W).
\]
We define the \(L^2\)-dimension of \(\nu\) relative to the $G$-orbit of \(V\) by
\[
\dim^{V}_2\nu
:=
\sup\left\{
0\leq t\leq d:
\int_1^\infty R^{t-1} \mathcal S_{\nu,V}(R)\,dR<\infty
\right\}.
\]
If \(G=\text{SO}(d)\) then $\dim_2 \nu = \dim^{V}_2\nu$ for all $V$; see Lemma \ref{lem:relative-dimension-full-SO}. For more discussion of this notion, see Section \ref{Section Peres schlag}. Notice that $\dim^{V}_2\nu$ also relies on the choice of $G$; however, for us, this choice will always be the closed group generated by the rotations of the IFS, so we suppress this in our notation.

Let \(H^\sigma(\mathcal O_V)\subseteq L^2(\mathcal O_V)\) denote the Sobolev
space of order \(\sigma\) on \(\mathcal O_V\); see
Section~\ref{subsec:stopped-cylinder-mixing} for the precise definition. Recall that the orthogonal complement to the subspace of constant functions in $L^2( \mathcal{O}_V, m_V)$ is denoted by $L^2 _0 (\mathcal{O}_V, m_V)$. 
\begin{thm} \label{thm: pre main} Let
$\Phi=\{f_i(x)=r_iO_i x+t_i\}_{i=1}^n$
be a self-similar IFS on \(\mathbb R^d\), \(d\geq3\), and let
\(\nu=\nu_{\mathbf p}\) be the corresponding self-similar measure.  Let
$\mathcal A:=\{|r_i|:1\leq i\leq n\}$
be the set of contraction moduli.  Fix \(1\leq k<d\) and
\(V\in G_{d,k}\).  Let
$G=\overline{\langle O_1,\ldots,O_n\rangle}\subseteq SO(d)$
and consider the orbit
$\mathcal O_V=G\cdot V.$

For \(a\in\mathcal A\), set
$I_a:=\{i:|r_i|=a\},
\,
\widetilde p_a:=\sum_{i\in I_a}p_i,$
and define the Markov operator on \(L^2(\mathcal O_V,m_V)\) by
\[
P_{a,V}F(W)
=
\sum_{i\in I_a}\frac{p_i}{\widetilde p_a}F(O_i^{-1}W).
\]

If
\(\dim_2^V\nu>k\), and  there exist
$\sigma>\frac12\dim\mathcal O_V$ and
$k\leq s<\dim_2^V\nu,$
such that
\[
\frac{-\log\left(
\sum_{a\in\mathcal A}
\widetilde p_a
\left\|P_{a,V}\right\|_{L^2 _0 (\mathcal{O}_V, m_V)}
\right)}{\max_i(-\log |r_i|)}
>
\frac{s(\dim_2^V\nu+\sigma)}{\dim_2^V\nu-s}.
\]
Then,
$\pi_V\nu\ll \operatorname{Leb}_V
\quad\text{and furthermore}\quad
\pi_V\nu\in H^{\frac{s-k}{2}}(V).$
\end{thm}
A few remarks are in order. First,
\(\dim\mathcal O_V\) may depend non-trivially on \(V\). For example,
consider the IFS in Corollary~\ref{cor:main}, acting on
\(\mathbb R^4\). Its orthogonal parts generate
\[
G
=
\left\{
\begin{pmatrix}
O&0\\
0&1
\end{pmatrix}
:
O\in\operatorname{SO}(3)
\right\}.
\]
For \(k=2\) and
\(V_1=\operatorname{span}\{e_1,e_2\}\), we have
$\mathcal O_{V_1}=G\cdot V_1\simeq G_{3,2},$
and hence
\(\dim\mathcal O_{V_1}=2\). On the other hand, for
\(V_2=\operatorname{span}\{e_1,e_2+e_4\}\), the stabilizer of \(V_2\)
in \(G\) is finite, and therefore
$\dim\mathcal O_{V_2}
=
\dim\operatorname{SO}(3)
=
3.$
 The relevant spectral gap may likewise depend on \(V\). The operators
\(P_{a,V}\) act boundedly on both
\(L^2_0(\mathcal O_V,m_V)\) and \(H^\sigma(\mathcal O_V)\), since they
are finite averages of pullbacks by isometries of \(\mathcal O_V\);
they also preserve \(H^\sigma_0(\mathcal O_V)\). Moreover, each
\(P_{a,V}\) commutes with the Laplace--Beltrami operator on
\(\mathcal O_V\). Consequently, for every \(\sigma\geq0\),
$\left\|P_{a,V}\right\|_
{H^\sigma_0(\mathcal O_V)}
=
\left\|P_{a,V}\right\|_
{L^2_0(\mathcal O_V,m_V)}.$
The additional requirement
$\sigma>\frac12\dim\mathcal O_V$
is needed in the proof in order to apply Sobolev embedding on
\(\mathcal O_V\), see \eqref{eq:sobolev-embedding-orbit}.

The orbit-level spectral gap can often be verified using spectral gap
on the ambient group; see
Lemma~\ref{lem:ambient-gap-controls-orbit-sobolev}. In particular, in
the setting of Theorem~\ref{thm: baby case}, the assumption
$\|P\|_{L^2_0(\operatorname{SO}(3))}<1$
forces the closed group generated by the rotations to be
\(\operatorname{SO}(3)\). Hence, for every
\(k\in\{1,2\}\) and every \(V\in G_{3,k}\),
$\mathcal O_V
=
\operatorname{SO}(3)\cdot V
=
G_{3,k}
\simeq\mathbb{RP}^2$
and therefore
$\dim\mathcal O_V=2.$
Since there is only one contraction modulus,
Lemma~\ref{lem:ambient-gap-controls-orbit-sobolev} gives
$\|P_V\|_{L^2_0(\mathcal O_V,m_V)}
\leq
\|P\|_{L^2_0(\operatorname{SO}(3))}$.  Choosing \(\sigma>1\) arbitrarily close to
\(1\), the Sobolev part of Theorem~\ref{thm: baby case} follows
directly from Theorem~\ref{thm: pre main}.

Finally, it is natural to ask whether one may replace
$\max_i(-\log|r_i|)$
in Theorems~\ref{thm: pre main} and~\ref{thm:main} by the Lyapunov
exponent
$\chi:=-\sum_i p_i\log|r_i|.$
Such a refinement is indeed possible; see Remark~\ref{rmk: lyp} below.

To state our most general  result, we require a more refined  notion of \(L^q\)-dimension of a
measure adapted to the $G$-orbit of a prescribed subspace. Fix \(1<q\leq2\).  Choose a smooth radial Schwartz function \(\psi\) on \(\mathbb R^k\) whose Fourier transform is supported in
$\lbrace\xi\in\mathbb R^k:\frac12\leq|\xi|\leq2\rbrace,$
and for which there exists a smooth low-frequency cutoff $(\varphi_0$) such that \eqref{eq: LP partition} holds.  For \(W\in G_{d,k}\) and \(m\geq1\), let
\(\psi_m^W\) denote the corresponding mean-zero Littlewood-Paley kernel on
\(W\), defined by
$\widehat{\psi_m^W}(\xi)=\widehat\psi(2^{-m}\xi),
\, \xi\in W.$ Since \(\psi\) is radial, this definition is independent of the choice of an orthogonal identification of \(W\) with \(\mathbb R^k\).

For \(S\geq k\) and $\nu \in \mathcal{P}(\mathbb{R}^d)$, define
\[
\mathcal E^V_{q,S}(\nu)
:=
\sum_{m\geq1}
2^{m(q-1)(S-k)}
\int_{\mathcal O_V}
\left\|(\pi_W\nu)*\psi_m^W\right\|_{L^q(W)}^q
\,dm_V(W).
\]
We define the relative \(L^q\)-dimension of \(\nu\) along the orbit of \(V\) by
\[
\dim_q^V\nu
:=
\sup\left\{
S\in[k,d]:
\mathcal E^V_{q,S'}(\nu)<\infty
\text{ for every } k<S'<S
\right\},
\]
with the convention that \(\dim_q^V\nu=k\) if the above set is empty (this case will not be of interest to us). This quantity may be viewed as a Besov version of \(L^q\)-dimension adapted to \(k\)-dimensional projections. Different choices of admissible Littlewood-Paley kernels yield equivalent quantities, and hence the value of \(\dim_q^V\nu\) is independent of this choice. A detailed discussion is given in Section~\ref{Section Peres schlag}.

In particular, Lemma~\ref{lem:besov-spherical-l2-equivalence} shows that there is no conflict with our earlier notation: when \(q=2\), the preceding Besov definition agrees with the spherical-average definition of \(\dim_2^V\nu\). As in the case \(q=2\), the quantity \(\dim_q^V\nu\) also depends on the acting group \(G\); but \(G\) is always the closed group generated by the rotations of the IFS, so we still suppress this dependence from the notation.

We will use the standard Besov spaces \(B^\beta_{q,q}(V)\), defined via the
same Littlewood-Paley decomposition as above; see 
Section~\ref{subsec:besov-orbit-estimates}.  We can finally state the main technical Theorem of this paper:
\begin{thm}
\label{thm:main}
Keep the notations and assumptions of Theorem~\ref{thm: pre main}. Let \(1<q\leq2\), and write
\(q'=q/(q-1)\). Suppose that
$k<b<S<\dim_q^V\nu.$
Assume 
\[
1<q<2
\,\text{ and }\quad
\frac12\dim\mathcal O_V<\sigma<q, \text{ or } 
q=2
\quad\text{and}\quad
\frac12\dim\mathcal O_V<\sigma.
\]
If
\[
\frac{-\log\left(
\sum_{a\in\mathcal A}
\widetilde p_a
\|P_{a,V}\|_{L^2_0(\mathcal O_V,\,m_V)}
\right)}{\max_i(-\log |r_i|)}
>
\left(
b(q-1)+
\frac{b(\sigma+b(q-1))}{S-b}
\right),
\]
then \(\pi_V\nu\ll\operatorname{Leb}_V\), and 
$\pi_V \nu \in B^{(b-k)/q'}_{q,q}(V);$
In particular,  \(\pi_V \nu \in  L^q(V)\).
\end{thm}
Notice  that Theorem~\ref{thm:main} recovers
Theorem~\ref{thm: pre main} at the endpoint \(q=2\). Indeed, taking
\(b=s\), we have
$B^{(s-k)/2}_{2,2}(V)=H^{(s-k)/2}(V),$
and, upon letting \(S\uparrow\dim_2^V\nu\), the numerical condition in
Theorem~\ref{thm:main} converges to that of
Theorem~\ref{thm: pre main}. Since the latter inequality is strict,
Theorem~\ref{thm:main} applies for \(S\) sufficiently close to
\(\dim_2^V\nu\).

Let us now explain how the absolute-continuity part of
Theorem~\ref{thm: baby case} follows from Theorem~\ref{thm:main}. In the
setting of Theorem~\ref{thm: baby case}, the spectral-gap assumption
forces the rotations to generate \(\operatorname{SO}(3)\). Hence, for
\(k\in\{1,2\}\),
$\mathcal O_V=G_{3,k}
\text{ and }
\dim\mathcal O_V=2.$
We may therefore choose \(1<\sigma<q<2\), with both \(\sigma\) and
\(q\) arbitrarily close to \(1\). Moreover, the contraction ratio is
uniform,and the orbit-level operator
satisfies
$\|P_V\|_{L^2_0(\mathcal O_V)}
\leq
\|P\|_{L^2_0(\operatorname{SO}(3))}.$

By Feng and Hu \cite{feng2009dimension}, self-similar measures are exact
dimensional, and hence \(\dim_e\nu=\dim\nu\), where $\dim_e \nu$ is its entropy dimension (see also \cite{Solomyak2000Peres}). Moreover, by Shmerkin and
Solomyak \cite[Theorem~5.1 and Remark~5.2]{Shmerkin2016Solomyak},
$\lim_{q\downarrow1}\dim_q\nu=\dim_e\nu=\dim\nu.$
Since \(G=\operatorname{SO}(3)\),
Lemma~\ref{lem:full-orbit-energy-to-relative-lq} gives
$\dim_q^V\nu\geq\dim_q\nu
\text{ for every }V\in G_{3,k}.$

Suppose now that \(\dim\nu>k\) and
$\frac{\log\|P\|_{L^2_0(\operatorname{SO}(3))}}{\log|r|}
>
\frac{k}{\dim\nu-k}.$
As \(q\downarrow1\), \(\sigma\downarrow1\), \(b\downarrow k\), and
\(S\uparrow\dim\nu\), subject to
$1<\sigma<q
\text{ and }\qquad
k<b<S<\dim_q\nu,$
we have
$$b(q-1)
+
\frac{b\bigl(\sigma+b(q-1)\bigr)}{S-b}
\longrightarrow
\frac{k}{\dim\nu-k}.$$
The strict inequality above therefore allows us to choose
\(q,\sigma,b,S\) so that the hypothesis of
Theorem~\ref{thm:main} is satisfied. The theorem then gives an
\(L^q(V)\) density for \(\pi_V\nu\), and in particular
$\pi_V\nu\ll\operatorname{Leb}_V.$
This proves Part~\textup{(1)} of
Theorem~\ref{thm: baby case}.

\subsection{Sketch of proof, and further comments}
We begin by sketching the proof of Part~\textup{(2)} of
Theorem~\ref{thm: baby case}; afterwards, we explain the modifications
needed for the general result, Theorem~\ref{thm:main}. To highlight the
main idea, we restrict to line projections: so, \(d=3\) and \(k=1\), and
assume that the IFS has a common contraction ratio \(0<r<1\). So, $\Phi = \lbrace f_i(x)=r\cdot O_i\cdot x+t_i\rbrace_{i=1} ^n$ is a self-similar IFS on $\mathbb{R}^3$, $\mathbf{p}$ is a probability vector, and  $\nu = \nu_\mathbf{p}$ is the corresponding self-similar measure.

Fix \(b\in S^2\), let
$V_b:=\operatorname{span}\{b\}\in G_{3,1},$
and write \(\pi_b:=\pi_{V_b}\). We sketch the \(L^2\) conclusion,
corresponding taking \(s=1\) in \eqref{eq: baby condition}. Our aim is to prove that
$\xi\longmapsto \mathcal F_\xi(\pi_b\nu)$
belongs to \(L^2(\mathbb R)\). Write
$\lambda:=-\log r.$

\medskip
\noindent{\textbf{Step 1: Annular decomposition.}}
Let \(\alpha>0\) be an auxiliary parameter. We decompose the large
frequencies into the annuli
\[
e^{n(\alpha+\lambda)}
\leq |\xi|
<
e^{(n+1)(\alpha+\lambda)},
\qquad n\geq1.
\]
Since the Fourier transform is bounded, the contribution of bounded
frequencies is finite. Hence
\[
\begin{aligned}
\left\|
\xi\mapsto\mathcal F_\xi(\pi_b\nu)
\right\|_{L^2(\mathbb R)}^2
\leq {}&
C_0
+
\sum_{n=1}^{\infty}
\int_{e^{n(\alpha+\lambda)}}^{e^{(n+1)(\alpha+\lambda)}}
\left|\mathcal F_\xi(\pi_b\nu)\right|^2\,d\xi
\\
&+
\sum_{n=1}^{\infty}
\int_{-e^{(n+1)(\alpha+\lambda)}}^{-e^{n(\alpha+\lambda)}}
\left|\mathcal F_\xi(\pi_b\nu)\right|^2\,d\xi .
\end{aligned}
\]
In the remainder of the outline, we treat the positive-frequency
contribution; the negative-frequency contribution is dealt with similarly.
\medskip

\noindent{\textbf{Step 2: From self-similarity to spherical averages.}} 
The idea of this step is to use the self-similarity of $\nu$ to convert the Fourier transform in the prescribed
direction into an average over an \(n\)-step random-walk  over rotated
directions, governed by the rotations in the IFS with the corresponding weights. Then, via the \(L^2\) spectral gap, we can replace this random-walk average
by the invariant spherical average, while Sobolev embedding is used
 to control the error at the prescribed direction. Having precise control over the error term  is key to our method.
 
Thus, for a word \(u=(u_1,\ldots,u_n)\), write
$p_u:=p_{u_1}\cdots p_{u_n},
\,
O_u:=O_{u_1}\cdots O_{u_n}.$
Then
$f_u(x)=r^nO_ux+t_u$, and iterating the self-similarity relation we have,
$$\nu=\sum_{|u|=n}p_u\,f_u\nu.$$
Since
\(\mathcal F_\xi(\pi_b\nu)=\mathcal F_{\xi b}(\nu)\), we have
$\mathcal F_\xi(\pi_b\nu)
=
\sum_{|u|=n}
p_u
e^{-2\pi i\xi\langle b,t_u\rangle}
\mathcal F_{r^n\xi O_u^{-1}b}(\nu).$
The translations contribute only unimodular phases. Hence, by Jensen's
inequality,
\begin{equation}
\label{eq:sketch-jensen}
\left|\mathcal F_\xi(\pi_b\nu)\right|^2
\leq
\sum_{|u|=n}
p_u
\left|
\mathcal F_{r^n\xi O_u^{-1}b}(\nu)
\right|^2.
\end{equation}

Put
$R:=r^n|\xi|=e^{-n\lambda}|\xi|,$
and define \(F_R:G_{3,1}\to\mathbb R\) by
$F_R(V_c):=
\left|\mathcal F_{Rc}(\nu)\right|^2,
\, c\in S^2.$
This is well defined, since \(V_c=V_{-c}\) and
\(\left|\mathcal F_{-Rc}(\nu)\right|
=\left|\mathcal F_{Rc}(\nu)\right|\).
Since
$\|P\|_{L^2_0(\operatorname{SO}(3))}<1$
 the rotations \(O_i\) generate a dense subgroup of
\(\operatorname{SO}(3)\), so
$\operatorname{SO}(3)\cdot V_b=G_{3,1}.$
We abuse notation by denoting \(P\) also for the induced Markov operator on
\(L^2(G_{3,1},m)\), where \(m \in \mathcal{P}(G_{3,1})\) is the normalized $G$-invariant measure invariant, defined by
$PF(V):=\sum_i p_iF(O_i^{-1}V).$
Then the right-hand side of \eqref{eq:sketch-jensen} is precisely
$P^nF_R(V_b).$
Moreover, since this is the quotient action of the original operator
\(P\) on \(\operatorname{SO}(3)\),
$\|P\|_{L^2_0(G_{3,1})}
\leq
\|P\|_{L^2_0(\operatorname{SO}(3))}.$ See Lemma \ref{lem:ambient-gap-controls-orbit-sobolev}.

We now separate \(F_R\) into its mean and zero-mean parts:
$F_R
=
\int_{G_{3,1}}F_R\,dm
+
F_R^0.$

Fix
$\sigma>1=\frac12\dim G_{3,1}.$
By Sobolev embedding
$H^\sigma(G_{3,1})\hookrightarrow C^0(G_{3,1}),$ and so
\[
\begin{aligned}
P^nF_R(V_b)
&=
\int_{G_{3,1}}F_R\,dm
+
P^nF_R^0(V_b)
\\
&\leq
\int_{G_{3,1}}F_R\,dm
+
\left|P^nF_R^0(V_b)\right|
\\
&\leq
\int_{G_{3,1}}F_R\,dm
+
C_\sigma
\left\|P^nF_R^0\right\|_{H^\sigma(G_{3,1})}
\\
&\leq
\int_{G_{3,1}}F_R\,dm
+
C_\sigma
\left\|P\right\|_
{H^\sigma_0(G_{3,1})}^{\,n}
\left\|F_R^0\right\|_{H^\sigma(G_{3,1})}.
\end{aligned}
\] 
As explained after Theorem \ref{thm: pre main}, 
$\left\|P \right\|_{H^\sigma_0(G_{3,1})}
=
\left\|P\right\|_{L^2_0(G_{3,1})}.$
 Moreover, the compact support of \(\nu\) gives
$\left\|F_R^0\right\|_{H^\sigma(G_{3,1})}
\lesssim_{\nu,\sigma}
(1+R)^\sigma.$
Finally, the invariant average on \(G_{3,1}\) is the spherical average:
$\int_{G_{3,1}}F_R\,dm
=
\int_{S^2}
\left|\mathcal F_{R\omega}(\nu)\right|^2
\,d\sigma_{S^2}(\omega).$
Combining these estimates yields
\[
\left|\mathcal F_\xi(\pi_b\nu)\right|^2
\lesssim
\int_{S^2}
\left|
\mathcal F_{e^{-n\lambda}|\xi|\omega}(\nu)
\right|^2
\,d\sigma_{S^2}(\omega)
+
C_{\nu,\sigma}
\left(1+e^{-n\lambda}|\xi|\right)^\sigma
\left\|P\right\|_{L^2_0(\operatorname{SO}(3))}^{\,n}.
\]

\medskip
\noindent{\textbf{Step 3: Spherical averages and the final parameter
choice.}}
We insert the estimate from Step~2 into the positive-frequency
contribution from Step~1. It remains to control the main term
\[
(A)
:=
\sum_{n=1}^{\infty}
\int_{e^{n(\alpha+\lambda)}}^{e^{(n+1)(\alpha+\lambda)}}
\int_{S^2}
\left|
\mathcal F_{e^{-n\lambda}q\omega}(\nu)
\right|^2
\,d\sigma_{S^2}(\omega)\,dq
\]
and the error term
\[
(B)
:=
\sum_{n=1}^{\infty}
\int_{e^{n(\alpha+\lambda)}}^{e^{(n+1)(\alpha+\lambda)}}
\left(1+e^{-n\lambda}q\right)^\sigma
\left\|P\right\|_{L^2_0(\operatorname{SO}(3))}^{\,n}
\,dq .
\]

On the \(n\)-th annulus,
$e^{-n\lambda}q\asymp e^{n\alpha},$
while the length of the annulus is comparable to
\(e^{n(\alpha+\lambda)}\). Therefore
$(B)
\lesssim
\sum_{n=1}^{\infty}
\exp\left(n\bigl((\sigma+1)\alpha+\lambda\bigr)\right)
\left\|P\right\|_{L^2_0(\operatorname{SO}(3))}^{\,n}.$
Thus \((B)<\infty\) provided
\begin{equation}
\label{eq:sketch-error-condition}
-\log
\left\|P\right\|_{L^2_0(\operatorname{SO}(3))}
>
(\sigma+1)\alpha+\lambda.
\end{equation}

We next estimate \((A)\). In its \(n\)-th summand, make the change of
variables
$x=e^{-n\lambda}q,
\,
dq=e^{n\lambda}\,dx.$
The resulting interval is
$e^{n\alpha}
\leq x
<
e^{n\alpha+\alpha+\lambda}.$
On this interval,
$e^{n\lambda}\leq x^{\lambda/\alpha},$
and these intervals have uniformly bounded overlap as \(n\) varies.
Consequently,
\[
(A)
\lesssim
\int_{S^2}\int_1^\infty
x^{\lambda/\alpha}
\left|
\mathcal F_{x\omega}(\nu)
\right|^2
\,dx\,d\sigma_{S^2}(\omega).
\]
The spherical-average energy estimate of Peres and Schlag
\cite[Proposition~2.2]{Peres200Schlag}, stated below as
Lemma~\ref{lem:peres-schlag-line}, gives
\[
\int_{S^2}\int_1^\infty
x^{\lambda/\alpha}
\left|
\mathcal F_{x\omega}(\nu)
\right|^2
\,dx\,d\sigma_{S^2}(\omega)
\lesssim
I_{1+\lambda/\alpha}(\nu).
\]
Hence \((A)<\infty\) whenever
\begin{equation}
\label{eq:sketch-main-condition}
1+\frac{\lambda}{\alpha}<\dim_2\nu.
\end{equation}

It remains to choose \(\alpha\) so that
\eqref{eq:sketch-error-condition} and
\eqref{eq:sketch-main-condition} hold simultaneously. Equivalently, we
need
$\frac{\lambda}{\dim_2\nu-1}
<
\alpha
<
\frac{
-\log\left\|P\right\|_{L^2_0(\operatorname{SO}(3))}-\lambda
}{\sigma+1}.$
Since \(\sigma>1\) may be chosen arbitrarily close to \(1\), such an
\(\alpha\) exists precisely under the strict inequality
\[
-\log\left\|P\right\|_{L^2_0(\operatorname{SO}(3))}
>
\frac{\dim_2\nu+1}{\dim_2\nu-1}\,\lambda,
\]
which is exactly \eqref{eq: baby condition} with \(s=1\). We conclude
that
$q\longmapsto\mathcal F_q(\pi_b\nu)$
belongs to \(L^2(\mathbb R)\), and hence that \(\pi_b\nu\) has an
\(L^2\) density.
\medskip

\noindent{\textbf{The general case and organization.}}
The Sobolev conclusion in Theorem~\ref{thm: baby case} follows from the
same argument, starting in Step~1 with
$\left\|
\xi\mapsto
|\xi|^{(s-1)/2}\mathcal F_\xi(\pi_b\nu)
\right\|_{L^2(\mathbb R)}^2.$

For the \(L^2\)-case of Theorems ~\ref{thm:main} and \ref{thm: pre main}, the spherical average
over \(S^2\) is replaced by the average over the orbit
\(\mathcal O_V=G\cdot V\). This leads to the relative dimension
\(\dim_2^V\nu\); see the discussion in Section~\ref{Section Peres schlag}.

For \(1<q<2\), we replace the decomposition in Step~1 by a
Littlewood-Paley decomposition. At scale \(m\), the analogue of the
spherical-average function is
$F_{\mu,m}(W)
:=
\left\|
(\pi_W\mu)*\psi_m^W
\right\|_{L^q(W)}^q,
\, W\in\mathcal O_V.$
Self-similarity, convexity, and spectral gap are then applied to
\(F_{\mu,m}\) in the spirit of Step~2, while its orbit average is
controlled by the relative \(L^q\)-dimension \(\dim_q^V\nu\).

The restriction \(\sigma<q\) comes from the Sobolev regularity available
for \(F_{\mu,m}\). Indeed, let \(W=W_y\) be local coordinates on
\(\mathcal O_V\), and choose a smooth family of orthogonal
identifications
$T_y:\mathbb R^k\longrightarrow W_y.$
Writing
\[
G_y(x)
:=
\bigl((\pi_{W_y}\mu)*\psi_m^{W_y}\bigr)(T_yx),
\qquad x\in\mathbb R^k,
\]
we have
$F_{\mu,m}(W_y)
=
\int_{\mathbb R^k}|G_y(x)|^q\,dx.$
Differentiating in an orbit variable \(y_j\) gives
\[
\partial_{y_j}F_{\mu,m}(W_y)
=
q\,\operatorname{Re}
\int_{\mathbb R^k}
|G_y|^{q-2}\overline{G_y}\,
\partial_{y_j}G_y\,dx.
\]
For \(1<q<2\), the map
$z\longmapsto |z|^{q-2}z$
is only \((q-1)\)-H\"older continuous. Consequently, the first
derivatives of \(F_{\mu,m}\) have only fractional regularity of order
strictly less than \(q-1\), which yields
\[
\|F_{\mu,m}\|_{H^\sigma(\mathcal O_V)}
\lesssim
2^{m(\sigma+k(q-1))}
\qquad\text{for }\sigma<q.
\]
This is the estimate needed to control the pointwise mixing error. At
\(\sigma=q\), the available H\"older estimate does not give the
endpoint Sobolev bound required by the argument. When \(q=2\), the
functional is quadratic, no singular nonlinearity appears upon
differentiation, and no upper restriction on \(\sigma\) is needed. We do not know whether this precise upper restriction is necessary; nevertheless, some loss relative to the quadratic case \(q=2\) is natural, since the map \(z\mapsto |z|^q\) is not smooth at the origin.

When the contraction ratios are not uniform, we use a model,
tree-type decomposition of the IFS based on the contraction moduli,
inspired by the decompositions in our work
\cite{algom2023polynomial} and in the work of Baker and Sahlsten
\cite{Baker2023Sahl}, and \cite{Algom2022Baker}. The mixing estimate is then governed by the
operator norms of \(P_{a,V}\), averaged with weights
\(\widetilde p_a\), exactly as in the hypothesis of
Theorem~\ref{thm:main}. These estimates are established in
Section~\ref{Section equi}. A probabilistic argument is then used to
control the error term \((B)\), which is more delicate in the
non-uniformly contracting case. A further refinement allows one to
replace the largest contraction rate by the Lyapunov exponent; see
Remark~\ref{rmk: lyp}.

Our use of random walks is inspired by Varj\'u's work
\cite{Varju2013random}, while the general strategy is influenced by our
previous work on Fourier decay for stationary measures
\cite{algom2020decay,algom2021decay,algom2023polynomial,algom2024plane}.
\medskip

We end this Introduction with a few more comments.

\begin{remark}
\label{rmk: lyp}
Assume the setting and notation of Theorem~\ref{thm:main}. Recall that the
Lyapunov exponent of the self-similar measure is
$\chi:=-\sum_i p_i\log |r_i|.$
Let
$R:=\max_{\rho,\rho'\in\mathcal A}
\left|-\log \rho+\log \rho'\right|.$
Then the conclusion of Theorem~\ref{thm:main} remains valid if the spectral
inequality in Theorem~\ref{thm:main} is replaced by the following two
conditions. Let \(1<q\leq2\), write \(q'=q/(q-1)\), and suppose that
$k<b<S<\dim_q^V\nu.$
Set
$C_{q,b,S,\sigma}
:=
b(q-1)
+
\frac{b\bigl(\sigma+b(q-1)\bigr)}{S-b}.$
Assume that
$\frac12\dim\mathcal O_V<\sigma<q
\qquad\text{if }1<q<2,$
and
$\frac12\dim\mathcal O_V<\sigma
\qquad\text{if }q=2,$
and that for some \(\varepsilon\in(0,1)\),
\[
-\log\left(
\sum_{\rho\in\mathcal A}
\widetilde p_\rho
\left\|P_{\rho,V}\right\|_
{L^2_0(\mathcal O_V)}
\right)
>
C_{q,b,S,\sigma}\cdot \frac{\chi}{1-\varepsilon},
\]
and
$\frac{\varepsilon^2}{1-\varepsilon}
>
C_{q,b,S,\sigma}\cdot \frac{R^2}{2\chi}.$

When \(R=0\), all contractions have the same modulus,
\(\chi=-\log |r|\), and letting \(\varepsilon\downarrow0\) recovers the
uniform-contraction condition in Theorem~\ref{thm:main}.
\end{remark}

Our argument also yields explicit bounds on the Fourier dimension of the measures as in Theorem~\ref{thm:main}:
\begin{claim} \label{Claim: fourier}
Let \(\nu\) be as in Theorem~\ref{thm:main}. For
\(a\in\mathcal A\), let \(P_a^G\) be the Markov operator on
\(L^2(\operatorname{SO}(d))\) defined by
$P_a^G F(g)
=
\sum_{i\in I_a}\frac{p_i}{\widetilde p_a}F(O_i^{-1}g).$
Fix \(\sigma>\frac{d-1}{2}\), and set
$\Lambda
:=
-\log\left(
\sum_a \widetilde p_a
\left\|P_a^G\right\|_{L^2_0(\operatorname{SO}(d))}
\right).$ 
Assume that
 $\Gamma:=\frac{\Lambda}{\max_i(-\log |r_i|)}>0$,  and that for some \(\beta>0\),
\[
\int_{S^{d-1}}| \mathcal{F}_{R \omega} (\nu)|^2\,d\sigma(\omega)
\lesssim R^{-\beta}.
\]
Then, 
$\dim_F\nu
\geq
\min\left\{
d,\,
\frac{\beta\Gamma}{\beta+\sigma+\Gamma}
\right\}.$ 
\end{claim}
The quantity $\int_{S^{d-1}}| \mathcal{F}_{R \omega} (\nu)|^2\,d\sigma(\omega)
$  is called the spherical average of
$\nu$ at radius $R$; such averages play a classical role in
geometric measure theory. See e.g. the works of Mattila
\cite{Mattila1987spherical} and Sj\"olin \cite{sjolin1993estimates}, 
and the more recent work of Du, Guth, Ou, Wang, Wilson, and Zhang
\cite{DuGuthOuWangWilsonZhang2021}.

As outlined explicitly in Sahlsten's survey \cite{sahlsten2025survey}, 
Varj\'u  obtained polylogarithmic Fourier
decay for self-similar measures in dimensions \(d\geq3\) with dense rotations in \(\mathrm{SO}(d)\). It is well known that Varj\'u's method gives positive Fourier dimension under the stronger assumption of spectral gap.  We also point out that recently Baker, Khalil, and Sahlsten \cite{baker2024Kahlil} gave logarithmic decay rates under a weaker Diophantine assumption on the rotations or the contraction ratios for planar self-similar measures.

\subsection{Acknowledgements}
We thank Simon Baker,  Mike Hochman, Xiong Jin,  Pablo Shmerkin, P\'eter Varj\'u, and Meng Wu for helpful comments and discussions. We are  grateful to Tuomas Orponen for historical remarks about the Marstrand-Mattila projection theorem,  and to Shai Evra for  background on spectral gaps and the Lubotzky--Phillips--Sarnak construction.

\section{From spectral gap to effective equidistribution on the $G$-orbit} \label{Section equi}
The purpose of this section is to convert the spectral-gap hypotheses in
Theorem~\ref{thm:main} into quantitative mixing estimates on the orbit
\(\mathcal O_V\), to be used in the proof of the theorem.  To treat non-uniform contraction ratios, we  introduce a model
decomposition of the IFS according to the contraction moduli. Related
decompositions have recently appeared in, for example,
\cite{algom2023polynomial,baker2024polynomial,Baker2026dis,algom2024plane},
although in settings different from the present one.

The underlying idea is related to the \(L^2\)-smoothing argument used by
Kittle and Kogler \cite[Section~7]{kittle2024absolute} for random walks on
compact orthogonal groups. Here, however, we work directly with Sobolev
spaces on the orbit \(\mathcal O_V\), and only with the particular
quantities arising in our (generally, Besov) estimates. This allows us to retain explicit
control of the error terms and to avoid some of the auxiliary constants
that arise in more general \(L^2\)-smoothing arguments.
\subsection{A model decomposition} \label{Section model}
Let
$\Phi=\{f_i(x)=r_i O_i x+t_i\}_{i=1}^n$
be a self-similar IFS on \(\mathbb R^d\), \(d\geq 3\), with probability vector
\(\mathbf p=(p_1,\ldots,p_n)\). Let \(\nu=\nu_{\mathbf p}\) be the
corresponding self-similar measure.  Write
$a_i:=|r_i|,\qquad 0<a_i<1.$
Let
$\{a^{(1)},\ldots,a^{(m)}\}=\{a_i:1\leq i\leq n\}$ 
be the distinct contraction moduli.  For \(1\leq j\leq m\), set
\[
\mathcal I_j:=\{i\in\{1,\ldots,n\}:a_i=a^{(j)}\},
        \,
\Phi_j:=\{f_i:i\in\mathcal I_j\}, \text{ and }
\widetilde p_j:=\sum_{i\in\mathcal I_j}p_i.
\]
Thus each subsystem \(\Phi_j\) has a common contraction modulus
\(a^{(j)}\), but the orthogonal parts and translations may vary inside
\(\Phi_j\). Note that we allow for $\Phi_j$ to be a singleton.

Let
\[
\Omega:=\{1,\ldots,m\}^{\mathbb N},
        \qquad
Q:=(\widetilde p_1,\ldots,\widetilde p_m)^{\mathbb N}.
\]
For a finite word \(\eta=(\eta_1,\ldots,\eta_q)\in\{1,\ldots,m\}^*\), write
\[
[\eta]:=\{\omega\in\Omega:\omega_1=\eta_1,\ldots,\omega_q=\eta_q\} \subseteq \Omega
\]
for the corresponding cylinder.  We also set
\[
X^{(\eta)}_{|\eta|}
:=
\mathcal I_{\eta_1}\times\cdots\times \mathcal I_{\eta_{|\eta|}}.
\]
Thus an element \(u\in X^{(\eta)}_{|\eta|}\) is a word
\(u=(u_1,\ldots,u_{|\eta|})\)  such that
\(u_\ell\in \mathcal I_{\eta_\ell}\) for every \(\ell\).

For \(u=(u_1,\ldots,u_q)\), define maps and weights via
\[
f_u:=f_{u_1}\circ\cdots\circ f_{u_q},
        \qquad
p_u:=\prod_{\ell=1}^q p_{u_\ell}.
\]
If \(u\in X^{(\eta)}_{|\eta|}\), define the conditional weight
\[
p^{(\eta)}(u)
:=
\prod_{\ell=1}^{|\eta|}
\frac{p_{u_\ell}}{\widetilde p_{\eta_\ell}}.
\]
Thus,
\[
Q([\eta])\,p^{(\eta)}(u)=p_u,
        \qquad u\in X^{(\eta)}_{|\eta|}.
\]

It will be convenient to write
\[
\lambda_j:=-\log a^{(j)}>0,
        \qquad
\chi:=\sum_{j=1}^m \widetilde p_j\lambda_j.
\]
Thus, \(\chi\) is both the Lyapunov exponent of the original self-similar measure, and of the  contraction process in the
model space \((\Omega,Q)\).  We also write
\[
R:=\max_{1\leq i,j\leq m}|\lambda_i-\lambda_j|
   =
   \max_{1\leq i,j\leq m}
   \left|\log a^{(i)}-\log a^{(j)}\right|.
\]

For \(L>0\), define the stopping time
$\tau_L:\Omega\to\mathbb N$ 
by
\begin{equation} \label{eq: stopping time}
\tau_L(\omega)
:=
\min\left\{
q\in\mathbb N:
\prod_{\ell=1}^q a^{(\omega_\ell)}<e^{-L}
\right\}, \quad \omega \in \Omega.
\end{equation}
Equivalently, if $S_q(\omega):=\sum_{\ell=1}^q \lambda_{\omega_\ell},$
then
$\tau_L(\omega)=\min\{q\in\mathbb N:S_q(\omega)>L\}.$
Finally, define the cut-set
\begin{equation} \label{eq: omega n}
\Omega_L
:=
\{\eta:\eta=\omega|_{\tau_L(\omega)}
        \text{ for some }\omega\in\Omega\}.
\end{equation}
Equivalently, \(\eta\in\Omega_L\) if and only if
$\prod_{\ell=1}^{|\eta|}a^{(\eta_\ell)}<e^{-L}$
and, when \(|\eta|\geq 2\),
$\prod_{\ell=1}^{|\eta|-1}a^{(\eta_\ell)}\geq e^{-L}.$
The set \(\Omega_L\) is finite, since the numbers \(\lambda_j\) are uniformly
bounded above and below away from \(0\).

\begin{lemma}
\label{lem:stopping-model}
The following holds in the Bernoulli space $(\Omega, Q)$.

\begin{enumerate}
\item For every \(L>0\),
\[
\nu
=
\int_{\Omega}
\left[
\sum_{u\in X^{(\omega|_{\tau_L(\omega)})}_{\tau_L(\omega)}}
p^{(\omega|_{\tau_L(\omega)})}(u)\,f_u\nu
\right]\,dQ(\omega).
\]
Equivalently, $\nu
=
\sum_{\eta\in\Omega_L}
Q([\eta])
\sum_{u\in X^{(\eta)}_{|\eta|}}
p^{(\eta)}(u)\,f_u\nu.$

\item For every \(\varepsilon\in(0,1)\),
\[
Q\left(
\tau_L\leq \frac{(1-\varepsilon)L}{\chi}
\right)
\leq
\exp\left(
-\frac{2\varepsilon^2\chi}{(1-\varepsilon)R^2}\,L
\right),
\]
with the convention that the right hand side is \(0\) if \(R=0\). In fact, one has the sharper Chernoff--Cramér bound
\[
Q\left(
\tau_L\leq \frac{(1-\varepsilon)L}{\chi}
\right)
\leq
\exp\left(
-\frac{1-\varepsilon}{\chi}\,
I\left(\frac{\chi}{1-\varepsilon}\right)L
\right), 
\]
\end{enumerate}
$\text{ where }
I(x):=\sup_{\theta\geq 0}
\left\{
\theta x-\log\left(
\sum_{j=1}^m \widetilde p_j e^{\theta\lambda_j}
\right)
\right\}.$
\end{lemma}

\begin{proof}
We begin with Part (1).  Let $\Sigma:=\{1,\ldots,n\}^{\mathbb N}$ 
be the usual symbolic space for the original IFS, endowed with the Bernoulli
measure \(\mathbf p^{\mathbb N}\).  For a finite word
\(\rho=(\rho_1,\ldots,\rho_q)\in\{1,\ldots,n\}^q\), write
\[
f_\rho=f_{\rho_1}\circ\cdots\circ f_{\rho_q},
        \qquad
p_\rho=\prod_{\ell=1}^q p_{\rho_\ell},
        \qquad
a_\rho=\prod_{\ell=1}^q a_{\rho_\ell}.
\]

We can define a cutset directly in $\Sigma$ via
\[
\mathcal W_L
:=
\left\{
\rho\in\{1,\ldots,n\}^*:
a_\rho<e^{-L}
\text{ and }
a_{\rho^-}\geq e^{-L}
\right\},
\]
where \(\rho^-\) denotes the word obtained from \(\rho\) by deleting its last
letter, and where \(a_{\emptyset}=1\).  Similarly to $\Omega_L$, the set \(\mathcal W_L\) is also finite.

We first claim that
\[
\nu=\sum_{\rho\in\mathcal W_L}p_\rho f_\rho\nu.
\]
Indeed, choose \(M\) larger than the length of every word in \(\mathcal W_L\).
By iterating the self-similar identity \(\nu=\sum_i p_i f_i\nu\), we have
\[
\nu=\sum_{|I|=M}p_I f_I\nu.
\]
Every word \(I\) of length \(M\) has a unique prefix
\(\rho\in\mathcal W_L\).  Grouping the preceding sum according to this prefix,
we obtain
\[
\nu
=
\sum_{\rho\in\mathcal W_L}
\sum_{\substack{|I|=M\\ I|_{|\rho|}=\rho}}
p_I f_I\nu.
\]
Writing \(I=\rho J\), this becomes
\[
\nu
=
\sum_{\rho\in\mathcal W_L}
p_\rho f_\rho
\left(
\sum_{|J|=M-|\rho|}p_J f_J\nu
\right).
\]
The inner sum is again \(\nu\), by self-similarity.  Hence
\begin{equation} \label{eq: for nu}
\nu=\sum_{\rho\in\mathcal W_L}p_\rho f_\rho\nu.
\end{equation}

Now we rewrite this cutset decomposition in the two-stage model.  Given
\(\eta\in\Omega_L\) and \(u\in X^{(\eta)}_{|\eta|}\), the word \(u\) belongs to
\(\mathcal W_L\).  Conversely, every \(\rho\in\mathcal W_L\) determines a unique
\(\eta\in\Omega_L\), namely the word recording the contraction class of each
letter of \(\rho\).  Therefore
\[
\mathcal W_L
=
\bigsqcup_{\eta\in\Omega_L} X^{(\eta)}_{|\eta|}.
\]
Moreover, for \(u\in X^{(\eta)}_{|\eta|}\),
\[
Q([\eta])p^{(\eta)}(u)
=
\prod_{\ell=1}^{|\eta|}\widetilde p_{\eta_\ell}
\prod_{\ell=1}^{|\eta|}
\frac{p_{u_\ell}}{\widetilde p_{\eta_\ell}}
=
\prod_{\ell=1}^{|\eta|}p_{u_\ell}
=
p_u.
\]
Thus
\[
\sum_{\rho\in\mathcal W_L}p_\rho f_\rho\nu
=
\sum_{\eta\in\Omega_L}
Q([\eta])
\sum_{u\in X^{(\eta)}_{|\eta|}}
p^{(\eta)}(u)f_u\nu.
\]
Via \eqref{eq: for nu}, this proves the second displayed formula in part (1).

The integral formula follows from the same identity, since for every
\(\eta\in\Omega_L\) and every \(\omega\in[\eta]\), we have
$\tau_L(\omega)=|\eta|,
\omega|_{\tau_L(\omega)}=\eta.$
Thus the integrand is constant on each stopped cylinder \([\eta]\), and so the
integral over \(\Omega\) is exactly the preceding finite sum. This completes the proof of Part (1).

We proceed now to Part (2).  Let
\[
Y_\ell(\omega):=\lambda_{\omega_\ell},
        \qquad
S_q(\omega):=\sum_{\ell=1}^q Y_\ell(\omega).
\]
Then \(Y_1,Y_2,\ldots\) are i.i.d. bounded random variables with respect to
\(Q\), and
$\E_Q Y_1=\chi.$
Moreover,
$\tau_L(\omega)=\min\{q:S_q(\omega)>L\}.$ 

Fix \(\varepsilon\in(0,1)\), and put
$c:=\frac{1-\varepsilon}{\chi}.$
Since \(\tau_L\) is integer-valued,
\[
\left\{\tau_L\leq \frac{(1-\varepsilon)L}{\chi}\right\}
=
\{\tau_L\leq \lfloor c L\rfloor\}.
\]
If \(\tau_L\leq \lfloor c L\rfloor\), then, since \(S_q\) is increasing in \(q\),
$S_{\lfloor cL\rfloor}>L.$
Hence
\[
Q\left(
\tau_L\leq \frac{(1-\varepsilon)L}{\chi}
\right)
\leq
Q(S_{\lfloor cL\rfloor}>L).
\]

We first give the sharper Chernoff bound.  For \(\theta\geq0\), define
\[
\Lambda(\theta)
:=
\log \E_Q(e^{\theta Y_1})
=
\log\left(
\sum_{j=1}^m \widetilde p_j e^{\theta\lambda_j}
\right).
\]
By Markov's inequality,
\[
Q(S_{\lfloor cL\rfloor}>L)
\leq
e^{-\theta L}
\E_Q(e^{\theta S_{\lfloor cL\rfloor}})
=
\exp\left(
-\theta L+\lfloor cL\rfloor\Lambda(\theta)
\right).
\]
Since \(\lfloor cL\rfloor\leq cL\), and since \(\Lambda(\theta)\geq0\) for
\(\theta\geq0\), this gives
\[
Q(S_{\lfloor cL\rfloor}>L)
\leq
\exp\left(
-L\{\theta-c\Lambda(\theta)\}
\right).
\]
Optimizing over \(\theta\geq0\), we obtain
\[
Q\left(
\tau_L\leq \frac{(1-\varepsilon)L}{\chi}
\right)
\leq
\exp\left(
-L\sup_{\theta\geq0}\{\theta-c\Lambda(\theta)\}
\right).
\]
Since \(c=(1-\varepsilon)/\chi\), the exponent can be rewritten as
\[
\sup_{\theta\geq0}\{\theta-c\Lambda(\theta)\}
=
\frac{1-\varepsilon}{\chi}
\sup_{\theta\geq0}
\left\{
\theta\frac{\chi}{1-\varepsilon}-\Lambda(\theta)
\right\}.
\]
Therefore
\[
Q\left(
\tau_L\leq \frac{(1-\varepsilon)L}{\chi}
\right)
\leq
\exp\left(
-\frac{1-\varepsilon}{\chi}\,
I\left(\frac{\chi}{1-\varepsilon}\right)L
\right),
\]
where
$I(x):=\sup_{\theta\geq0}\{\theta x-\Lambda(\theta)\}.$
This proves the refined estimate.

It remains to recover the explicit Hoeffding form.  If \(R=0\), then all
\(\lambda_j\) are equal to \(\chi\), so
$S_q=q\chi$
deterministically.  In this case
\[
\tau_L=\left\lfloor \frac{L}{\chi}\right\rfloor+1
>
\frac{L}{\chi}
>
\frac{(1-\varepsilon)L}{\chi},
\]
and the event in question is empty.

Assume therefore that \(R>0\).  Hoeffding's lemma gives, for every
\(\theta\geq0\),
\[
\Lambda(\theta)-\theta\chi
\leq
\frac{\theta^2 R^2}{8}.
\]
Consequently,
\[
\theta-c\Lambda(\theta)
=
\theta-c\theta\chi-c(\Lambda(\theta)-\theta\chi)
\geq
\varepsilon\theta
-
\frac{c\theta^2R^2}{8}.
\]
Optimizing the right hand side over \(\theta\geq0\), and recalling that
\(c=(1-\varepsilon)/\chi\), gives
\[
\sup_{\theta\geq0}\{\theta-c\Lambda(\theta)\}
\geq
\sup_{\theta\geq0}
\left\{
\varepsilon\theta
-
\frac{(1-\varepsilon)R^2}{8\chi}\theta^2
\right\}
=
\frac{2\varepsilon^2\chi}{(1-\varepsilon)R^2}.
\]
Therefore
\[
Q\left(
\tau_L\leq \frac{(1-\varepsilon)L}{\chi}
\right)
\leq
\exp\left(
-\frac{2\varepsilon^2\chi}{(1-\varepsilon)R^2}\,L
\right).
\]
This completes the proof.
\end{proof}

\subsection{Sobolev spaces on $\mathcal O_V$} \label{subsec: sobolev}
Recall that
\[
G:=\overline{\langle O_1,\ldots,O_n\rangle}\subseteq \mathrm{SO}(d),
\,
V\in G_{d,k},
\,
\mathcal O_V:=G\cdot V .
\]
We equip \(G_{d,k}\) with its standard \(\mathrm{SO}(d)\)-invariant
Riemannian metric and \(\mathcal O_V\) with the induced Riemannian metric.
Then \(\mathcal O_V\) is a compact homogeneous Riemannian manifold, and
the action of \(G\) on \(\mathcal O_V\) is by isometries.  Let \(m_V\)
denote the normalized Riemannian volume measure on \(\mathcal O_V\);
equivalently, \(m_V\) is the normalized \(G\)-invariant probability
measure on \(\mathcal O_V\).

For \(\rho\in\mathcal P(G)\) define the Markov operator
$P_{\rho,V}:L^2(\mathcal O_V,m_V)\to L^2(\mathcal O_V,m_V)$
by
\[
P_{\rho,V}F(W)
:=
\int_G F(g^{-1}W)\,d\rho(g).
\]
It preserves constant functions, and restricts to a contraction on
\begin{equation} \label{def: L two zero}
L^2_0(\mathcal O_V,m_V):=\left\{
f\in L^2(\mathcal O_V,m_V):
\int_{\mathcal O_V} f\,dm_V=0
\right\}.
\end{equation}
For $\rho,\rho'\in\mathcal P(G)$, let $\rho*\rho'$ denote the pushforward
of $\rho\times\rho'$ under $(g,h)\mapsto gh$. With our convention for
$P_{\rho,V}$, we have $P_{\rho,V}P_{\rho',V}=P_{\rho*\rho',V}.$ 
For $\rho\in\mathcal P(G)$, let $\rho\cdot\delta_V$ denote the law of
$g^{-1}V$ when $g$ has law $\rho$.

We next introduce the Sobolev spaces on \(\mathcal O_V\) that will be used
below; we follow the presentation as in \cite[Chapter 4]{Taylor1996partial}.  If \(F\in C^\infty(\mathcal O_V)\), its Riemannian gradient
\(\nabla_{\mathcal O_V}F\) is the tangent vector field characterized by
\[
\big\langle \nabla_{\mathcal O_V}F(W),X\big\rangle_W
=
dF_W(X),
\qquad
W\in\mathcal O_V,\quad X\in T_W\mathcal O_V .
\]
If \(X\) is a smooth tangent vector field on \(\mathcal O_V\), its
divergence \(\operatorname{div}_{\mathcal O_V}X\) is characterized by
\[
\int_{\mathcal O_V}
F\,\operatorname{div}_{\mathcal O_V}X\,dm_V
=
-
\int_{\mathcal O_V}
\big\langle \nabla_{\mathcal O_V}F,X\big\rangle\,dm_V
\qquad
\text{for every }F\in C^\infty(\mathcal O_V).
\]
We use the non-negative Laplace--Beltrami operator
\[
\Delta_{\mathcal O_V}F
:=
-
\operatorname{div}_{\mathcal O_V}
\bigl(\nabla_{\mathcal O_V}F\bigr),
\qquad
F\in C^\infty(\mathcal O_V).
\]
Thus, for \(F,H\in C^\infty(\mathcal O_V)\),
$\int_{\mathcal O_V}
(\Delta_{\mathcal O_V}F)H\,dm_V
=
\int_{\mathcal O_V}
\big\langle
\nabla_{\mathcal O_V}F,
\nabla_{\mathcal O_V}H
\big\rangle\,dm_V.$
In particular, \(\Delta_{\mathcal O_V}\) is non-negative and annihilates
constant functions.

For \(\sigma\geq0\), define \[ H^\sigma(\mathcal O_V) := \operatorname{Dom} \left((I+\Delta_{\mathcal O_V})^{\sigma/2}\right), \] with norm $\|F\|_{H^\sigma(\mathcal O_V)} := \left\| (I+\Delta_{\mathcal O_V})^{\sigma/2}F \right\|_{L^2(\mathcal O_V,m_V)}.$ Here the fractional power is defined by the spectral theorem for the non-negative self-adjoint operator \(\Delta_{\mathcal O_V}\). If \(\mathcal O_V\) has positive dimension, this agrees with the usual spectral definition using the discrete eigenvalues of \(\Delta_{\mathcal O_V}\); if \(\mathcal O_V\) is finite, all these Sobolev spaces coincide with \(L^2(\mathcal O_V,m_V)\).

We will use the Sobolev embedding theorem \cite[Chapter 4, Section 3]{Taylor1996partial}, in the following form: if
$\sigma>\frac12\dim\mathcal O_V,$
then every \(F\in H^\sigma(\mathcal O_V)\) has a continuous representative,
and there exists \(C_{V,\sigma}<\infty\) such that
\begin{equation}
\label{eq:sobolev-embedding-orbit}
\|F\|_{C^0(\mathcal O_V)}
\leq
C_{V,\sigma}
\|F\|_{H^\sigma(\mathcal O_V)}.
\end{equation}
In particular,
$|F(W)|
\leq
C_{V,\sigma}
\|F\|_{H^\sigma(\mathcal O_V)}$ for every $W\in \mathcal{O}_V$.

\subsection{Operator norms and exponential mixing on $H^\sigma (\mathcal O_V)$}
\label{subsec:stopped-cylinder-mixing}

For \(\sigma\geq0\), put
$H^\sigma_0(\mathcal O_V)
:=
H^\sigma(\mathcal O_V)\cap L^2_0(\mathcal O_V,m_V).$
\begin{lemma}
\label{lem:ambient-gap-controls-orbit-sobolev}
Let
\(\rho\in\mathcal P(G)\), and let \(P_\rho^G\) denote the  Markov operator on
\(L^2(G,m_G)\), where \(m_G\) is normalized Haar measure:
$P_\rho^G\Psi(h)
:=
\int_G \Psi(g^{-1}h)\,d\rho(g),
\,
\Psi\in L^2(G,m_G).$
Then, for every $\sigma\geq 0$,
\[
\left\|
P_{\rho,V}
\right\|_{H^\sigma_0(\mathcal O_V)\to H^\sigma_0(\mathcal O_V)}
=
\left\|
P_{\rho,V}
\right\|_{L^2_0(\mathcal O_V)\to L^2_0(\mathcal O_V)}
\leq
\left\|
P_\rho^G
\right\|_{L^2_0(G)\to L^2_0(G)}.
\]
\end{lemma}

\begin{proof}
Let $K_V:=\operatorname{Stab}_G(V).$ 
The map
$\iota_V:L^2(\mathcal O_V,m_V)\longrightarrow L^2(G,m_G),
\,
(\iota_VF)(h):=F(hV),$
is an isometric embedding onto the closed subspace of
right-\(K_V\)-invariant functions.  Moreover,
$\iota_V\bigl(P_{\rho,V}F\bigr)
=
P_\rho^G\bigl(\iota_VF\bigr).$ 
Indeed, for \(h\in G\),
\[
\begin{aligned}
\iota_V\bigl(P_{\rho,V}F\bigr)(h)
&=
P_{\rho,V}F(hV)
\\
&=
\int_G F(g^{-1}hV)\,d\rho(g)
\\
&=
\int_G (\iota_VF)(g^{-1}h)\,d\rho(g)
=
P_\rho^G(\iota_VF)(h).
\end{aligned}
\]
since \(m_V\) is the pushforward of \(m_G\) under the map
\(h\mapsto hV\), we have
$\int_G (\iota_VF)(h)\,dm_G(h)
=
\int_{\mathcal O_V}F(W)\,dm_V(W)$;
thus, \(\iota_V\) maps \(L^2_0(\mathcal O_V,m_V)\) isometrically into
\(L^2_0(G,m_G)\).
It follows that
\[
\left\|
P_{\rho,V}
\right\|_{L^2_0(\mathcal O_V)\to L^2_0(\mathcal O_V)}
\leq
\left\|
P_\rho^G
\right\|_{L^2_0(G)\to L^2_0(G)}.
\]

It remains to establish the Sobolev estimate. For \(g\in G\), let \(U_g\) denote the unitary operator on
\(L^2(\mathcal O_V,m_V)\) given by
\[
U_gF(W)=F(g^{-1}W).
\]
Since \(g\) acts isometrically on \(\mathcal O_V\), \(U_g\) also acts
unitarily on \(H^\sigma(\mathcal O_V)\) for every \(\sigma\geq0\); and,
 \(U_g\)
commutes with the Laplace--Beltrami operator:
$U_g\Delta_{\mathcal O_V}
=
\Delta_{\mathcal O_V}U_g.$ 
Consequently, by spectral calculus,
$U_g(I+\Delta_{\mathcal O_V})^{\sigma/2}
=
(I+\Delta_{\mathcal O_V})^{\sigma/2}U_g,$
and averaging over \(g\) gives
\begin{equation} \label{eq: above}
P_{\rho,V}(I+\Delta_{\mathcal O_V})^{\sigma/2}
=
(I+\Delta_{\mathcal O_V})^{\sigma/2}P_{\rho,V}.
\end{equation}

Set
\(A_\sigma:=(I+\Delta_{\mathcal O_V})^{\sigma/2}.\)
Since \(A_\sigma1=1\), the operator \(A_\sigma\) restricts to an
isometric isomorphism
\(A_\sigma:
H^\sigma_0(\mathcal O_V)
\longrightarrow
L^2_0(\mathcal O_V,m_V).\)
Indeed, the calculation above shows that \(A_\sigma\) preserves the
zero-mean subspace, and its inverse
\((I+\Delta_{\mathcal O_V})^{-\sigma/2}\) does so as well.

By \eqref{eq: above},
\(A_\sigma P_{\rho,V}A_\sigma^{-1} =
P_{\rho,V}
\quad\text{on }L^2_0(\mathcal O_V,m_V).\)
Therefore,
\[
\left|
P_{\rho,V}
\right|_{H^\sigma_0(\mathcal O_V)\to H^\sigma_0(\mathcal O_V)}
=
\left|
P_{\rho,V}
\right|_{L^2_0(\mathcal O_V,m_V)\to L^2_0(\mathcal O_V,m_V)}.
\]
Together with the ambient-group estimate proved above, this completes the
proof.
\end{proof}

For each model symbol \(j\in\{1,\ldots,m\}\), define
\[
\rho_j
:=
\sum_{i\in\mathcal I_j}
\frac{p_i}{\widetilde p_j}\delta_{O_i}
\in\mathcal P(G).
\]
Equivalently, if \(\eta=(\eta_1,\ldots,\eta_\ell)\), then
$\rho_{\eta_1}*\cdots *\rho_{\eta_\ell}
=
\sum_{u\in X^{(\eta)}_\ell}
p^{(\eta)}(u)\,\delta_{O_u},$
where
$f_u(x)=r_uO_u x+t_u,$ and
$O_u:=O_{u_1}\cdots O_{u_\ell}.$

Let \(L>0\) and \(\eta\in\Omega_L\), recall \eqref{eq: omega n}.  Since all words
\(u\in X^{(\eta)}_{|\eta|}\) have the same contraction modulus, there is
\(\varepsilon_{\eta,L}\geq0\) such that
\begin{equation}
\label{eq: eps eta L relative}
|r_u|
=
\prod_{\ell=1}^{|\eta|}a^{(\eta_\ell)}
=
e^{-L-\varepsilon_{\eta,L}},
        \qquad u\in X^{(\eta)}_{|\eta|}.
\end{equation}
We write
$r_\eta:=e^{-L-\varepsilon_{\eta,L}}$
and
$\mu_\eta
:=
\sum_{u\in X^{(\eta)}_{|\eta|}}
p^{(\eta)}(u)f_u\nu.$

\begin{claim}
\label{claim:stopped-cylinder-mixing}
Let
\(\sigma>\frac{1}{2}\dim \mathcal O_V\).
There exists a constant \(C_{V,\sigma}<\infty\) such that for every
\(L>0\), every \(\eta\in\Omega_L\), every \(W\in\mathcal O_V\), and every
\(F\in H^\sigma(\mathcal O_V)\),
\[
\begin{aligned}
&
\left|
\sum_{u\in X^{(\eta)}_{|\eta|}}
p^{(\eta)}(u)\,
F(O_u^{-1}W)
-
\int_{\mathcal O_V}F\,dm_V
\right|
\\
&\qquad\leq
C_{V,\sigma}
\prod_{\ell=1}^{|\eta|}
\left\|
P_{\rho_{\eta_\ell},V}
\right\|_{L^2 _0(\mathcal O_V)\to L^2 _0(\mathcal O_V)}
\cdot
\left\|
F-\int_{\mathcal O_V}F\,dm_V
\right\|_{H^\sigma(\mathcal O_V)}.
\end{aligned}
\]
\end{claim}

\begin{proof}
Write
\[
\overline F:=\int_{\mathcal O_V}F\,dm_V,
\qquad
F_0:=F-\overline F.
\]
Then \(F_0\in H^\sigma_0(\mathcal O_V)\). Since each \(P_{\rho,V}\) preserves
constants, it is enough to estimate the same expression with \(F\) replaced by
\(F_0\).

By the definition of the measures \(\rho_j\), and by the identity
\[
\rho_{\eta_1}*\cdots *\rho_{\eta_{|\eta|}}
=
\sum_{u\in X^{(\eta)}_{|\eta|}}
p^{(\eta)}(u)\,\delta_{O_u},
\]
we have
\[
\sum_{u\in X^{(\eta)}_{|\eta|}}
p^{(\eta)}(u)\,
F_0(O_u^{-1}W)
=
P_{\rho_{\eta_1},V}\cdots P_{\rho_{\eta_{|\eta|}},V}F_0(W).
\]
Since \(\sigma>\frac12\dim\mathcal O_V\), Sobolev embedding gives
\[
\left|
P_{\rho_{\eta_1},V}\cdots P_{\rho_{\eta_{|\eta|}},V}F_0(W)
\right|
\leq
C_{V,\sigma}
\left\|
P_{\rho_{\eta_1},V}\cdots P_{\rho_{\eta_{|\eta|}},V}F_0
\right\|_{H^\sigma(\mathcal O_V)}.
\]
The operators \(P_{\rho,V}\) preserve \(H^\sigma_0(\mathcal O_V)\). Hence
\[
\begin{aligned}
&
\left\|
P_{\rho_{\eta_1},V}\cdots P_{\rho_{\eta_{|\eta|}},V}F_0
\right\|_{H^\sigma(\mathcal O_V)}
\\
&\qquad\leq
\prod_{\ell=1}^{|\eta|}
\left\|
P_{\rho_{\eta_\ell},V}
\right\|_{H^\sigma_0(\mathcal O_V)\to H^\sigma_0(\mathcal O_V)}
\cdot
\|F_0\|_{H^\sigma(\mathcal O_V)}.
\end{aligned}
\]
Combining the last two estimates with Lemma \ref{lem:ambient-gap-controls-orbit-sobolev} proves the claim.
\end{proof}
\subsection{Besov spaces and Littlewood-Paley type averages}
\label{subsec:besov-orbit-estimates}
We first recall the Littlewood-Paley conventions underlying the definition of
\(\dim_q^V\nu\), and record the Besov and orbit-Sobolev estimates that will
be used in the proof of Theorem~\ref{thm:main}.
.
Let \(U\in G_{d,k}\).  Fix a smooth radial Schwartz function
\(\psi\) on \(\mathbb R^k\) such that
$\supp \widehat\psi
\subset
\{\xi\in\mathbb R^k:1/2\leq |\xi|\leq2\}.$
For \(m\geq1\), let \(\psi_m^U\) denote the corresponding Littlewood-Paley
kernel on \(U\), defined by
$\widehat{\psi_m^U}(\xi)=\widehat\psi(2^{-m}\xi),
\, \xi\in U.$
We also fix a smooth low-frequency cutoff \(\varphi_0\) on \(\mathbb R^k\), and
write \(\varphi_0^U\) for the corresponding kernel on \(U\). We choose \(\varphi_0\) and \(\psi\) so that
\begin{equation} \label{eq: LP partition}
\widehat{\varphi_0}(\xi)+\sum_{m\geq1}\widehat\psi(2^{-m}|\xi|)=1,
\qquad \xi\in\mathbb R^k.
\end{equation}

For \(0<\beta<\infty\) and \(1<q\leq2\), the Besov space
\(B^\beta_{q,q}(U)\) consists of all tempered distributions \(f\) on \(U\) such
that
\begin{equation} \label{eq: besov}
\|f\|_{B^\beta_{q,q}(U)}^q
:=
\|f*\varphi_0^U\|_{L^q(U)}^q
+
\sum_{m\geq1}
2^{mq\beta}
\|f*\psi_m^U\|_{L^q(U)}^q
<\infty .
\end{equation}

The following Lemma is standard:
\begin{lemma} \label{lem:besov-reconstruction-measure} Let \(\mu\in\mathcal P(U)\), let \(1<q<\infty\), and let \(\beta>0\). Suppose that \[ \sum_{m\geq1} 2^{mq\beta} \|\mu*\psi_m^U\|_{L^q(U)}^q <\infty. \] Then \(\mu\ll\operatorname{Leb}_U\), and its density belongs to \(B^\beta_{q,q}(U)\). 
\end{lemma}
 \begin{proof} First, Young's inequality gives $\|\mu*\varphi_0^U\|_{L^q(U)} \leq \|\varphi_0^U\|_{L^q(U)}.$  For \(N\geq1\), set  $f_N := \mu*\varphi_0^U+\sum_{m=1}^N\mu*\psi_m^U.$
 By \eqref{eq: LP partition}, \(f_N\to\mu\) in the sense of tempered distributions on \(U\). On the other hand, since \(\beta>0\), Hölder's inequality gives 
 \[ \sum_{m\geq1}\|\mu*\psi_m^U\|_{L^q(U)} \leq \left(\sum_{m\geq1}2^{-mq'\beta}\right)^{1/q'} \left( \sum_{m\geq1} 2^{mq\beta} \|\mu*\psi_m^U\|_{L^q(U)}^q \right)^{1/q} <\infty, \]
  where \(q'=q/(q-1)\). Hence \(f_N\) converges in \(L^q(U)\) to some \(f\in L^q(U)\). Since also \(f_N\to\mu\) as distributions, we have  $\mu=f\,\operatorname{Leb}_U.$
  The assumed summability, together with the low-frequency estimate above, gives \(f\in B^\beta_{q,q}(U)\).
  \end{proof}

For
\(\mu\in\mathcal P(\mathbb R^d)\) and \(m\geq1\), define
\[
F_{\mu,m}(W)
:=
\left\|
(\pi_W\mu)*\psi_m^W
\right\|_{L^q(W)}^q,
\qquad W\in\mathcal O_V.
\]

\begin{proposition}
\label{prop:besov-orbit-sobolev-bound}
Let \(R_0>0\), and suppose that
\(\mu\in\mathcal P(\mathbb R^d)\) is supported in \(B(0,R_0)\).  Suppose either
\(1<q<2\) and \(0\leq\sigma<q\), or  \(q=2\) and \(\sigma\geq0\). Then
there exists a constant
\(C_{\psi,q,\sigma,V,R_0}<\infty\), independent of \(\mu\) and \(m\), such that
\[
\|F_{\mu,m}\|_{H^\sigma(\mathcal O_V)}
\leq
C_{\psi,q,\sigma,V,R_0}\,
2^{m(\sigma+k(q-1))}.
\]
\end{proposition}

\begin{proof}
We work in local coordinates on \(\mathcal O_V\), and then sum over a finite
atlas.  Let \(Y\subset\mathbb R^r\), \(r=\dim\mathcal O_V\), be a coordinate
chart, and write \(W_y\) for the corresponding element of \(\mathcal O_V\).
After shrinking the chart, we may assume that \(Y\) is a bounded Lipschitz convex
domain whose closure is compactly contained in the coordinate patch, and
choose smoothly in \(y\) an orthonormal identification
$A_y:\mathbb R^k\to W_y.$
We write \(A_y^*:\mathbb R^d\to\mathbb R^k\) for its adjoint. Set
\[
G_y(x)
:=
\bigl((\pi_{W_y}\mu)*\psi_m^{W_y}\bigr)(A_yx),
\qquad x\in\mathbb R^k.
\]
Then
$F_{\mu,m}(W_y)=\|G_y\|_{L^q(\mathbb R^k)}^q.$

We first derive estimates for \(G_y\) and its derivatives that are
uniform in \(y\) and explicit in the Littlewood-Paley scale \(m\). Since \(A_y\) is an
isometry onto \(W_y\), we have
$\pi_{W_y}z=A_yA_y^*z$.
And, under the identification \(A_y:\mathbb R^k\to W_y\),
$\psi_m^{W_y}(A_yu)=2^{mk}\psi(2^mu).$
Consequently,
\begin{equation}
\label{eq:local-expression-Gym}
G_y(x)
=
\int_{\mathbb R^d}
2^{mk}\psi\bigl(2^m(x-A_y^*z)\bigr)\,d\mu(z).
\end{equation}
All derivatives of \(A_y^*\) are uniformly bounded on \(Y\). Since
\(\mu\) is supported in \(B(0,R_0)\), the chain rule gives, for every
multi-index \(\alpha\),
\[
\begin{aligned}
&
\left|
\partial_y^\alpha
\left[
2^{mk}\psi\bigl(2^m(x-A_y^*z)\bigr)
\right]
\right|
\\
&\qquad\lesssim_{\psi,\alpha,V,R_0}
2^{m(k+|\alpha|)}
\sum_{|\beta|\leq|\alpha|}
\left|
(\partial^\beta\psi)
\bigl(2^m(x-A_y^*z)\bigr)
\right|.
\end{aligned}
\]
Taking the \(L^q(\mathbb R^k)\)-norm in \(x\), Minkowski's integral
inequality gives, for \(Nq>k\),
\[
\begin{aligned}
\|\partial_y^\alpha G_y\|_{L^q(\mathbb R^k)}
&\lesssim_{\psi,\alpha,V,R_0}
2^{m(k+|\alpha|)}
\int
\left\|
\bigl(1+2^m|\cdot-A_y^*z|\bigr)^{-N}
\right\|_{L^q(\mathbb R^k)}
\,d\mu(z)
\\
&=
2^{m(k+|\alpha|)}
\int
\left(
\int_{\mathbb R^k}
\bigl(1+2^m|x-A_y^*z|\bigr)^{-Nq}
\,dx
\right)^{1/q}
d\mu(z)
\\
&=
2^{m(k+|\alpha|)}2^{-mk/q}
\left(
\int_{\mathbb R^k}
(1+|u|)^{-Nq}\,du
\right)^{1/q}
\int d\mu(z)
\\
&\lesssim_{\psi,q,\alpha,V,R_0}
2^{m(k+|\alpha|-k/q)}
=
2^{m(|\alpha|+k/q')},
\text{ for }
q'=\frac{q}{q-1},
\end{aligned}
\]
where we used the change of variables
\(u=2^m(x-A_y^*z)\) and the fact that \(\mu\) is a probability measure.
So,
\begin{equation} \label{eq: estimates}
\|G_y\|_{L^q(\mathbb R^k)}
\lesssim_{\psi,q,V,R_0}
2^{mk/q'},\text{ and }
\|\partial_{y_i}G_y\|_{L^q(\mathbb R^k)} \lesssim_{\psi,q,V,R_0} 2^{m(1+k/q')}, \quad 1\leq i\leq r;
\end{equation}
$\text{ Consequently, } 
\|G_y-G_{y'}\|_{L^q(\mathbb R^k)}
\lesssim_{\psi,q,V,R_0}
2^{m(1+k/q')}|y-y'|.$\newline

\noindent{\textbf{The case $q=2$.}} 
First, assume \(q=2\).  In this case
$F_{\mu,m}(W_y)=\|G_y\|_{L^2(\mathbb R^k)}^2$
is quadratic in \(G_y\).  For every integer \(j\geq0\), differentiating \(j\)
times in \(y\) gives a finite sum of terms of the form
\[
\int_{\mathbb R^k}
\partial_y^{j_1}G_y(x)\,
\overline{\partial_y^{j_2}G_y(x)}\,dx,
\qquad j_1+j_2=j.
\]
Using the Cauchy–Schwarz inequality and the general scale estimate
immediately preceding \eqref{eq: estimates} specialized to \(q=2\), each such
term is bounded by
$$2^{m(j_1+k/2)}2^{m(j_2+k/2)}
=
2^{m(j+k)}.$$
Thus, for every integer \(j\geq0\),
$$\|\partial_y^jF_{\mu,m}\|_{L^\infty(Y)}
\lesssim_{\psi,j,V,R_0}
2^{m(j+k)}.$$
Therefore, for every integer \(N\geq0\),
\[
\|F_{\mu,m}\|_{H^N(Y)}
\lesssim_{\psi,N,V,R_0}
\sum_{|\alpha|\leq N}
\|\partial_y^\alpha F_{\mu,m}\|_{L^2(Y)}.
\]
Since \(Y\) has finite measure and
$\|\partial_y^\alpha F_{\mu,m}\|_{L^\infty(Y)}
\lesssim_{\psi,|\alpha|,V,R_0}
2^{m(|\alpha|+k)},$
we get
$$\|F_{\mu,m}\|_{H^N(Y)}
\lesssim_{\psi,N,V,R_0}
2^{m(N+k)}.$$
Now let \(\sigma\geq0\).  Choose an integer \(N\leq\sigma<N+1\), and write
$$\sigma=(1-\theta)N+\theta(N+1),
\, 0\leq\theta<1.$$
Interpolating between \(H^N(Y)\) and \(H^{N+1}(Y)\), we obtain
\[
\|F_{\mu,m}\|_{H^\sigma(Y)}
\lesssim
\|F_{\mu,m}\|_{H^N(Y)}^{1-\theta}
\|F_{\mu,m}\|_{H^{N+1}(Y)}^\theta .
\]
Using the integer estimates gives
$$\|F_{\mu,m}\|_{H^\sigma(Y)}
\lesssim_{\psi,\sigma,V,R_0}
\left(2^{m(N+k)}\right)^{1-\theta}
\left(2^{m(N+1+k)}\right)^\theta
=
2^{m(\sigma+k)}.$$ This is the desired estimate when \(q=2\).
\newline
\medskip

\medskip
\noindent{\textbf{The case  \(1<q<2,\, 0\leq\sigma<1\).}} Assume first that \(1<q<2\) and  \(0\leq\sigma<1\).  By the mean value theorem
$$\bigl||z|^q-|w|^q\bigr|
\lesssim_q
\bigl(|z|^{q-1}+|w|^{q-1}\bigr)|z-w|$$
implies, by Hölder's inequality and \eqref{eq: estimates}, that
\[
\begin{aligned}
|F_{\mu,m}(W_y)-F_{\mu,m}(W_{y'})|
&=
\left|
\int_{\mathbb R^k}
\bigl(|G_y(x)|^q-|G_{y'}(x)|^q\bigr)\,dx
\right|
\\
&\leq
\int_{\mathbb R^k}
\bigl||G_y(x)|^q-|G_{y'}(x)|^q\bigr|\,dx
\\
&\lesssim_q
\int_{\mathbb R^k}
\bigl(|G_y(x)|^{q-1}+|G_{y'}(x)|^{q-1}\bigr)
|G_y(x)-G_{y'}(x)|\,dx
\\
&\lesssim_q
\Bigl(
\|G_y\|_{L^q(\mathbb R^k)}^{q-1}
+
\|G_{y'}\|_{L^q(\mathbb R^k)}^{q-1}
\Bigr)
\|G_y-G_{y'}\|_{L^q(\mathbb R^k)}
\\
&\lesssim_{\psi,q,V,R_0}
2^{m(1+k(q-1))}|y-y'|.
\end{aligned}
\]
Also, by \eqref{eq: estimates},
$$|F_{\mu,m}(W_y)|
=
\|G_y\|_q^q
\lesssim_{\psi,q,V,R_0}
2^{mk(q-1)}.$$
Since \(Y\) is bounded, and  since \(m_V\) is a probability measure, these two estimates
imply
$$\|F_{\mu,m}\|_{L^2(Y)}
\lesssim_{\psi,q,V,R_0}
2^{mk(q-1)},\text{ and }
\|F_{\mu,m}\|_{H^1(Y)}
\lesssim_{\psi,q,V,R_0}
2^{m(1+k(q-1))}.$$
By interpolation between \(L^2(Y)\) and \(H^1(Y)\), since
\(0\leq\sigma<1\),
\[
\|F_{\mu,m}\|_{H^\sigma(Y)}
\lesssim_{\psi,q,\sigma,V,R_0}
\|F_{\mu,m}\|_{L^2(Y)}^{1-\sigma}
\|F_{\mu,m}\|_{H^1(Y)}^\sigma .
\]
Therefore,
$$\|F_{\mu,m}\|_{H^\sigma(Y)}
\lesssim_{\psi,q,\sigma,V,R_0}
\left(2^{mk(q-1)}\right)^{1-\sigma}
\left(2^{m(1+k(q-1))}\right)^\sigma
=
2^{m(\sigma+k(q-1))}.$$
\newline

\medskip
\noindent{\textbf{The case  \(1<q<2,\, 1\leq\sigma<q\).}}
We treat  the case \(1<q<2\) and \(1\leq\sigma<q\).  Write
\(\sigma=1+\tau\), where \(0\leq\tau<q-1\).  The map
$\mathcal N(g)=\|g\|_{L^q(\mathbb R^k)}^q$
is \(C^1\) on \(L^q(\mathbb R^k)\), with
$$D\mathcal N(g)[h]
=
q\,\operatorname{Re}\int_{\mathbb R^k}|g|^{q-2}\overline g\,h\,dx.$$
For each \(1\leq i\leq r=\dim Y\), the Banach-space chain rule and the
identity
$F_{\mu,m}(W_y)=\mathcal N(G_y)$
give
\begin{equation} \label{eq: formula}
\begin{aligned}
\partial_{y_i}F_{\mu,m}(W_y)
&=
D\mathcal N(G_y)\bigl[\partial_{y_i}G_y\bigr]
\\
&=
q\,\operatorname{Re}
\int_{\mathbb R^k}
|G_y|^{q-2}\overline{G_y}\,
\partial_{y_i}G_y\,dx.
\end{aligned}
\end{equation}
Hölder's inequality gives
\begin{equation} \label{eq: L infty bound}
|\partial_{y_i}F_{\mu,m}(W_y)|
\lesssim_q
\|G_y\|_q^{q-1}\|\partial_{y_i}G_y\|_q
\lesssim_{\psi,q,V,R_0}
2^{m(1+k(q-1))}.
\end{equation}

We next estimate the Hölder modulus of \(\partial_{y_i}F_{\mu,m}\).  Put
$J(z):=|z|^{q-2}\overline z,$ for $z\in\mathbb C.$ 
Since \(1<q<2\), the map \(J\) is \((q-1)\)-Hölder. 
Consequently,
\[
\|J(G_y)-J(G_{y'})\|_{L^{q'}(\mathbb R^k)}
\lesssim_q
\|G_y-G_{y'}\|_{L^q(\mathbb R^k)}^{q-1}.
\]
For each \(1\leq i\leq r\), \eqref{eq: estimates} applied with
\(|\alpha|=1\) and \(|\alpha|=2\), gives
\[
\|\partial_{y_i}G_y\|_{L^q(\mathbb R^k)}
\lesssim_{\psi,q,V,R_0}
2^{m(1+k/q')}, \text{ and }
\|\partial_{y_j}\partial_{y_i}G_y\|_{L^q(\mathbb R^k)}
\lesssim_{\psi,q,V,R_0}
2^{m(2+k/q')}
\,\,
1\leq j\leq r.
\]
Hence, by the triangle inequality,
\[
\|\partial_{y_i}G_y-\partial_{y_i}G_{y'}\|_{L^q(\mathbb R^k)}
\lesssim_{\psi,q,V,R_0}
2^{m(1+k/q')}.
\]
On the other hand, since \(Y\) is convex, the fundamental theorem of
calculus gives
\[
\partial_{y_i}G_y-\partial_{y_i}G_{y'}
=
\int_0^1
D_y\partial_{y_i}G_{y'+t(y-y')}[y-y']\,dt.
\]
Therefore,
\[
\|\partial_{y_i}G_y-\partial_{y_i}G_{y'}\|_{L^q(\mathbb R^k)}
\lesssim_{\psi,q,V,R_0}
2^{m(2+k/q')}|y-y'|.
\]
Combining these two bounds, we obtain
\[
\|\partial_{y_i}G_y-\partial_{y_i}G_{y'}\|_{L^q(\mathbb R^k)}
\lesssim_{\psi,q,V,R_0}
\min\left\{
2^{m(1+k/q')},
\,
2^{m(2+k/q')}|y-y'|
\right\}.
\]
Since
$\min\{A,B\}\leq A^{2-q}B^{q-1},
\, A,B\geq0,$
it follows that
\[
\|\partial_{y_i}G_y-\partial_{y_i}G_{y'}\|_{L^q(\mathbb R^k)}
\lesssim_{\psi,q,V,R_0}
2^{m(q+k/q')}|y-y'|^{q-1}.
\] By \eqref{eq: formula}, we write
\[
\begin{aligned}
&
\partial_{y_i}F_{\mu,m}(W_y)
-
\partial_{y_i}F_{\mu,m}(W_{y'})
\\
&\quad=
q\,\operatorname{Re}
\int_{\mathbb R^k}
\bigl(J(G_y)-J(G_{y'})\bigr)
\partial_{y_i}G_y\,dx
\\
&\qquad+
q\,\operatorname{Re}
\int_{\mathbb R^k}
J(G_{y'})
\bigl(
\partial_{y_i}G_y-\partial_{y_i}G_{y'}
\bigr)\,dx.
\end{aligned}
\]
By Hölder's inequality, the absolute value of the first term is bounded by
\[
\|J(G_y)-J(G_{y'})\|_{q'}
\|\partial_{y_i}G_y\|_q
\lesssim_{\psi,q,V,R_0}
2^{m(q+k(q-1))}|y-y'|^{q-1},
\]
while the second is bounded by
\[
\begin{aligned}
\|J(G_{y'})\|_{q'}
\|\partial_{y_i}G_y-\partial_{y_i}G_{y'}\|_q
&=
\|G_{y'}\|_q^{q-1}
\|\partial_{y_i}G_y-\partial_{y_i}G_{y'}\|_q
\\
&\lesssim_{\psi,q,V,R_0}
2^{m(q+k(q-1))}|y-y'|^{q-1}.
\end{aligned}
\] Combining these estimates gives
\[
|\partial_{y_i}F_{\mu,m}(W_y)
-
\partial_{y_i}F_{\mu,m}(W_{y'})|
\lesssim_{\psi,q,V,R_0}
2^{m(q+k(q-1))}|y-y'|^{q-1}.
\]
Thus
\begin{equation} \label{eq: thus}
\|\partial_yF_{\mu,m}\|_{C^{q-1}(Y)}
\lesssim_{\psi,q,V,R_0}
2^{m(q+k(q-1))}.
\end{equation}

We now pass from the Hölder estimate on
\(\partial_yF_{\mu,m}\) to a fractional Sobolev estimate. Set
$\gamma:=q-1\in(0,1).$ Recall that \(\sigma=1+\tau\) where \(0\leq\tau<q-1\); so, for \(\tau=0\), the required estimate follows directly from the
\(L^\infty(Y)\)-bound \eqref{eq: L infty bound}. Assume therefore that
$0<\tau<\gamma.$
For \(1\leq i\leq r\), define the scalar function
\[
h_i:Y\to\mathbb R,
\qquad
h_i(y):=\partial_{y_i}\bigl(F_{\mu,m}(W_y)\bigr),
\]
and put
$M_i:=\|h_i\|_{L^\infty(Y)},
\,
K_i:=\|h_i\|_{C^\gamma(Y)}.$
Write
\[
[h_i]_{H^\tau(Y)}^2
=
\int_Y\int_Y
\frac{|h_i(y)-h_i(y')|^2}
{|y-y'|^{r+2\tau}}
\,dy'\,dy.
\]
Since \(Y\subset\mathbb R^r\) is a bounded Lipschitz domain, the
Gagliardo-Slobodeckij characterization of fractional Sobolev spaces
(see, for example, \cite[Chapter~4]{Taylor1996partial}) gives, for
\(0<\tau<1\),
\[
\|h_i\|_{H^\tau(Y)}^2
\asymp_{Y,\tau}
\|h_i\|_{L^2(Y)}^2
+
[h_i]_{H^\tau(Y)}^2.
\]

Since
$|h_i(y)-h_i(y')|
\leq
\min\{2M_i,K_i|y-y'|^\gamma\},$
we obtain
\[
\begin{aligned}
[h_i]_{H^\tau(Y)}^2
&\leq
\int_Y\int_Y
\frac{
\min\{4M_i^2,K_i^2|y-y'|^{2\gamma}\}
}{
|y-y'|^{r+2\tau}
}
\,dy'\,dy
\\
&\lesssim_Y
\int_0^{\operatorname{diam}Y}
\frac{
\min\{4M_i^2,K_i^2\rho^{2\gamma}\}
}{
\rho^{1+2\tau}
}
\,d\rho.
\end{aligned}
\]
Splitting the integral at
$\rho_0:=\left(\frac{M_i}{K_i}\right)^{1/\gamma},$
with the evident interpretation if \(M_i=0\), and truncating at
\(\operatorname{diam}Y\) if necessary, gives
\[
\begin{aligned}
[h_i]_{H^\tau(Y)}^2
&\lesssim_Y
K_i^2
\int_0^{\rho_0}
\rho^{2\gamma-2\tau-1}\,d\rho
+
M_i^2
\int_{\rho_0}^{\operatorname{diam}Y}
\rho^{-2\tau-1}\,d\rho
\\
&\lesssim_{Y,\gamma,\tau}
M_i^{2(1-\tau/\gamma)}
K_i^{2\tau/\gamma}.
\end{aligned}
\]
Since \(K_i\geq M_i\), the same expression also controls
\(\|h_i\|_{L^2(Y)}^2\). Hence
\begin{equation}
\label{eq:holder-to-fractional-sobolev}
\|h_i\|_{H^\tau(Y)}
\lesssim_{Y,\gamma,\tau}
M_i^{1-\tau/\gamma}
K_i^{\tau/\gamma}.
\end{equation}
By \eqref{eq: L infty bound} and \eqref{eq: thus}
\[
M_i
\lesssim_{\psi,q,V,R_0}
2^{m(1+k(q-1))}
\text{  and  }
K_i
\lesssim_{\psi,q,V,R_0}
2^{m(q+k(q-1))},
\]
we obtain from \eqref{eq:holder-to-fractional-sobolev}
\[
\begin{aligned}
\|\partial_{y_i}F_{\mu,m}\|_{H^\tau(Y)}
&\lesssim_{\psi,q,\tau,V,R_0}
\left(
2^{m(1+k(q-1))}
\right)^{1-\tau/(q-1)}
\left(
2^{m(q+k(q-1))}
\right)^{\tau/(q-1)}
\\
&=
2^{m(1+\tau+k(q-1))}.
\end{aligned}
\]

By the standard local characterization of Sobolev spaces
(see, for example, \cite[Chapter~4]{Taylor1996partial}),
\[
\|F_{\mu,m}\|_{H^{1+\tau}(Y)}
\lesssim_{Y,\tau}
\|F_{\mu,m}\|_{L^2(Y)}
+
\sum_{i=1}^r
\|\partial_{y_i}F_{\mu,m}\|_{H^\tau(Y)}.
\]
Since
$\|F_{\mu,m}\|_{L^2(Y)}
\lesssim_{\psi,q,V,R_0}
2^{mk(q-1)},$
we conclude that
\[
\|F_{\mu,m}\|_{H^{1+\tau}(Y)}
\lesssim_{\psi,q,\tau,V,R_0}
2^{m(1+\tau+k(q-1))}.
\]
Writing \(\sigma=1+\tau\), this gives
\[
\|F_{\mu,m}\|_{H^\sigma(Y)}
\lesssim_{\psi,q,\sigma,V,R_0}
2^{m(\sigma+k(q-1))}
\]
for every \(1\leq\sigma<q\).  This completes the proof in the case
\(1<q<2\).
\newline
\medskip

The estimates above are uniform over a finite atlas of the compact manifold
\(\mathcal O_V\).  Summing over a partition of unity gives
$\|F_{\mu,m}\|_{H^\sigma(\mathcal O_V)}
\leq
C_{\psi,q,\sigma,V,R_0}\,
2^{m(\sigma+k(q-1))}.$
This proves the proposition.
\end{proof}

We shall use the preceding proposition together with
Claim~\ref{claim:stopped-cylinder-mixing} in the following form.

\begin{cor}
\label{cor:besov-stopped-cylinder-mixing}
Let \(R_0>0\), and suppose that
\(\mu\in\mathcal P(\mathbb R^d)\) is supported in \(B(0,R_0)\).  Assume either
that
$1<q<2
\,\text{ and }
\frac12\dim\mathcal O_V<\sigma<q,$
or 
$q=2
\,\text{ and } 
\frac12\dim\mathcal O_V<\sigma.$
Then there exists a constant
\(C_{\psi,q,\sigma,V,R_0}<\infty\) such that for every \(m\geq1\), every
\(L>0\), every \(\eta\in\Omega_L\), and every \(W\in\mathcal O_V\),
\[
\begin{aligned}
&
\sum_{u\in X^{(\eta)}_{|\eta|}}
p^{(\eta)}(u)\,
\left\|
(\pi_{O_u^{-1}W}\mu)*\psi_m^{O_u^{-1}W}
\right\|_{L^q(O_u^{-1}W)}^q
\\
&\qquad\leq
\int_{\mathcal O_V}
\left\|
(\pi_U\mu)*\psi_m^U
\right\|_{L^q(U)}^q
\,dm_V(U)
\\
&\qquad\quad+
C_{\psi,q,\sigma,V,R_0}\,
2^{m(\sigma+k(q-1))}
\prod_{\ell=1}^{|\eta|}
\left\|
P_{\rho_{\eta_\ell},V}
\right\|_{L^2_0(\mathcal O_V)\to L^2 _0(\mathcal O_V)} .
\end{aligned}
\]
\end{cor}

\begin{proof}
Apply Claim~\ref{claim:stopped-cylinder-mixing} to the function
$F_{\mu,m}(U)
=
\left\|
(\pi_U\mu)*\psi_m^U
\right\|_{L^q(U)}^q,
\, U\in\mathcal O_V.$ 
Since \(O_u^{-1}W\in\mathcal O_V\), the left hand side in the Corollary is precisely
$\sum_{u\in X^{(\eta)}_{|\eta|}}
p^{(\eta)}(u)\,
F_{\mu,m}(O_u^{-1}W).$
The orbit average of \(F_{\mu,m}\) is
\[
\int_{\mathcal O_V}
F_{\mu,m}(U)\,dm_V(U)
=
\int_{\mathcal O_V}
\left\|
(\pi_U\mu)*\psi_m^U
\right\|_{L^q(U)}^q
\,dm_V(U).
\]
Therefore Claim~\ref{claim:stopped-cylinder-mixing} gives
\[
\begin{aligned}
&
\sum_{u\in X^{(\eta)}_{|\eta|}}
p^{(\eta)}(u)\,
F_{\mu,m}(O_u^{-1}W)
\\
&\qquad\leq
\int_{\mathcal O_V}F_{\mu,m}\,dm_V
+
C_{V,\sigma}
\prod_{\ell=1}^{|\eta|}
\left\|
P_{\rho_{\eta_\ell},V}
\right\|_{L^2_0(\mathcal O_V)\to L^2_0(\mathcal O_V)}
\\
&\qquad\quad\cdot
\left\|
F_{\mu,m}
-
\int_{\mathcal O_V}F_{\mu,m}\,dm_V
\right\|_{H^\sigma(\mathcal O_V)} .
\end{aligned}
\]
Subtracting the mean changes the \(H^\sigma(\mathcal O_V)\)-norm by at most a
constant depending only on \(\mathcal O_V\).  Hence
Proposition~\ref{prop:besov-orbit-sobolev-bound} gives
\[
\left\|
F_{\mu,m}
-
\int_{\mathcal O_V}F_{\mu,m}\,dm_V
\right\|_{H^\sigma(\mathcal O_V)}
\lesssim_{\psi,q,\sigma,V,R_0}
2^{m(\sigma+k(q-1))}.
\]
Combining the last two estimates proves the corollary.
\end{proof}

\section{Energy of projections orbit relative dimensions}
\label{Section Peres schlag}
\subsection{$L^2$ theory}

In this section we discuss some spherical averages  needed in the proof of
Theorem~\ref{thm:main}. To motivate our discussion, let  us first recall an estimate of Peres and Schlag \cite[Proposition 2.2]{Peres200Schlag} regarding line projections. Let \(\nu\in\mathcal P(\mathbb R^d)\) be compactly supported.  Then,
\begin{equation} \label{eq energy id}
I_t(\nu)
\asymp_{d,t}
\int_{\mathbb R^d}
|\widehat\nu(\xi)|^2|\xi|^{t-d}\,d\xi \asymp_{d,t}
\int_0^\infty R^{t-1}
\int_{S^{d-1}}|\mathcal{F}_{R \omega} (\nu)|^2\,d\sigma(\omega)\,dR.
\end{equation}
See \cite[Theorem 3.10]{Mattila2015Fourier} for the first equality; also, 
note the use of polar coordinates in the second equality, where we denote the normalized surface measure on \(S^{d-1}\) by  \(\sigma =\sigma_d \). Using \eqref{eq energy id}, Peres and Schlag noted the following:
\begin{lemma}[{\cite[Proposition~2.2]{Peres200Schlag}}]
\label{lem:peres-schlag-line}
Let $\nu\in\mathcal P(\mathbb R^d)$ be compactly supported, and let
$0<t<d$. Then
\[
\int_{S^{d-1}}
\int_{\mathbb R}
|R|^{t-1}
\left|\mathcal F_R(\pi_\omega\nu)\right|^2\,dR\,d\sigma(\omega)
\asymp_{d,t}
I_t(\nu).
\]
\end{lemma}

The estimate in Lemma~\ref{lem:peres-schlag-line} is expressed in terms of
the spherical average \newline
$
\int_{S^{d-1}}
|\mathcal F_{R\omega}(\nu)|^2\,d\sigma(\omega).$
However, our  argument  yields equidistribution only
along the rotational orbit of the prescribed subspace $V$. This motivates our orbit-relative analogue, the notion of $\dim_2 ^V$. Let us recall it.  For \(G\leq SO(d)\)  a closed subgroup and \(V\in G_{d,k}\),  write as usual
$\mathcal O_V:=G\cdot V.$
Let \(m_V\) be the normalized \(G\)-invariant probability measure on
\(\mathcal O_V\).  For \(W\in G_{d,k}\), write the sphere
$S(W):=\{\theta\in W:|\theta|=1\},$
and let \(\sigma_W\) be the normalized surface measure on \(S(W)\).  For
\(\nu\in\mathcal P(\mathbb R^d)\), recall  the spherical average
\[
\mathcal S_{\nu,V}(R)
:=
\int_{\mathcal O_V}
\int_{S(W)}
|\mathcal F_{R\theta}(\nu)|^2\,d\sigma_W(\theta)\,dm_V(W),
        \qquad R>0.
\]
And recall that 
$\dim_2^V\nu
:=
\sup\left\{
0\leq t\leq d:
\int_1^\infty R^{t-1}\mathcal S_{\nu,V}(R)\,dR<\infty
\right\}.$

When $G=\operatorname{SO}(d)$, this notion agrees with the usual $L^2$-dimension:
\begin{lemma} \label{lem:relative-dimension-full-SO}
If \(G=SO(d)\) then for every $V\in G_{d,k}$ we have $\dim_2 \nu = \dim_2^V\nu$.
\end{lemma}
\begin{proof}
Put \(G=SO(d)\) and let $V\in G_{d,k}$. Then \(\mathcal O_V = G_{d,k}\), and the measure on \(S^{d-1}\) obtained by first choosing
\(W\in G_{d,k}\) according to Haar measure and then choosing
\(\theta\in S(W)\) according to \(\sigma_W\) is the normalized surface measure
\(\sigma\) on \(S^{d-1}\). Hence
\[
\mathcal S_{\nu,V}(R)
=
\int_{S^{d-1}}|\mathcal F_{R\omega}(\nu)|^2\,d\sigma(\omega).
\]
For every $0<t<d$
\eqref{eq energy id} gives
$I_t(\nu)
\asymp_{d,t}
\int_0^\infty
R^{t-1}\mathcal S_{\nu,V}(R)\,dR.$ 
Moreover, since $|\mathcal F_\xi(\nu)|\leq1$, the contribution of
$0<R<1$ is finite for every $t>0$. Hence
\[
I_t(\nu)<\infty
\quad\Longleftrightarrow\quad
\int_1^\infty
R^{t-1}\mathcal S_{\nu,V}(R)\,dR<\infty,
\qquad 0<t<d.
\]
Taking suprema over $t$ proves that
$\dim_2^V\nu=\dim_2\nu.$
\end{proof}

\subsection{Besov theory}
Let us briefly recall the Besov version of the relative \(L^q\)-dimension
introduced before Theorem \ref{thm:main}.  Fix \(1<q\leq2\), and let \(\psi\) be the admissible (Schwartz) radial
Littlewood-Paley function chosen in Section~\ref{subsec:besov-orbit-estimates}.
Thus, for every \(W\in G_{d,k}\) and every \(m\geq1\), the kernel
\(\psi_m^W\) is defined by
$\widehat{\psi_m^W}(\xi)=\widehat\psi(2^{-m}\xi),
\, \xi\in W$, is adapted to frequencies \(|\xi|\asymp2^m\), and satisfies \eqref{eq: LP partition}.

For \(S\geq k\) and \(\nu\in\mathcal P(\mathbb R^d)\), we defined
\[
\mathcal E^V_{q,S}(\nu)
:=
\sum_{m\geq1}
2^{m(q-1)(S-k)}
\int_{\mathcal O_V}
\left\|(\pi_W\nu)*\psi_m^W\right\|_{L^q(W)}^q
\,dm_V(W).
\]
Thus, 
$\dim_q^V\nu
:=
\sup\left\{
S\in(k,d]:
\mathcal E^V_{q,S'}(\nu)<\infty
\text{ for every }k<S'<S
\right\},$
with the convention that \(\dim_q^V\nu=k\) if the set on the
right-hand side is empty.  The low-frequency part of the
Littlewood-Paley decomposition is not included in \(\mathcal E^V_{q,S}\),
since it is harmless for finite measures and does not affect the dimension.
The value of \(\dim_q^V\nu\) is independent of the particular admissible
Littlewood--Paley partition, by the standard equivalence of Besov norms.

Let us now prove that, for \(q=2\), this definition agrees with the
spherical-average definition of \(\dim_2^V\nu\) used above.

\begin{lemma}
\label{lem:besov-spherical-l2-equivalence}
Let \(1\leq k<d\), be fixed. 
For \(S>k\), 
\[
\sum_{m\geq1}
2^{m(S-k)}
\int_{\mathcal O_V}
\left\|(\pi_W\nu)*\psi_m^W\right\|_{L^2(W)}^2
\,dm_V(W)
<\infty
\iff 
\int_1^\infty
R^{S-1}\mathcal S_{\nu,V}(R)\,dR<\infty.
\]
Consequently, the Besov definition of \(\dim_2^V\nu\) agrees with the spherical-average definition.
\end{lemma}

\begin{proof}
By Plancherel on \(W\), and since
$\mathcal{F}_\xi(\pi_W\nu)=\mathcal{F}_\xi(\nu),
\, \xi\in W,$
we have
\[
\left\|(\pi_W\nu)*\psi_m^W\right\|_{L^2(W)}^2
=
c_k
\int_W
\left|\mathcal{F}_\xi(\nu)\right|^2
\left|\widehat\psi(2^{-m}|\xi|)\right|^2
\,d\xi,
\]
where \(c_k>0\) depends only on the Fourier normalization. Passing to polar coordinates in \(W\), this becomes
\[
\left\|(\pi_W\nu)*\psi_m^W\right\|_{L^2(W)}^2
=
c_k'
\int_0^\infty
r^{k-1}
\left|\widehat\psi(2^{-m}r)\right|^2
\int_{S(W)}
\left|\mathcal{F}_{r\theta}(\nu)\right|^2
\,d\sigma_W(\theta)
\,dr .
\]
Integrating over \(W\in\mathcal O_V\), and using the definition of
\(\mathcal S_{\nu,V}\), gives
\[
\int_{\mathcal O_V}
\left\|(\pi_W\nu)*\psi_m^W\right\|_{L^2(W)}^2
\,dm_V(W)
=
c_k'
\int_0^\infty
r^{k-1}
\left|\widehat\psi(2^{-m}r)\right|^2
\mathcal S_{\nu,V}(r)
\,dr .
\]
Therefore
\[
\begin{aligned}
&\sum_{m\geq1}
2^{m(S-k)}
\int_{\mathcal O_V}
\left\|(\pi_W\nu)*\psi_m^W\right\|_{L^2(W)}^2
\,dm_V(W)
\\
&\qquad =
c_k'
\int_0^\infty
r^{k-1}
\left(
\sum_{m\geq1}
2^{m(S-k)}
\left|\widehat\psi(2^{-m}r)\right|^2
\right)
\mathcal S_{\nu,V}(r)\,dr .
\end{aligned}
\]
Since \(\widehat\psi\) is supported in a fixed dyadic annulus, for each
\(r\geq1\) only \(O_\psi(1)\) values of \(m\) contribute to $
\sum_{m\geq1}
2^{m(S-k)}
\left|\widehat\psi(2^{-m}r)\right|^2$, and for each
such \(m\) we have \(2^m\asymp_\psi r\). Moreover, for all sufficiently
large \(r\), \eqref{eq: LP partition} and the Cauchy--Schwarz inequality
give
$\sum_{m\geq1}
\left|\widehat\psi(2^{-m}r)\right|^2
\asymp_\psi 1.$
Consequently,
\[
\sum_{m\geq1}
2^{m(S-k)}
\left|\widehat\psi(2^{-m}r)\right|^2
\asymp_{\psi,S}
r^{S-k}.
\]
The remaining bounded range of \(r\) is harmless. 

It follows that the preceding sum is finite if and only if
$$\int_1^\infty
r^{k-1}r^{S-k}
\mathcal S_{\nu,V}(r)\,dr
=
\int_1^\infty
r^{S-1}\mathcal S_{\nu,V}(r)\,dr$$
is finite. The contribution of \(0<r<1\) is harmless as explained above, since
$\left|\mathcal{F}_\xi(\nu)\right|\leq1$
for all \(\xi\), and only finitely many Littlewood-Paley pieces can contribute
there. This proves the equivalence.

Finally, applying the equivalence with \(S'\) in place of \(S\), and taking the
corresponding suprema over \(S'\), shows that the Besov and spherical-average
definitions of \(\dim_2^V\nu\) coincide.
\end{proof}
When deducing Theorem \ref{thm: baby case} from Theorem \ref{thm:main}, we  need to compare the orbit-relative
\(L^q\)-dimension with the usual ambient \(L^q\)-dimension; we record this
comparison in the case where the orbit is the full Grassmannian.

We first require an equivalent definition of $\dim_q \nu$. Let \(\Psi\) be a smooth radial Schwartz function on \(\mathbb R^d\) such that
\(\widehat\Psi\in C_c^\infty(\mathbb R^d\setminus{0})\), and such that
$\widehat\Psi(\xi)=1
\,\text{ whenever }
\widehat\psi(|\xi|)\neq0.$
For \(m\geq1\), define $\Psi_m$ via
$\widehat{\Psi_m}(\xi)=\widehat\Psi(2^{-m}\xi),
\, \xi\in\mathbb R^d .$ 

\begin{lemma}
\label{lem:lp-usual-lq-dimension}
Let \(1<q<\infty\), and let \(\nu\in\mathcal P(\mathbb R^d)\) be compactly
supported.  Then
$$
\sup\left(
\{0\}\cup
\left\{
S\in(0,d):
\sum_{m\geq1}
2^{m(q-1)(S-d)}
\left\|\nu*\Psi_m\right\|_{L^q(\mathbb R^d)}^q
<\infty
\right\}
\right)
=
\liminf_{m\to\infty}
\frac{-\log\sum_{Q\in\mathcal D_m}\nu(Q)^q}
{(q-1)m\log2},
$$
where \(\mathcal D_m\) denotes the dyadic cubes of side \(2^{-m}\).
\end{lemma}
In particular, in this case $\dim_q\nu$, which is usually defined as the RHS of the equation in the Lemma, equals the LHS of the equation.
\begin{proof}
Choose \(\varphi\in C_c^\infty(\mathbb R^d)\), \(\varphi\geq0\),
such that \(\varphi\geq 1\) on \(B(0,\sqrt d)\), and write
\(\varphi_m(x)=2^{md}\varphi(2^m x)\). Then
\begin{equation} \label{eq: one} 
\|\nu*\varphi_m\|_{L^q(\mathbb R^d)}^q
\asymp
2^{md(q-1)}\sum_{Q\in\mathcal D_m}\nu(Q)^q,
\end{equation}
with constants independent of \(m\). Indeed, if \(x,y\in Q\in\mathcal D_m\),
then \(2^m(x-y)\in B(0,\sqrt d)\), whence
$\nu*\varphi_m(x)\geq 2^{md}\nu(Q),\, x\in Q,$
and summing over \(Q\) gives the lower bound in \eqref{eq: one}. 
For the upper bound, suppose that
\(\operatorname{spt}\varphi\subset B(0,R)\). For
\(Q,Q'\in\mathcal D_m\), write \(Q'\sim Q\) if
$Q'\cap\bigl(Q+B(0,R2^{-m})\bigr)\neq\varnothing.$
For every \(Q\), there are \(O_{R,d}(1)\) cubes \(Q'\sim Q\), and
conversely every \(Q'\) is related to \(O_{R,d}(1)\) cubes \(Q\).
If \(x\in Q\), then, putting $C=\lVert \varphi \rVert_{L^\infty (\mathbb{R}^d)}$,
\[
\nu*\varphi_m(x)
\leq
C2^{md}\nu\bigl(B(x,R2^{-m})\bigr)
\leq
C2^{md}\sum_{Q'\sim Q}\nu(Q').
\]
Therefore, by convexity,
\[
\begin{aligned}
\|\nu*\varphi_m\|_{L^q}^q
&\leq
C2^{mdq}\sum_{Q\in\mathcal D_m}|Q|
\left(\sum_{Q'\sim Q}\nu(Q')\right)^q\\
&\leq
C2^{md(q-1)}
\sum_{Q\in\mathcal D_m}\sum_{Q'\sim Q}\nu(Q')^q\\
&\leq
C2^{md(q-1)}
\sum_{Q'\in\mathcal D_m}\nu(Q')^q.
\end{aligned}
\]
This proves \eqref{eq: one}.

Put $D:=\liminf_{m\to\infty}
\frac{-\log \sum_{Q\in\mathcal D_m}\nu(Q)^q }{(q-1)m\log 2}.$ 
The root test shows that
$$\sum_{m\geq1}2^{m(q-1)S} \sum_{Q\in\mathcal D_m}\nu(Q)^q$$
converges whenever \(S<D\) and diverges whenever \(S>D\). Hence
\[
D=
\sup\left(
\{0\}\cup
\left\{
S>0:
\sum_{m\geq1}2^{m(q-1)S}\sum_{Q\in\mathcal D_m}\nu(Q)^q<\infty
\right\}
\right).
\]
It is standard that \(D\leq d\). Combining this with \eqref{eq: one}, we obtain
\begin{equation}
\label{eq: three}
D=
\sup\left(
\{0\}\cup
\left\{
S\in(0,d]:
\sum_{m\geq1}
2^{m(q-1)(S-d)}
\|\nu*\varphi_m\|_{L^q}^q<\infty
\right\}
\right).
\end{equation}

It remains to replace the low-frequency mollifications by annular
Littlewood--Paley pieces. Fix \(S<d\), and set
$\alpha:=\frac{(q-1)(d-S)}{q}>0.$
Then the series in \eqref{eq: three} is
\[
\sum_{m\geq1}
2^{-mq\alpha}\|\nu*\varphi_m\|_{L^q}^q.
\]
Notice that
$\widehat\varphi(0)=\int_{\mathbb R^d}\varphi(x)\,dx>0,$
so \(\widehat\varphi\) is nonzero in a neighbourhood of the origin.
Thus \((\varphi_m)\) is an admissible family of local means for spaces of
negative smoothness. Moreover, the condition
\(\widehat\Psi=1\) on the support of the fixed Littlewood--Paley
annulus implies that \((\Psi_m)\) is an admissible annular family.
The equivalent local-means and Littlewood--Paley characterizations of \(B^{-\alpha}_{q,q}(\mathbb R^d)\) \cite[Sections~2.3.1 and~2.4.6]{triebel1983theory} therefore give
\[
\sum_{m\geq1}
2^{-mq\alpha}\|\nu*\varphi_m\|_{L^q}^q<\infty
\quad\Longleftrightarrow\quad
\sum_{m\geq1}
2^{-mq\alpha}\|\nu*\Psi_m\|_{L^q}^q<\infty.
\]
The omitted fixed low-frequency terms are finite by Young's inequality.
Since
$q\alpha=(q-1)(d-S),$
the latter series is
\[
\sum_{m\geq1}
2^{m(q-1)(S-d)}
\|\nu*\Psi_m\|_{L^q}^q.
\]
Thus the two conditions have the same threshold for every \(S<d\).
Taking the supremum over \(S\in[0,d)\), with the convention above at
dimension zero, proves the claim. The endpoint \(S=d\) does not affect
the value of the supremum.
\end{proof}

We can now finally show that when $G$ is full, $\dim_q^V\nu\geq \dim_q\nu.$
\begin{lemma}
\label{lem:full-orbit-energy-to-relative-lq} Assume that \(G=\operatorname{SO}(d)\)
and let \(1<q\leq 2\) and \(1\leq k<d\), and let
\(\nu\in\mathcal P(\mathbb R^d)\) be compactly supported.
Then, for every \(V\in G_{d,k}\) and every \(S\geq k\),
\[
\mathcal E^V_{q,S}(\nu)
\leq
C_{\psi,q,d,k}
\sum_{m\geq 1}
2^{m(q-1)(S-d)}
\|\nu*\Psi_m\|_{L^q(\mathbb R^d)}^q .
\]
Consequently,
$\dim_q^V\nu\geq \dim_q\nu .$
\end{lemma}

\begin{proof}
Since \(G=\operatorname{SO}(d)\), we have
$\mathcal O_V=G_{d,k},$
and \(m_V\) is the normalized invariant measure on \(G_{d,k}\).
For \(m\geq 1\) and \(f\in\mathcal S(\mathbb R^d)\), define
$T_m f(W,x)
=
(\pi_W f)*\psi_m^W(x),
\, W\in G_{d,k},\ x\in W,$  
where
$\pi_Wf(x)
:=
\int_{W^\perp}f(x+y)\,d\operatorname{Leb}_{W^\perp}(y).$
We first prove two endpoint estimates for \(T_m\).

An \(L^1\) estimate follows from Fubini and from the uniform
\(L^1\)-boundedness of the kernels \(\psi_m^W\):
\[
\begin{aligned}
\int_{G_{d,k}}
\|T_m f(W,\cdot)\|_{L^1(W)}\,dm_V(W)
&\leq
\int_{G_{d,k}}
\|\pi_W f\|_{L^1(W)}
\|\psi_m^W\|_{L^1(W)}
\,dm_V(W)  \\
&\leq
C_\psi \|f\|_{L^1(\mathbb R^d)} .
\end{aligned}
\]

For an \(L^2\) estimate, since \(\psi\) is radial, the notation \(\widehat\psi(2^{-m}\xi)\), for \(\xi\in W\), is independent of the choice of an orthogonal identification \(W\simeq\mathbb R^k\). Now, Plancherel on each \(W\) gives
\[
\begin{aligned}
\int_{G_{d,k}}
\|T_m f(W,\cdot)\|_{L^2(W)}^2\,dm_V(W)
&=
c_k
\int_{G_{d,k}}\int_W
|\widehat f(\xi)|^2
|\widehat\psi(2^{-m} \xi)|^2
\,d\xi\,dm_V(W).
\end{aligned}
\]
By the argument as in the proof of Lemma~\ref{lem:relative-dimension-full-SO}, followed by polar coordinates in \(W\) and in \(\mathbb R^d\), we have
\[
\int_{G_{d,k}}\int_W h(\xi)\,d\xi\,dm_V(W)
=
C_{d,k}
\int_{\mathbb R^d} h(\xi)|\xi|^{k-d}\,d\xi ,
\]
and therefore
$\begin{aligned}
\int_{G_{d,k}}
\|T_m f(W,\cdot)\|_{L^2(W)}^2\,dm_V(W)
&=
C_{d,k}
\int_{\mathbb R^d}
|\widehat f(\xi)|^2
|\widehat\psi(2^{-m} \xi )|^2
|\xi|^{k-d}\,d\xi .
\end{aligned}$
Since \(\widehat\psi(2^{-m}\xi )\) is supported where
\(|\xi|\asymp_\psi 2^m\), we get
\[
\int_{G_{d,k}}
\|T_m f(W,\cdot)\|_{L^2(W)}^2\,dm_V(W)
\leq
C_{\psi,d,k}
2^{-m(d-k)}
\|f\|_{L^2(\mathbb R^d)}^2 .
\]

Regard \(T_m\) as an operator into the \(L^p\)-space of the incidence bundle 
$$ \mathcal I_{d,k} := \left\{ (W,x):W\in G_{d,k},\ x\in W \right\},$$
equipped with the measure \(dm_V(W)\,d\operatorname{Leb}_W(x).\) The preceding estimates provide  \(L^1\) and \(L^2\) operator bounds for \( T_m:L^p(\mathbb R^d)\mapsto L^p(\mathcal I_{d,k}). \) Interpolating these \(L^1\) and \(L^2\) estimates gives, for every
\(1<q\leq 2\),
\[
\int_{G_{d,k}}
\|T_m f(W,\cdot)\|_{L^q(W)}^q\,dm_V(W)
\leq
C_{\psi,q,d,k}
2^{-m(d-k)(q-1)}
\|f\|_{L^q(\mathbb R^d)}^q .
\]

We now apply this estimate with
$f=\nu*\Psi_m .$
By the choice of \(\Psi\), for every \(W\in G_{d,k}\),
$(\pi_W\nu)*\psi_m^W
=
(\pi_W(\nu*\Psi_m))*\psi_m^W .$
Indeed, taking Fourier transforms on \(W\), both sides have transform
$\xi\mapsto
\mathcal F_\xi(\nu)\widehat\psi(2^{-m}\xi),
\, \xi\in W,$
because
$\widehat\Psi(2^{-m}\xi)=1
\text{ on the support of }
\widehat\psi(2^{-m}\xi).$
Therefore,
\[
\int_{G_{d,k}}
\left\|(\pi_W\nu)*\psi_m^W\right\|_{L^q(W)}^q
\,dm_V(W)
\leq
C_{\psi,q,d,k}
2^{-m(d-k)(q-1)}
\|\nu*\Psi_m\|_{L^q(\mathbb R^d)}^q .
\]
Multiplying by \(2^{m(q-1)(S-k)}\) and summing over \(m\geq 1\), we obtain
\[
\begin{aligned}
\mathcal E^V_{q,S}(\nu)
&\leq
C_{\psi,q,d,k}
\sum_{m\geq 1}
2^{m(q-1)(S-k)}
2^{-m(d-k)(q-1)}
\|\nu*\Psi_m\|_{L^q(\mathbb R^d)}^q  \\
&=
C_{\psi,q,d,k}
\sum_{m\geq 1}
2^{m(q-1)(S-d)}
\|\nu*\Psi_m\|_{L^q(\mathbb R^d)}^q .
\end{aligned}
\]
This proves the quantitative estimate.

Now let \(S<\dim_q\nu\). If \(S\leq k\), there is nothing to prove. If
\(S>k\), then for every \(k<S'<S\), the preceding estimate gives
$\mathcal E^V_{q,S'}(\nu)<\infty .$
Hence \(S\leq \dim_q^V\nu\). Taking the supremum over
\(S<\dim_q\nu\) gives
$\dim_q^V\nu\geq \dim_q\nu.$
\end{proof}
\section{Proof of Theorem~\ref{thm:main}}
\label{Section:proof-main}
\noindent{\textbf{Recalling the notations, hypotheses,  and goal}} We now prove Theorem~\ref{thm:main}.  We use the notation introduced in
Sections~\ref{Section model}-\ref{subsec:besov-orbit-estimates}.  Thus
$\Phi=\{f_i(x)=r_iO_i x+t_i\}_{i=1}^n$ 
is a self-similar IFS on \(\mathbb R^d\), \(d\geq3\), with probability vector
\(\mathbf p=(p_1,\ldots,p_n)\), and \(\nu=\nu_{\mathbf p}\) is the corresponding
self-similar measure.  Fix \(1\leq k<d\) and \(V\in G_{d,k}\).  Let
\[
G:=\overline{\langle O_1,\ldots,O_n\rangle}\subseteq SO(d),
\qquad
\mathcal O_V:=G\cdot V,
\]
and let \(m_V\) be the normalized \(G\)-invariant probability measure on
\(\mathcal O_V\).

Recall that
$\{a^{(1)},\ldots,a^{(m)}\}
=
\{|r_i|:1\leq i\leq n\}$
is the set of distinct contraction moduli, and that
$\mathcal I_j:=\{i: |r_i|=a^{(j)}\},
\,
\widetilde p_j:=\sum_{i\in\mathcal I_j}p_i.$
We write
\[
\lambda_j:=-\log a^{(j)},
\qquad
\lambda_{\max}:=\max_{1\leq j\leq m}\lambda_j
=
\max_i(-\log |r_i|).
\]
For each \(j\), define
$\rho_j
:=
\sum_{i\in\mathcal I_j}
\frac{p_i}{\widetilde p_j}\delta_{O_i}
\in\mathcal P(G).$
Thus, \(P_{\rho_j,V}\) is the operator \(P_{a^{(j)},V}\) in the statement of
Theorem~\ref{thm:main}.

Set
\[
\Theta_{V}:=-\log \left( \sum_{a\in\mathcal A}
\widetilde p_a
\left\|P_{a,V}\right\|_{L^2_0(\mathcal O_V,\,m_V)} \right).
\]
Let \(1<q\leq2\), and write
\(q'=q/(q-1)\). Suppose that
$k<b<S<\dim_q^V\nu.$
Assume 
\[
1<q<2
\,\text{ and }\quad
\frac12\dim\mathcal O_V<\sigma<q, \text{ or } 
q=2
\quad\text{and}\quad
\frac12\dim\mathcal O_V<\sigma.
\]

The main hypothesis of Theorem~\ref{thm:main} is that
\begin{equation}
\label{eq:theta-main-hypothesis}
\Theta_{V}
>
\left(
b(q-1)+
\frac{b(\sigma+b(q-1))}{S-b}
\right)\lambda_{\max}.
\end{equation}
We aim to show that \(\pi_V\nu\ll\operatorname{Leb}_V\), and its density belongs to
$B^{(b-k)/q'}_{q,q}(V),$ which was defined in Section \ref{subsec:besov-orbit-estimates}.

\medskip

\noindent{ \textbf{Parameter choice and first reduction}}
We may choose \(\alpha>0\) such that
\begin{equation}
\label{eq:alpha-main-term-condition}
\alpha(S-b)>b\lambda_{\max}
\end{equation}
and
\begin{equation}
\label{eq:alpha-error-condition}
b(q-1)\lambda_{\max}
+
(\sigma+b(q-1))\alpha
<
\Theta_{V}.
\end{equation}
Indeed, \eqref{eq:theta-main-hypothesis} is precisely the condition that the following
interval is non-empty:
\[
\left(
\frac{b\lambda_{\max}}{S-b},
\,
\frac{\Theta_{V}-b(q-1)\lambda_{\max}}
{\sigma+b(q-1)}
\right).
\]

Let
$\beta:=\frac{b-k}{q'}=\frac{(b-k)(q-1)}{q}.$
By Lemma~\ref{lem:besov-reconstruction-measure}, it is enough to prove
\begin{equation}
\label{eq:need-besov-sum-main}
\sum_{N\geq1}
2^{Nq\beta}
\left\|
(\pi_V\nu)*\psi_N^V
\right\|_{L^q(V)}^q
<\infty .
\end{equation}
Since \(q\beta=(q-1)(b-k)\), the weight in
\eqref{eq:need-besov-sum-main} can be written as
$2^{N(q-1)(b-k)}.$
The contribution of bounded \(N\)'s is finite by Young's inequality, so it suffices to consider only large \(N\).

\medskip

\noindent{ \textbf{A scale decomposition}} We decompose the high frequency scales into bounded length blocks.  For
\(n\geq1\), put
$L_n:=n\lambda_{\max}$
and define
\[
\mathcal N_n
:=
\left\{
N\in\mathbb N:
e^{L_n+\alpha n}\leq 2^N<e^{L_{n+1}+\alpha(n+1)}
\right\}.
\]
Equivalently,
$\mathcal N_n
=
\mathbb N\cap
\left[
\frac{n(\lambda_{\max}+\alpha)}{\log 2},
\frac{(n+1)(\lambda_{\max}+\alpha)}{\log 2}
\right).$
Hence the sets \(\mathcal N_n\), \(n\geq1\), are pairwise disjoint and
$\bigcup_{n\geq1}\mathcal N_n
=
\left\{
N\in\mathbb N:
2^N\geq e^{\lambda_{\max}+\alpha}
\right\}.$ 
Moreover, for every \(n\geq1\),
\begin{equation}
\label{eq:block-cardinality-bound}
\#\mathcal N_n
\leq
\left\lceil
\frac{\lambda_{\max}+\alpha}{\log 2}
\right\rceil .
\end{equation}
For \(N\in\mathcal N_n\), let
\(m=m(N,n)\) be the integer determined by
$2^m\leq 2^N e^{-L_n}<2^{m+1}.$
Then
\begin{equation}
\label{eq:residual-scale-comparison}
2^m\asymp_{\alpha,\lambda_{\max}} e^{\alpha n}.
\end{equation}
\medskip

\noindent{\textbf{A first application of Jensen}} We now use the model decomposition from Lemma~\ref{lem:stopping-model}
at the scale \(L=L_n\). Recalling \eqref{eq: omega n}, for \(\eta\in\Omega_{L_n}\), write
$\mu_\eta
:=
\sum_{u\in X^{(\eta)}_{|\eta|}}
p^{(\eta)}(u)\,f_u\nu.$
Thus Lemma~\ref{lem:stopping-model} gives
$\nu
=
\int_{\Omega_{L_n}}\mu_\eta\,dQ(\eta),$
where, as usual, the integral over \(\Omega_{L_n}\) denotes the finite sum
with weights \(Q([\eta])\).  Since projection and convolution are linear, for all $N$,
\[
(\pi_V\nu)*\psi_N^V
=
\int_{\Omega_{L_n}}
\sum_{u\in X^{(\eta)}_{|\eta|}}
p^{(\eta)}(u)\,
(\pi_V f_u\nu)*\psi_N^V
\,dQ(\eta).
\]
Applying Jensen's inequality, first with respect to \(\eta\) and then with
respect to \(u\), we obtain
\begin{equation} \label{eq:first jensen}
\begin{aligned}
\left\|
(\pi_V\nu)*\psi_N^V
\right\|_{L^q(V)}^q
&\leq
\int_{\Omega_{L_n}}
\left\|
\sum_{u\in X^{(\eta)}_{|\eta|}}
p^{(\eta)}(u)\,
(\pi_V f_u\nu)*\psi_N^V
\right\|_{L^q(V)}^q
\,dQ(\eta)
\\
&\leq
\int_{\Omega_{L_n}}
\sum_{u\in X^{(\eta)}_{|\eta|}}
p^{(\eta)}(u)
\left\|
(\pi_V f_u\nu)*\psi_N^V
\right\|_{L^q(V)}^q
\,dQ(\eta).
\end{aligned}
\end{equation}
\medskip

\noindent{\textbf{Reducing to the main term and error term}}
Writing \(f_u(x)=r_uO_ux+t_u\), by \eqref{eq: eps eta L relative},
for every \(u\in X^{(\eta)}_{|\eta|}\),
$|r_u|
=
e^{-L_n-\varepsilon_{\eta,L_n}},$
where
\begin{equation}
\label{eq:overshoot-main-proof}
0\leq \varepsilon_{\eta,L_n}\leq\lambda_{\max}.
\end{equation}
We next rewrite the terms
\(\|(\pi_V f_u\nu)*\psi_N^V\|_{L^q(V)}^q\) in coordinates adapted to the
subspace \(O_u^{-1}V\).  For every
\(x\in\mathbb R^d\),
\[
O_u^{-1}\pi_V f_u(x)
=
r_u\pi_{O_u^{-1}V}x+O_u^{-1}\pi_Vt_u .
\]
The map \(O_u^{-1}:V\to O_u^{-1}V\) is an isometry, and therefore preserves
\(L^q\)-norms.  Moreover, by the equivariant choice of the
Littlewood-Paley kernels,
$(O_u^{-1})_*\psi_N^V=\psi_N^{O_u^{-1}V}.$
It follows that, up to a translation, which does not affect the
\(L^q\)-norm,
\[
\left\|
(\pi_V f_u\nu)*\psi_N^V
\right\|_{L^q(V)}^q
=
\left\|
\bigl((x\mapsto r_u x)_*\pi_{O_u^{-1}V}\nu\bigr)
*
\psi_N^{O_u^{-1}V}
\right\|_{L^q(O_u^{-1}V)}^q .
\]
The dilation \(x\mapsto r_u x\) contributes the factor
\(|r_u|^{-k(q-1)}\) to the \(L^q\)-norm and moves the Littlewood--Paley
frequency scale from \(2^N\) to the scale \(2^N|r_u|\).

  Put
$W:=O_u^{-1}V
\text{ and }
s_{u,N}:=2^N|r_u|$.
Let \(K_{u,N}^W\) be the kernel on \(W\) defined by
$$\widehat{K_{u,N}^W}(\xi)
=
\widehat\psi(s_{u,N}^{-1}\xi),
\, \xi\in W.$$
Since \(\psi\) is radial, a change of variables under the similarity
\(x\mapsto r_uO_ux+t_u\) gives
\begin{equation}
\label{eq:rescaled-lp-kernel}
\left\|
(\pi_V f_u\nu)*\psi_N^V
\right\|_{L^q(V)}^q
=
|r_u|^{-k(q-1)}
\left\|
(\pi_W\nu)*K_{u,N}^W
\right\|_{L^q(W)}^q.
\end{equation}
Moreover, by the definition of \(m=m(N,n)\) as in \eqref{eq:residual-scale-comparison} and the bound
\(0\leq\varepsilon_{u,L_n}\leq\lambda_{\max}\),
$e^{-\lambda_{\max}}2^m
\leq
s_{u,N}
<
2^{m+1}.$
Consequently, the  support of \( \widehat{K_{u,N}^W}\) meets only those
Littlewood-Paley annuli whose indices differ from \(m\) by a uniformly
bounded constant.

Thus there exists \(J_\Phi\geq1\), depending only on
\(\lambda_{\max}\), such that, for all sufficiently large \(n\),
$K_{u,N}^W
=
\sum_{|j|\leq J_\Phi}
K_{u,N}^W*\psi_{m+j}^W.$
The finitely many excluded values of \(n\) contribute only a bounded
amount. Since
$\|K_{u,N}^W\|_{L^1(W)}
=
\|\psi\|_{L^1(\mathbb R^k)}$
uniformly in \(u,N\), Young's inequality and
$(\sum_{j=1}^M a_j)^q\leq M^{q-1}\sum_{j=1}^M a_j^q$
give
\[
\left\|
(\pi_W\nu)*K_{u,N}^W
\right\|_{L^q(W)}^q
\leq
C_{\Phi,\psi,q}
\sum_{|j|\leq J_\Phi}
\left\|
(\pi_W\nu)*\psi_{m+j}^W
\right\|_{L^q(W)}^q.
\]
Combining this with \eqref{eq:rescaled-lp-kernel}, we obtain
\[
\left\|
(\pi_V f_u\nu)*\psi_N^V
\right\|_{L^q(V)}^q
\leq
C_{\Phi,\psi,q}
|r_u|^{-k(q-1)}
\sum_{|j|\leq J_\Phi}
\left\|
(\pi_{O_u^{-1}V}\nu)*\psi_{m+j}^{O_u^{-1}V}
\right\|_{L^q(O_u^{-1}V)}^q .
\]

Since \(|r_u|^{-1}\leq e^{L_n+\lambda_{\max}}\), the factor
\(e^{k(q-1)\lambda_{\max}}\) may be absorbed into the constant.  Plugging the
preceding estimate into \eqref{eq:first jensen}, we obtain
\begin{equation}
\label{eq:rescaled-cylinder-bound}
\begin{aligned}
\left\|
(\pi_V\nu)*\psi_N^V
\right\|_{L^q(V)}^q
&\leq
C_{\Phi,\psi,q}
e^{k(q-1)L_n}
\sum_{|j|\leq J_\Phi}
\int_{\Omega_{L_n}}
\sum_{u\in X^{(\eta)}_{|\eta|}}
p^{(\eta)}(u)
\\
&\qquad\qquad\qquad\cdot
\left\|
(\pi_{O_u^{-1}V}\nu)*\psi_{m+j}^{O_u^{-1}V}
\right\|_{L^q(O_u^{-1}V)}^q
\,dQ(\eta).
\end{aligned}
\end{equation}
Since \(2^{m(N,n)}\asymp e^{\alpha n}\), there is \(n_0\) such that
\(m(N,n)+j\geq1\) for every \(n\geq n_0\), \(N\in\mathcal N_n\), and
\(|j|\leq J_\Phi\). The contribution of \(n<n_0\) involves only finitely many
dyadic scales \(N\), and is finite by Young's inequality. Hence, in what follows
we assume \(n\geq n_0\).

We now apply Corollary~\ref{cor:besov-stopped-cylinder-mixing}, with
\(\mu=\nu\) and \(W=V\), to the inner sum in
\eqref{eq:rescaled-cylinder-bound}.  Let \(R_0\) be such that
\(\operatorname{spt}\nu\subset B(0,R_0)\).  For each fixed \(j\), put
\(\ell=m+j\).  Then
\[
\begin{aligned}
&\sum_{u\in X^{(\eta)}_{|\eta|}}
p^{(\eta)}(u)
\left\|
(\pi_{O_u^{-1}V}\nu)*\psi_{\ell}^{O_u^{-1}V}
\right\|_{L^q(O_u^{-1}V)}^q
\\
&\qquad\leq
\int_{\mathcal O_V}
\left\|
(\pi_U\nu)*\psi_{\ell}^{U}
\right\|_{L^q(U)}^q
\,dm_V(U)
\\
&\qquad\quad+
C_{\psi,q,\sigma,V,R_0}
2^{\ell(\sigma+k(q-1))}
\prod_{a=1}^{|\eta|}
\left\|
P_{\rho_{\eta_a},V}
\right\|_{L^2_0(\mathcal O_V,\,m_V)} .
\end{aligned}
\]
Inserting this estimate into \eqref{eq:rescaled-cylinder-bound}
therefore splits the Besov sum \eqref{eq:need-besov-sum-main} into a main term (with the spherical average), 
and an error term (that deals with the spectral norms). 
\medskip

\noindent{\textbf{The main term}}
We first consider the contribution of the orbit-average term. Recall that  by \eqref{eq:residual-scale-comparison},
$2^{m(N,n)}\geq \frac12 e^{\alpha n}$
uniformly in \(N\in\mathcal N_n\).  Hence there exists \(n_0\), depending only
on \(\alpha\) and \(J_\Phi\), such that
$m(N,n)+j\geq1
\text{ for all }n\geq n_0,\ N\in\mathcal N_n,\ |j|\leq J_\Phi.$
The contribution of the finitely many indices \(n<n_0\) is finite and may be
absorbed into the constant.  Thus, in estimating the main term, we also assume
\(m(N,n)+j\geq1\) for all \(|j|\leq J_\Phi\). Thus, the main term  is bounded by
\[
\begin{aligned}
\mathcal M
&:=
\sum_{n\geq1}
\sum_{N\in\mathcal N_n}
2^{N(q-1)(b-k)}
e^{k(q-1)L_n}
\sum_{|j|\leq J_\Phi}
\int_{\mathcal O_V}
\left\|
(\pi_U\nu)*\psi_{m(N,n)+j}^{U}
\right\|_{L^q(U)}^q
\,dm_V(U).
\end{aligned}
\]
For \(N\in\mathcal N_n\), we have
\(2^N\asymp e^{L_n+\alpha n}\), and by
\eqref{eq:residual-scale-comparison},
\(2^{m(N,n)}\asymp e^{\alpha n}\).  Since \(L_n=n\lambda_{\max}\), it follows
that
\[
2^{N(q-1)(b-k)}
e^{k(q-1)L_n}
\asymp
\exp\left(
(q-1)n\bigl(b\lambda_{\max}+\alpha(b-k)\bigr)
\right).
\]
On the other hand, since \(j\) ranges over a fixed finite set,
\[
2^{(m(N,n)+j)(q-1)(S-k)}
\asymp
\exp\left(
\alpha n(q-1)(S-k)
\right),
\]
with constants independent of \(n,N\), and \(j\).  Therefore
\[
\begin{aligned}
&
2^{N(q-1)(b-k)}
e^{k(q-1)L_n}
\int_{\mathcal O_V}
\left\|
(\pi_U\nu)*\psi_{m(N,n)+j}^{U}
\right\|_{L^q(U)}^q
\,dm_V(U)
\\
&\qquad\lesssim
\exp\left(
(q-1)n\bigl(b\lambda_{\max}-\alpha(S-b)\bigr)
\right)
2^{(m(N,n)+j)(q-1)(S-k)}
\\
&\qquad\qquad\cdot
\int_{\mathcal O_V}
\left\|
(\pi_U\nu)*\psi_{m(N,n)+j}^{U}
\right\|_{L^q(U)}^q
\,dm_V(U).
\end{aligned}
\]
By \eqref{eq:alpha-main-term-condition}, the exponential factor is uniformly
bounded in \(n\).  Since the sets \(\mathcal N_n\) have uniformly bounded
cardinality and \(j\) ranges over a fixed finite set, each integer
\(\ell\geq1\) can occur as \(\ell=m(N,n)+j\) only boundedly many times (this also uses \eqref{eq:residual-scale-comparison}).
Consequently,
\[
\begin{aligned}
\mathcal M
&\lesssim
\sum_{\ell\geq1}
2^{\ell(q-1)(S-k)}
\int_{\mathcal O_V}
\left\|
(\pi_U\nu)*\psi_\ell^{U}
\right\|_{L^q(U)}^q
\,dm_V(U)
<\infty,
\end{aligned}
\]
where the final inequality follows from the standing assumption \(S<\dim_q^V\nu\).

\medskip

\noindent{\textbf{The error term}}
It remains to estimate the contribution of the mixing error.  For
\(1\leq i\leq m\), put
$B_i
:=
\left\|
P_{\rho_i,V}
\right\|_{L^2_0(\mathcal O_V,\,m_V)}.$ 
Each \(P_{\rho_i,V}\) is an average of pullbacks by isometries of
\(\mathcal O_V\), and hence \(B_i\leq1\).  Moreover, by the definition of
\(\Theta_{V}\),
$\sum_{i=1}^m \widetilde p_i B_i
=
\exp(-\Theta_{V}).$

Since \(L_n=n\lambda_{\max}\)
and \(\lambda_i\leq\lambda_{\max}\) for every \(i\), every
\(\eta\in\Omega_{L_n}\) has \(|\eta|\geq n\).  Therefore, using also
\(0\leq B_i\leq1\),
$\prod_{\ell=1}^{|\eta|}
B_{\eta_\ell}
\leq
\prod_{\ell=1}^{n}
B_{\eta_\ell}.$
Thus, if \(\omega\in\Omega\) is an infinite sequence whose stopped
prefix at level \(L_n\) is \(\eta\), then the right-hand side is the product
over the first \(n\) symbols of \(\omega\).  Hence, by independence of the
coordinates under \(Q\),
\begin{equation}
\label{eq:product-expectation-main-proof}
\begin{aligned}
\int_{\Omega_{L_n}}
\prod_{\ell=1}^{|\eta|}
B_{\eta_\ell}
\,dQ(\eta)
&\leq
\int_{\Omega}
\prod_{\ell=1}^{n}
B_{\omega_\ell}
\,dQ(\omega)
\\
&=
\left(
\sum_{i=1}^m \widetilde p_i B_i
\right)^n
=
e^{-n\Theta_{V}} .
\end{aligned}
\end{equation}

We now return to the error term coming from
\eqref{eq:rescaled-cylinder-bound}.  Using
\eqref{eq:product-expectation-main-proof}, this contribution is bounded, up to
a multiplicative constant, by
\[
\begin{aligned}
\mathcal R
&:=
\sum_{n\geq1}
\sum_{N\in\mathcal N_n}
2^{N(q-1)(b-k)}
e^{k(q-1)L_n}
\sum_{|h|\leq J_\Phi}
2^{(m(N,n)+h)(\sigma+k(q-1))}
e^{-n\Theta_{V}}.
\end{aligned}
\]
We estimate this exactly as in the main term.  For \(N\in\mathcal N_n\),
we have \(2^N\asymp e^{L_n+\alpha n}\), and by
\eqref{eq:residual-scale-comparison},
\(2^{m(N,n)}\asymp e^{\alpha n}\).  Since \(L_n=n\lambda_{\max}\), and since
\(\mathcal N_n\) and the range \(|h|\leq J_\Phi\) have uniformly bounded
cardinalities, it follows that
\[
\begin{aligned}
\mathcal R
&\lesssim
\sum_{n\geq1}
\exp\left(
(q-1)n\bigl(b\lambda_{\max}+\alpha(b-k)\bigr)
\right)
\exp\left(
\alpha n(\sigma+k(q-1))
\right)
e^{-n\Theta_{V}}
\\
&=
\sum_{n\geq1}
\exp\left(
n\left[
b(q-1)\lambda_{\max}
+
(\sigma+b(q-1))\alpha
-
\Theta_{V}
\right]
\right).
\end{aligned}
\]
By \eqref{eq:alpha-error-condition}, the exponent is negative.  Thus
\(\mathcal R<\infty\), completing the estimate of the error term.

\medskip

Combining the estimates for \(\mathcal M\) and \(\mathcal R\) proves
\[
\sum_{N\geq1}
2^{N(q-1)(b-k)}
\left\|
(\pi_V\nu)*\psi_N^V
\right\|_{L^q(V)}^q
<\infty .
\]
Since \(q\beta=(q-1)(b-k)\), this is exactly
$\sum_{N\geq1}
2^{Nq\beta}
\left\|
(\pi_V\nu)*\psi_N^V
\right\|_{L^q(V)}^q
<\infty,
\,
\beta=\frac{b-k}{q'}.$
Lemma~\ref{lem:besov-reconstruction-measure} now implies that
\(\pi_V\nu\ll\operatorname{Leb}_V\), and that its density belongs to
\(B^\beta_{q,q}(V)=B^{(b-k)/q'}_{q,q}(V)\).  In particular, the density belongs
to \(L^q(V)\).  This completes the proof of Theorem~\ref{thm:main}.

\section{Remaining proofs} \label{Section proof of coro}
\subsection{Proof of Corollary~\ref{cor:explicit-singular-smooth-projections}} \label{subsection: cor explicit singular smooth}
Recall that 
$$\Phi:=
\left\{
f_{q,b}(x)=
18^{-2/5}O_{q,b}x+q+2b\mathbf{1}
\right\}_{
 q\in\{0,1\}^{3}:1\leq \lVert q \rVert_1 \leq 2, \, b\in\{-1,0,1\} }$$ where
$\mathbf{1}:=(1,1,1),$ and
$(q,b)\mapsto O_{q,b},$
is 
 some bijection onto the Ramanujan set
\(\mathcal{R}_{17}\subset \operatorname{SO}(3)\).
Let \(\nu \in \mathcal{P}(\mathbb{R}^3)\) be the uniform self-similar measure with respect to $\Phi$. 

We show that $\dim\nu=\frac{5}{2}$; but,
nevertheless, for every \(V\in G_{3,1}\),
$\pi_{V}\nu\ll \operatorname{Leb}_{V}.$

Let us recall some arithmetic properties of Ramanujan sets. Define 
\begin{equation} \label{eq: omega p}
\Omega_p
:=
\left\{
\frac{1}{\sqrt p}(a,b,c,d)\in S^3:
a,b,c,d\in\mathbb Z,\ 
a^2+b^2+c^2+d^2=p,\ 
a>0 \text{ odd},\ 
b,c,d \text{ even}
\right\}.
\end{equation}
This is the standard set of \(SU(2)\)-representatives whose image under
the adjoint map \(SU(2)\to SO(3)\) is the  Ramanujan set
\(\mathcal R_p\). 
For
$\omega \in \Omega_{17}$, put 
$$\omega'=\frac{1}{\sqrt{17}}(a+bi+cj+dk), \text{ with } a^2+b^2+c^2+d^2=17.$$
We identify \(\mathbb R^3\) with the space of purely imaginary quaternions
\[
\operatorname{Im}\mathbb H
=
\{x_1 i+x_2 j+x_3 k:\ x_1,x_2,x_3\in\mathbb R\}.
\]
For a unit quaternion \(\omega'\), let
\[
O_{\omega'}=\operatorname{Ad}_{\omega'}:
\operatorname{Im}\mathbb H\to\operatorname{Im}\mathbb H,
\qquad
\operatorname{Ad}_{\omega'}(x)=\omega' x(\omega')^{-1}.
\]
With respect to the basis \((i,j,k)\) of \(\operatorname{Im}\mathbb H\), this rotation is given by
\[
O_{\omega'}
=
\frac{1}{17}
\begin{pmatrix}
a^2+b^2-c^2-d^2 & 2(bc-ad) & 2(bd+ac)\\
2(bc+ad) & a^2-b^2+c^2-d^2 & 2(cd-ab)\\
2(bd-ac) & 2(cd+ab) & a^2-b^2-c^2+d^2
\end{pmatrix}.
\]
For the standard LPS representatives, $a$ is odd and $b,c,d$
are even. Since $17\equiv 1\pmod 2$, the displayed formula gives
\begin{equation}\label{eq: properties of LPS}
O_{\omega'}\equiv I\mod 2,\, \text{ and } O_{\omega'} ^TO_{\omega'}=I \mod 3 \text{ so } O_{\omega'} \mod 3 \text{ is orthogonal}.
\end{equation}
Indeed, for the second identity, \(O_{\omega'}^TO_{\omega'}=I\) over \(\mathbb Q\), and \(17\) is
invertible modulo \(3\).

\begin{proof}
Put
$$\beta:=18^{1/5},
\,
r:=\beta^{-2}=18^{-2/5},
\,
t_{q,b}:=q+2b\mathbf{1}.$$
We first show that \(\Phi\) has no exact overlaps. Observe that the map
$$(q,b)\mapsto
\left(t_{q,b}\bmod 2,\ |t_{q,b}|^2\bmod 3\right)$$
is injective. Indeed,
$t_{q,b}=q+2b\mathbf 1\equiv q\pmod 2,$
so the first component determines \(q\). Writing
\(m:=|q|_1\in\{1,2\}\), we also have
$$|t_{q,b}|^2
=m+4bm+12b^2
\equiv m(1+b)\pmod 3.$$
Since \(m\not\equiv0\pmod3\) and \(1+b\) takes the three distinct
values \(0,1,2\), the second component then determines \(b\).

Now suppose that \(f_u=f_v\). Equality of the contraction ratios gives
\(n:=|u|=|v|\). Comparing the translation parts and multiplying by
\(\beta^{2(n-1)}\), we obtain
\[
\sum_{k=1}^{n}
\beta^{2(n-k)}
O_{u_1}\cdots O_{u_{k-1}}t_{u_k}
=
\sum_{k=1}^{n}
\beta^{2(n-k)}
O_{v_1}\cdots O_{v_{k-1}}t_{v_k}.
\]
After this multiplication all coefficients belong to
\(\mathbb Z[1/17,\beta]\). Reduce it modulo the ideal \((2,\beta)\). All terms with \(k<n\) vanish, while every product of rotations is congruent to \(I\) modulo \(2\), by \eqref{eq: properties of LPS}. Hence $t_{u_n}\equiv t_{v_n}\pmod 2.$ Next reduce the same identity modulo \((3,\beta)\). Again only the terms with \(k=n\) remain. Since the reductions modulo \(3\) of the products of rotations are orthogonal, we obtain $|t_{u_n}|^2\equiv |t_{v_n}|^2\pmod 3.$ By the injectivity observation above, \(u_n=v_n\). Cancelling the common final map and proceeding inductively, we conclude that \(u=v\). Therefore \(\Phi\) has no exact overlaps.

The IFS \(\Phi\) is algebraic. Moreover, since \(\mathcal R_{17}\) generates a dense subgroup of \(\operatorname{SO}(3)\), the orthogonal parts of \(\Phi\) act linearly irreducibly on \(\mathbb R^3\). Since \(\Phi\) has no exact overlaps, \cite[Corollary~1.7]{hochman2015Rd} gives \[
\dim\nu
=
\min\left\{
3,\frac{\log 18}{-\log r}
\right\}
=
\frac{5}{2}.
\]

By \eqref{eq: ramanujan}
$\left\|P\right\|_{L^{2}_{0}(\operatorname{SO}(3))}
=
\frac{2\sqrt{17}}{18}.$ Furthermore, $\log\frac{18}{2\sqrt{17}} -\frac{4}{15}\log18 >0.$ Applying
Theorem~\ref{thm: baby case} Part (1)  with \(k=1\), therefore yields
$\pi_{V}\nu\ll\operatorname{Leb}_{V}$
for every \(V\in G_{3,1}\).
\end{proof}

\subsection{Proof of Corollary~\ref{cor:singular-ssc-all-line-projections}}
\label{subsec:proof-singular-ssc-all-line-projections}

Recall the definition of $\Omega_p$ from \eqref{eq: omega p} In particular \(|\Omega_p|=p+1\).
Writing \(\omega=\frac{1}{\sqrt p}(a,b,c,d) \in \Omega_p \), let \(A_\omega\) denote left multiplication by \(\omega\) on
\(\mathbb H\simeq\mathbb R^4\). Thus,
\begin{equation} \label{eq: A omega}
A_\omega
= \frac{1}{\sqrt p}
\begin{pmatrix}
a & -b & -c & -d\\
b & a & -d & c\\
c & d & a & -b\\
d & -c & b & a
\end{pmatrix}
\in \mathrm{SO}(4).
\end{equation}
Enumerate \(\Omega_p=\{\omega_1,\ldots,\omega_{p+1}\}\), and let
\(M=\lceil (p+1)^{1/4}\rceil\). Write 
$$\text{ for } 1\leq j\leq p+1,\, \, j-1=n_1(j)+M n_2(j)+M^2 n_3(j)+M^3 n_4(j),
\, 0\leq n_i(j)\leq M-1,$$ and
put
$t_j=(n_1(j),n_2(j),n_3(j),n_4(j)).$ For $D>1$ define
\[
\Phi_{p,D}
=
\left\{
x\mapsto (p+1)^{-1/D}\cdot A_{\omega_j}x+t_j
\right\}_{j=1}^{p+1}.
\]

We show that for every  $\frac{3+\sqrt{21}}{2}<D<4$  and  prime
\(p\equiv 1 \text{ mod } 4\) such that  
$$p>
\max\left\{
5^{\frac{4D}{4-D}},\,
2^{\frac{4D(D-1)}{D^2-3D-3}}
\right\},$$ 
the uniform self-similar measure $\nu \in \mathcal{P}(\mathbb{R}^4)$ with respect to $\Phi_{p,D}$ satisfies: 
$$\dim_H\nu=\dim_2\nu=D<4, \text{ and for every } V\in G_{4,1},\, \pi_V\nu\ll\operatorname{Leb}_V, \text{ with }
\pi_V\nu\in L^2(V).$$

\begin{proof}
Fix
$\frac{3+\sqrt{21}}{2}<D<4.$ 
Choose \(\sigma>3/2\) sufficiently close to \(3/2\) so that
\begin{equation}
\label{eq:ssc-R4-parameter}
\frac{D+\sigma}{D(D-1)}<\frac{1}{2}.
\end{equation}

Let \(p\equiv 1 \pmod 4\) be a prime.
Recall that the  Ramanujan set
\(\mathcal R_p\subseteq \operatorname{SO}(3)\) has \(p+1\) elements and
that the associated normalized averaging operator satisfies
\[
\left\|P_p\right\|_{L^2_0(\operatorname{SO}(3))}
=
\frac{2\sqrt p}{p+1}.
\]

Set \(r:=(p+1)^{-1/D}\). It is direct to check that
the first lower bound on \(p\) in the statement, namely
$p>5^{\frac{4D}{4-D}},$
ensures that the IFS $\Phi:= \Phi_{p,D}$
has the SSC. Indeed, write $K=K_\Phi$. Since every coordinate of \(t_j\) lies in \(\{0,\ldots,M-1\}\), we have
$|t_j|\leq 2(M-1),
\text{ and hence }
\sup_{x\in K}|x|
\leq
\frac{2(M-1)}{1-r}.$
Moreover, the vectors \(t_j\) are distinct integer vectors, so
\(|t_i-t_j|\geq1\) whenever \(i\neq j\). Therefore,
\[
\begin{aligned}
\operatorname{dist}\bigl(f_i(K),f_j(K)\bigr)
&=
\inf_{x,y\in K}
\left|
(t_i-t_j)+rA_i x-rA_j y
\right| \\
&\geq
|t_i-t_j|
-r\sup_{x\in K}|x|
-r\sup_{y\in K}|y| \\
&\geq
1-\frac{4r(M-1)}{1-r}.
\end{aligned}
\]
Now, writing $N=p+1$,
$$ M-1<N^{1/4} \text{ and so } 
4r(M-1)+r
<
4N^{-\frac{4-D}{4D}}+N^{-1/D}
<
5N^{-\frac{4-D}{4D}}.$$
The assumption
$p>5^{\frac{4D}{4-D}}$
implies that the last expression is smaller than \(1\). Hence
$$4r(M-1)<1-r,$$ so the images \(f_i(K)\) are pairwise disjoint.
Thus \(\Phi\) satisfies the SSC.

Let \(\nu\) be the self-similar measure associated to \(\Phi\) and the uniform
probability vector. By the strong separation condition,
\[
\dim_H\nu=\dim_2\nu
=
\frac{\log N}{-\log r}
=
D.
\]
In particular, since \(D<4\), we have
$\nu\perp\operatorname{Leb}_{\mathbb R^4}.$ 

We next verify the relevant hypotheses of Theorem~\ref{thm:main} for every
\(V\in G_{4,1}\). Let
$G
:=
\overline{
\left\langle
A_\omega: \omega \in\Omega_p
\right\rangle
}
\subseteq\operatorname{SO}(4).$
Since the image of \(\Omega_p\) in
\(\operatorname{SO}(3)\) has spectral gap, it generates a dense subgroup
of \(\operatorname{SO}(3)\). It follows that
$G=\{L_q:q\in\operatorname{SU}(2)\},$ where we identify \(\operatorname{SU}(2)\) with the group of unit
quaternions, and where \(L_q:\mathbb H\to\mathbb H\) denotes left
multiplication, \(L_q(x)=qx\).   
This group acts transitively on \(G_{4,1}\). Indeed, if \(u,v\in\mathbb H\)
are unit quaternions, then
$L_{vu^{-1}}(\mathbb Ru)=\mathbb Rv.$ 
Thus, for every \(V\in G_{4,1}\),
$\mathcal O_V=G_{4,1}$ and
$\dim\mathcal O_V=3.$

Fix \(V\in G_{4,1}\), and choose a unit quaternion \(u\in\mathbb H\) such that
\(V=\mathbb Ru\).  The map
$\Theta_u:SU(2)/\{\pm1\}\to G_{4,1},
\,
\Theta_u([q])=\mathbb R(qu),$
is a well-defined diffeomorphism.  Under the standard identification
\(SU(2)/\{\pm1\}\simeq SO(3)\), given by \([q]\mapsto \operatorname{Ad}_q\),
this identifies \(G_{4,1}\) with \(SO(3)\), and identifies the invariant
measure \(m_V\) with Haar measure on \(SO(3)\). Moreover, for \(\omega\in\Omega_p\),
\[
L_\omega^{-1}\Theta_u([q])
=
\mathbb R(\omega^{-1}qu)
=
\Theta_u([\omega^{-1}q]).
\]
Thus, if \(F\in L^2(G_{4,1},m_V)\) and
$H:=F\circ\Theta_u\), then
$$(P_VF)\circ\Theta_u([q])
=
\frac1{p+1}\sum_{\omega\in\Omega_p}
H([\omega^{-1}q]).$$
Under the quotient map \(SU(2)\to SU(2)/\{\pm1\}\simeq SO(3)\), the elements
\([\omega]\) are precisely \(\operatorname{Ad}_\omega\in\mathcal R_p\).
Therefore the operator \(P_V\) is unitarily conjugate to the  averaging
operator
\[
P_pH(g)=\frac1{p+1}\sum_{O\in\mathcal R_p}H(O^{-1}g)
\]
on \(L^2(SO(3))\). Hence
\[
\|P_V\|_{L^2_0(\mathcal O_V,m_V)}
=
\|P_p\|_{L^2_0(SO(3))}
=
\frac{2\sqrt p}{p+1}.
\]

Since \(\mathcal O_V=G_{4,1}\), the relative spherical average defining
\(\dim_2^V\nu\) is the ordinary spherical average on \(S^3\). Therefore, by an argument similar to Lemma \ref{lem:relative-dimension-full-SO},
$\dim_2^V\nu=\dim_2\nu=D
\text{ for every }V\in G_{4,1}.$ 

We now apply Theorem~\ref{thm: pre main} with \(k=s=1\). It is enough to verify
that
\[
-\log\left(\frac{2\sqrt p}{p+1}\right)
>
\frac{D+\sigma}{D-1}\cdot(-\log r).
\]
Since
$-\log r=\frac{\log(p+1)}{D},$
this is equivalent to
\[
\frac{-\log\left(2\sqrt p/(p+1)\right)}{\log(p+1)}
>
\frac{D+\sigma}{D(D-1)}.
\]
The left-hand side tends to \(1/2\) as \(p\rightarrow\infty\),
while the right-hand side is strictly smaller than \(1/2\) by
\eqref{eq:ssc-R4-parameter}. Therefore the required inequality holds
for all sufficiently large primes \(p\equiv1\pmod4\).
Theorem~\ref{thm: pre main} now gives
\[
\pi_V\nu\ll\operatorname{Leb}_V,
\qquad
\pi_V\nu\in L^2(V)
\qquad\text{for every }V\in G_{4,1}.
\]
This completes the proof. \end{proof}

A few more comments are in order. First, about the lower bound on \(p\) in the statement of
the corollary. The first condition
$p>5^{\frac{4D}{4-D}}$
was already used above to guarantee the SSC. The second condition
$p>
2^{\frac{4D(D-1)}{D^2-3D-3}}
$
is used to verify the spectral inequality . Indeed, since
$\frac{D+\frac32}{D(D-1)}<\frac12$
is equivalent to \(D^2-3D-3>0\), this bound gives enough room to choose
\(\sigma>3/2\) sufficiently close to \(3/2\) so that
$\frac{-\log\left(2\sqrt p/(p+1)\right)}{\log(p+1)}
>
\frac{D+\sigma}{D(D-1)}.$
Thus the explicit hypothesis on \(p\) in the statement implies both the strong
separation condition and the spectral condition required by
Theorem~\ref{thm:main}.

Finally, let \(0\leq t<1/30\), and put
$s:=1+2t.$
Then
$s<\frac{16}{15}.$
Since
\[
\lim_{\substack{D\uparrow4\\ \sigma\downarrow3/2}}
\frac{s(D+\sigma)}{D(D-s)}
=
\frac{11s}{8(4-s)}
<
\frac{1}{2},
\]
we may choose \(D<4\) sufficiently close to \(4\) and
\(\sigma>3/2\) sufficiently close to \(3/2\) so that
\[
\frac{s(D+\sigma)}{D(D-s)}<\frac{1}{2}.
\]
Repeating the above construction and applying Theorem~\ref{thm: pre main} with this value of \(s\), we obtain a singular
uniformly contracting self-similar measure satisfying the strong
separation condition such that
\[
\pi_V\nu\in H^{(s-1)/2}(V)=H^t(V)
\qquad\text{for every }V\in G_{4,1}.
\]

\subsection{Proof of Corollary \ref{cor:main}}
\label{subsec:proof-explicit-examples}
We first prove a uniform spherical-average contraction estimate for the
self-similar measures appearing in Corollary~\ref{cor:main}. This estimate
will also be used later in the
proof of Corollary~\ref{cor:explicit-ambient-ac}.

Let \(p\equiv1\pmod 4\) be prime, and let
\(\mathcal R_p\subseteq \operatorname{SO}(3)\) be the corresponding
Ramanujan set.  Thus \(|\mathcal R_p|=p+1\). In this section we denote the averaging operator
$$T_pF(g)
=
\frac1{p+1}\sum_{O\in\mathcal R_p}F(O^{-1}g),
        \, F\in L^2(\operatorname{SO}(3)),$$ which
satisfies, by \eqref{eq: ramanujan} and Lemma \ref{lem:ambient-gap-controls-orbit-sobolev},
$$\|T_p\|_{L^2_0(\operatorname{SO}(3))}
\leq 
\frac{2\sqrt p}{p+1}.$$
We write
$\rho_p:=\frac{2\sqrt p}{p+1}.$

For \(0<r<1\), recall that
\[
\Phi_r
=
\left\{
f_{O,b,c}(x,y)=\bigl(rOx+be_1,\, ry+c\bigr)
\right\}_{O\in\mathcal R_5,\ b,c\in\{0,1\}} .
\]   For \(0<\varepsilon<1/10\), define
\[
w_{00}=\frac12-\frac{\varepsilon}{4},\qquad
w_{01}=\frac{\varepsilon}{4},\qquad
w_{10}=\frac12-\frac{3\varepsilon}{4},\qquad
w_{11}=\frac{3\varepsilon}{4}.
\]
Let \(\mu_{r,\varepsilon}\in\mathcal P(\mathbb R^4)\) be the self-similar
measure with respect to \(\Phi_r\) and the probability vector
$p_{O,b,c}:=\frac{w_{bc}}{6}, O\in R_5,\ b,c\in\{0,1\}.$ Note that here \[
G=
\left\{
\begin{pmatrix}
O&0\\
0&1
\end{pmatrix}:O\in \operatorname{SO}(3)
\right\}
\simeq \operatorname{SO}(3).
\]

\begin{lemma}
\label{lem:mask-estimate} Fix \(0<r<1\) and $0<\varepsilon<\frac{1}{10}$,  and let
\(\mu_{r,\varepsilon}\) be the self-similar measure  above. Identifying \(\mathbb R^4=\mathbb R^3\times\mathbb R\),
for \(0<c\leq1\), \(\varrho\in[c,1]\), \(z\in\mathbb R\), and \(R>0\), 
define
$$H_R^{\varrho,z}(\omega) := \mathcal F_{R(\varrho\omega,z)} \bigl(\mu_{r,\varepsilon}\bigr), \, \omega\in S^2.$$
Then, for every \(\vartheta>
\kappa_5
:=
\frac{19+\sqrt{91}}{36}\), there exists
\(R_0=R_0(c,\vartheta)\) such that for every \(R\geq R_0\), every \(\varrho\in[c,1]\), and every
\(z\in\mathbb R\),
\[
\int_{S^2}
\left|H_R^{\varrho,z}(\omega)\right|^2
\,d\sigma(\omega)
\leq
\vartheta
\int_{S^2}
\left|H_{rR}^{\varrho,z}(\omega)\right|^2
\,d\sigma(\omega)
\]
\end{lemma}

\begin{proof}
For \(u\in\mathbb R\), put
$$a_0(u)
:=
w_{00}+w_{01}e^{-2\pi iu},
\,
a_1(u)
:=
w_{10}+w_{11}e^{-2\pi iu}.$$
By the definition of the weights,
$|a_0(u)|
\leq
w_{00}+w_{01}
=
\frac12,
\qquad
|a_1(u)|
\leq
w_{10}+w_{11}
=
\frac12.$

Let \(T_5\) denote the averaging operator induced by
\(\mathcal R_5\) on \(L^2(S^2)\). It is direct to check that, by self-similarity,
\[
H_R^{\varrho,z}
=
m_R^{\varrho,z}\cdot T_5H_{rR}^{\varrho,z},
\]
where
$m_R^{\varrho,z}(\omega)
=
a_0(Rz)
+
a_1(Rz)e^{-2\pi iR\varrho\omega_1}.$
Recall that
$\left\|T_5\right\|_{L^2_0(S^2)}
\leq
\rho_5
:=
\frac{\sqrt5}{3}.$

We first record two estimates for the multiplier, uniform in
\(\varrho\in[c,1]\) and \(z\in\mathbb R\). For every \(\delta>0\),
if \(R\) is sufficiently large, then
\begin{equation}
\label{eq:uniform-mask-mean}
\int_{S^2}
\left|m_R^{\varrho,z}\right|^2
\,d\sigma
\leq
\frac12+\delta,
\end{equation}
and
\begin{equation}
\label{eq:uniform-mask-centered}
\left\|
\left|m_R^{\varrho,z}\right|^2
-
\int_{S^2}
\left|m_R^{\varrho,z}\right|^2
\,d\sigma
\right\|_{L^2(S^2)}
\leq
\frac{1}{2\sqrt2}+\delta.
\end{equation}
Indeed, writing
$a
:=
a_0(Rz)\overline{a_1(Rz)},$
we have
\[
\left|m_R^{\varrho,z}(\omega)\right|^2
=
|a_0(Rz)|^2+|a_1(Rz)|^2
+
2\operatorname{Re}
\left(
ae^{2\pi iR\varrho\omega_1}
\right),
\]
where
$|a_0(Rz)|^2+|a_1(Rz)|^2
\leq
\frac12,
\,
|a|
\leq
\frac14.$
Since
\[
\int_{S^2}
e^{2\pi it\omega_1}
\,d\sigma(\omega)
=
\frac{\sin(2\pi t)}{2\pi t},
\]
the oscillatory terms tend to zero uniformly for
\(t=R\varrho\geq cR\). This proves
\eqref{eq:uniform-mask-mean}. Moreover,
\[
\int_{S^2}
\left|
2\operatorname{Re}
\left(
ae^{2\pi it\omega_1}
\right)
\right|^2
\,d\sigma(\omega)
=
2|a|^2+o_{t\to\infty}(1)
\leq
\frac18+o_{t\to\infty}(1),
\]
uniformly for \(t\geq cR\). Since subtracting the mean can only decrease
the \(L^2\)-norm, this proves
\eqref{eq:uniform-mask-centered}.

Now decompose
\[
H_{rR}^{\varrho,z}
=
c_R^{\varrho,z}
+
\varphi_R^{\varrho,z},
\qquad
c_R^{\varrho,z}
:=
\int_{S^2}
H_{rR}^{\varrho,z}
\,d\sigma,
\qquad
\int_{S^2}
\varphi_R^{\varrho,z}
\,d\sigma
=
0,
\]
and put
$g_R^{\varrho,z}
:=
T_5\varphi_R^{\varrho,z}.$
Then
\[
\int_{S^2}
g_R^{\varrho,z}
\,d\sigma
=
0,
\qquad
\left\|g_R^{\varrho,z}\right\|_2
\leq
\rho_5
\left\|\varphi_R^{\varrho,z}\right\|_2.
\]
Also,
\[
\left\|m_R^{\varrho,z}\right\|_\infty
\leq
|a_0(Rz)|+|a_1(Rz)|
\leq
1.
\]

Expanding the square, we obtain
\[
\begin{aligned}
\int_{S^2}
\left|H_R^{\varrho,z}\right|^2
\,d\sigma
&=
M_R^{\varrho,z}
\left|c_R^{\varrho,z}\right|^2
\\
&\quad+
2\operatorname{Re}
\left(
\overline{c_R^{\varrho,z}}
\int_{S^2}
\left|m_R^{\varrho,z}\right|^2
g_R^{\varrho,z}
\,d\sigma
\right)
\\
&\quad+
\int_{S^2}
\left|m_R^{\varrho,z}\right|^2
\left|g_R^{\varrho,z}\right|^2
\,d\sigma,
\end{aligned}
\]
where
$M_R^{\varrho,z}
:=
\int_{S^2}
\left|m_R^{\varrho,z}\right|^2
\,d\sigma.$
Since
$\int_{S^2}
g_R^{\varrho,z}
\,d\sigma
=
0,$
we have
\[
\int_{S^2}
\left|m_R^{\varrho,z}\right|^2
g_R^{\varrho,z}
\,d\sigma
=
\int_{S^2}
\left(
\left|m_R^{\varrho,z}\right|^2
-
M_R^{\varrho,z}
\right)
g_R^{\varrho,z}
\,d\sigma.
\]
Therefore, by Cauchy--Schwarz,
\eqref{eq:uniform-mask-mean},
\eqref{eq:uniform-mask-centered}, and the spectral bound for \(T_5\),
for all sufficiently large \(R\),
\[
\begin{aligned}
\int_{S^2}
\left|H_R^{\varrho,z}\right|^2
\,d\sigma
&\leq
\left(\frac12+\delta\right)
\left|c_R^{\varrho,z}\right|^2
\\
&\quad+
2\left(\frac{1}{2\sqrt2}+\delta\right)
\rho_5
\left|c_R^{\varrho,z}\right|
\left\|\varphi_R^{\varrho,z}\right\|_2
\\
&\quad+
\rho_5^2
\left\|\varphi_R^{\varrho,z}\right\|_2^2.
\end{aligned}
\]

As \(\delta\downarrow0\), the corresponding quadratic form converges
to the one represented by
\[
\begin{pmatrix}
\frac12 & \frac{\rho_5}{2\sqrt2}\\
\frac{\rho_5}{2\sqrt2} & \rho_5^2
\end{pmatrix}
=
\begin{pmatrix}
\frac12 & \frac{\sqrt{10}}{12}\\
\frac{\sqrt{10}}{12} & \frac59
\end{pmatrix}.
\]
Its largest eigenvalue is
$\kappa_5
=
\frac{19+\sqrt{91}}{36}.$

Since
$\left|c_R^{\varrho,z}\right|^2
+
\left\|\varphi_R^{\varrho,z}\right\|_2^2
=
\int_{S^2}
\left|H_{rR}^{\varrho,z}\right|^2
\,d\sigma,$
it follows that, for every \(\vartheta>\kappa_5\), there exists
\(R_0=R_0(c,\vartheta)\) such that
\[
\int_{S^2}
\left|H_R^{\varrho,z}(\omega)\right|^2
\,d\sigma(\omega)
\leq
\vartheta
\int_{S^2}
\left|H_{rR}^{\varrho,z}(\omega)\right|^2
\,d\sigma(\omega)
\]
whenever \(R\geq R_0\), uniformly in
\(\varrho\in[c,1]\) and \(z\in\mathbb R\).
\end{proof}
\medskip

\noindent{\textbf{Proof of Corollary~\ref{cor:main}.}}
Fix \(\eta>0\) and \(\tau>0\). Put
$\tau_0:=\min\{\tau,1/4\},$ and recall that
$\rho_5:=\frac{\sqrt5}{3}.$ 
Choose \(\sigma_0>3/2\). We choose \(0<r<1\) so that
\begin{equation}
\label{eq:choice-r-coupled-dim}
\frac{-\log \left( \frac{19+\sqrt{91}}{36} \right)}{-\log r}>4
\end{equation}
and
\begin{equation}
\label{eq:choice-r-coupled-spectral}
-\log\rho_5
>
\frac{(4-2\tau_0)(4+\sigma_0)}{2\tau_0}
\cdot(-\log r).
\end{equation}
Both conditions hold when \(r\) is sufficiently close to \(1\).

Let \(0<\varepsilon<1/10\), to be chosen below, and let
\(\mu_{r,\varepsilon}\in\mathcal P(\mathbb R^4)\) be the corresponding
self-similar measure. We first choose \(\varepsilon\) so that the ambient Fourier
dimension of \(\mu_{r,\varepsilon}\) is smaller than \(\eta\). Recall that
$P_4:\mathbb R^4=\mathbb R^3\times\mathbb R\to\mathbb R$
denotes the fourth-coordinate projection, and put
\[
\lambda_\varepsilon:=(P_4)_*\mu_{r,\varepsilon}.
\]
By the definition of the weights,
$w_{00}+w_{10}=1-\varepsilon,
\,
w_{01}+w_{11}=\varepsilon.$
Therefore \(\lambda_\varepsilon\) is the self-similar measure for the IFS
$ \lbrace y\mapsto ry,
\qquad
y\mapsto ry+1 \rbrace$
with probability vector \((1-\varepsilon,\varepsilon)\). It is thus well known that we have the upper bound
\[
\dim\lambda_\varepsilon
\leq
\frac{
-(1-\varepsilon)\log(1-\varepsilon)
-\varepsilon\log\varepsilon
}{
-\log r
}.
\]
Since the right-hand side tends to \(0\) as \(\varepsilon\downarrow0\), we may
choose \(0<\varepsilon<1/10\) such that
\begin{equation}
\label{eq:thin-vertical-choice}
\frac{
-(1-\varepsilon)\log(1-\varepsilon)
-\varepsilon\log\varepsilon
}{
-\log r
}
<
\min\{1,\eta\}.
\end{equation}
In particular,
$\lambda_\varepsilon\perp\operatorname{Leb}_{\mathbb R},$ and so $\mu_{r,\varepsilon}\perp\operatorname{Leb}_{\mathbb R^4}.$ Thus, given the structure of the IFS and the measure, it is direct to check that
$$\dim_F\mu_{r,\varepsilon}
\leq
\dim_F\lambda_\varepsilon
\leq
\dim\lambda_\varepsilon
<
\eta.$$

We now prove the positive projection statements. Let
\(1\leq k\leq3\), and let \(W\in G_{4,k}\) satisfy \(e_4\notin W\).
Write
$
P_3:\mathbb R^4=\mathbb R^3\times\mathbb R\to\mathbb R^3,
\,
P_4:\mathbb R^4\to\mathbb R$
for the coordinate projections. Since \(P_3|_W\) is injective,
\[
c_W
:=
\min_{\theta\in S(W)}
|P_3\theta|
>
0.
\]
For \(\theta\in S(W)\), put
$\varrho_\theta
:=
|P_3\theta|,
\,
z_\theta
:=
P_4\theta.$
Averaging over the \(G\)-orbit of \(W\), and using the rotational
invariance of surface measure on \(S^2\), gives
\[
\mathcal S_{\mu_{r,\varepsilon},W}(R)
=
\int_{S(W)}
\int_{S^2}
\left|
H_R^{\varrho_\theta,z_\theta}(\omega)
\right|^2
\,d\sigma(\omega)
\,d\sigma_W(\theta).
\]

Let
$0<\beta<
\frac{-\log\kappa_5}{-\log r}.$
Choose \(\vartheta>\kappa_5\) sufficiently close to \(\kappa_5\) that
$\beta
<
\frac{-\log\vartheta}{-\log r}.$
Since
$\varrho_\theta\in[c_W,1],
\,
z_\theta\in[-1,1],$
Iterating Lemma~\ref{lem:mask-estimate}, uniformly for
\(\varrho\in[c_W,1]\) and \(z\in[-1,1]\), and then integrating in
\(\theta\), gives
\[
\mathcal S_{\mu_{r,\varepsilon},W}(R)
\lesssim_{\beta,W}
R^{-\beta},
\qquad
R\geq1.
\]
By \eqref{eq:choice-r-coupled-dim}, we may choose \(\beta>4\).
It follows directly from the definition of relative \(L^2\)-dimension
that
\begin{equation}
\label{eq:full-relative-dimension-coupled}
\dim_2^W\mu_{r,\varepsilon}
=
4.
\end{equation}

It remains to verify the spectral hypothesis of Theorem~\ref{thm: pre main}.
Write
\[
\widetilde O
:=
\begin{pmatrix}
O&0\\
0&1
\end{pmatrix},
\qquad O\in\mathcal R_5.
\]
All similarities defining \(\mu_{r,\varepsilon}\) have the same contraction
ratio \(r\). Moreover, since
$\sum_{b,c}w_{bc}=1,$
the induced rotational Markov operator (on any Sobolev space) for any orbit \(\mathcal O_W\) is
$P_WF(U)
=
\frac16\sum_{O\in\mathcal R_5}
F(\widetilde O^{-1}U).$
By \eqref{eq: ramanujan} and
Lemma~\ref{lem:ambient-gap-controls-orbit-sobolev},
$\left\|P_W\right\|_{L^2_0(\mathcal O_W)}
\leq
\rho_5.$
Also, for every \(1\leq k\leq3\) and every \(W\in G_{4,k}\) with
\(e_4\notin W\), the orbit \(\mathcal O_W\) has dimension at most \(3\).
Thus our fixed choice \(\sigma_0>3/2\) is admissible for all such \(W\).

Let \(1\leq k\leq3\), let \(W\in G_{4,k}\) satisfy \(e_4\notin W\), and let
\(s\) satisfy
$k\leq s<4-2\tau_0.$
Using \eqref{eq:full-relative-dimension-coupled}, the spectral condition in
Theorem~\ref{thm: pre main} is implied by
\[
-\log\rho_5
>
\frac{s(4+\sigma_0)}{4-s}\cdot(-\log r).
\]
Since the function
$s\mapsto \frac{s(4+\sigma_0)}{4-s}$
is increasing on \([0,4)\), this follows from
\eqref{eq:choice-r-coupled-spectral}. Therefore Theorem~\ref{thm: pre main} gives
\[
\pi_W\mu_{r,\varepsilon}\ll\operatorname{Leb}_W
\text{ and, for every } k\leq s<4-2\tau_0,\,
\pi_W\mu_{r,\varepsilon}\in H^{(s-k)/2}(W).
\]
Equivalently,
$\pi_W\mu_{r,\varepsilon}\in H^t(W)
\text{ for every }0\leq t<\frac{4-k}{2}-\tau_0.$
Since \(\tau_0\leq\tau\), this implies in particular that
\[
\pi_W\mu_{r,\varepsilon}\in H^t(W)
\qquad
\text{for every }0\leq t<\frac{4-k}{2}-\tau.
\]

It remains to verify the converse assertion. Let
\(1\leq k\leq3\) and let \(W\in G_{4,k}\) satisfy \(e_4\in W\). Then
\[
P_4 \circ\pi_W\mu_{r,\varepsilon}
=
P_4 \mu_{r,\varepsilon}
=
\lambda_\varepsilon.
\]
If \(\pi_W\mu_{r,\varepsilon}\) were absolutely continuous with respect to
\(\operatorname{Leb}_W\), then its one-dimensional projection onto
\(\operatorname{span}(e_4)\) would be absolutely continuous with respect to
Lebesgue measure on that line. This contradicts the singularity of
\(\lambda_\varepsilon\). Thus
\[
\pi_W\mu_{r,\varepsilon}\ll\operatorname{Leb}_W
\quad\Longleftrightarrow\quad
e_4\notin W.
\]
This completes the proof.
\subsection{Proof of Claim \ref{Claim: fourier}}
The proof of Claim \ref{Claim: fourier} is a rather direct consequence of the argument proving Theorem \ref{thm:main}.
\label{subsec:proof-fourier-dimension}
\begin{proof} Write
$ \lambda_{\max}:=\max_i(-\log |r_i|).$
It is enough to prove the asserted estimate for \(|\xi|\) sufficiently large,
since \(|\mathcal F_\xi(\nu)|\leq1\).

Fix \(\xi\in\mathbb R^d\) with \(R:=|\xi|\geq2\), and write
$V_\xi:=\operatorname{span}\{\xi\}\in G_{d,1}.$
We will apply Claim~\ref{claim:stopped-cylinder-mixing} with the ambient group
\(\operatorname{SO}(d)\) and the orbit
$\operatorname{SO}(d)\cdot V_\xi=G_{d,1}.$
Since \(\dim G_{d,1}=d-1\), our assumption
\(\sigma>\frac{d-1}{2}\) is precisely the  condition needed
there.

Let \(0<\theta<1\), to be chosen below, and put
$n:=\lfloor \theta\log R\rfloor.$
For \(R\) sufficiently large, \(n\geq1\). Applying Lemma \ref{lem:stopping-model} at scale \(L=n\), Jensen's inequality, Claim~\ref{claim:stopped-cylinder-mixing},
and Lemma~\ref{lem:ambient-gap-controls-orbit-sobolev}, gives
\[
\begin{aligned}
|\mathcal F_\xi(\nu)|^2
&\leq
\int_{\Omega_n}
\mathcal S_{\nu}(e^{-n-\varepsilon_{\eta,n}}R)\,dQ(\eta)
\\
&\quad+
C_{\nu,\sigma}
\left(1+e^{-n}R\right)^\sigma
\int_{\Omega_n}
\prod_{\ell=1}^{|\eta|}
\left\|
P_{\eta_\ell}^G
\right\|_{L^2_0(\operatorname{SO}(d))\to L^2_0(\operatorname{SO}(d))}
\,dQ(\eta),
\end{aligned}
\]
where
$\mathcal S_\nu(T)
:=
\int_{S^{d-1}}
|\mathcal F_{T\omega}(\nu)|^2\,d\sigma(\omega).$
Here \(P_{\eta_\ell}^G\) denotes \(P_a^G\) for the contraction modulus
\(a=a^{(\eta_\ell)}\).

We estimate the two terms separately. First, by the assumed spherical-average
decay and the overshoot bound \(0\leq\varepsilon_{\eta,n}\leq\lambda_{\max}\),
\[
\mathcal S_\nu(e^{-n-\varepsilon_{\eta,n}}R)
\lesssim
(e^{-n-\varepsilon_{\eta,n}}R)^{-\beta}
\lesssim_{\Phi}
e^{\beta n}R^{-\beta}.
\]
Since \(n=\lfloor\theta\log R\rfloor\), this gives
$\int_{\Omega_n}
\mathcal S_{\nu}(e^{-n-\varepsilon_{\eta,n}}R)\,dQ(\eta)
\lesssim_{\Phi}
R^{-\beta(1-\theta)}.$

For the error term, set
$B_a:=
\left\|
P_a^G
\right\|_{L^2_0(\operatorname{SO}(d))\to L^2_0(\operatorname{SO}(d))},
\,
B:=\sum_{a\in\mathcal A}\widetilde p_a B_a.$
Thus \(B=e^{-\Lambda}\). As in the proof of Theorem~\ref{thm:main}, the 
expectation satisfies
\[
\int_{\Omega_n}
\prod_{\ell=1}^{|\eta|}
B_{a^{(\eta_\ell)}}\,dQ(\eta)
\lesssim_\Phi
\exp\left(-\frac{\Lambda}{\lambda_{\max}}\,n\right)
=
e^{-\Gamma n}.
\]
Also,
$1+e^{-n}R\lesssim R^{1-\theta}.$
Therefore the error term is bounded by
\[
C_{\nu,\sigma,\Phi}
R^{\sigma(1-\theta)}e^{-\Gamma n}
\lesssim
R^{\sigma(1-\theta)-\Gamma\theta}.
\]
Combining the two estimates, we obtain
\[
|\mathcal F_\xi(\nu)|^2
\lesssim
R^{-\beta(1-\theta)}
+
R^{\sigma(1-\theta)-\Gamma\theta}.
\]

Now let
$0<\delta<
\frac{\beta\Gamma}{\beta+\sigma+\Gamma}.$ 
Choose \(\theta\in(0,1)\) such that
\[
\delta<
\min\left\{
\beta(1-\theta),\,
\Gamma\theta-\sigma(1-\theta)
\right\}.
\]
This is possible because the maximum of the right-hand minimum is attained at
$\theta=\frac{\beta+\sigma}{\beta+\sigma+\Gamma},$
and equals
$\frac{\beta\Gamma}{\beta+\sigma+\Gamma}.$
For this choice of \(\theta\), the preceding estimate gives
\[
|\mathcal F_\xi(\nu)|^2
\lesssim_\delta
|\xi|^{-\delta},
\qquad |\xi|\geq1.
\]
The  lower bound for
\(\dim_F\nu\) now follows immediately from the definition of Fourier dimension.
\end{proof}
\subsection{Proof of Corollary~\ref{cor:rough-all-line-projections}}
\label{subsec:proof-rough-all-line-projections}
 
Recall that for $0<r,q<1$  we
let $\nu_{r,q}\in \mathcal P(\mathbb R^3)$ be the self-similar measure
associated to the IFS
$$\Phi_{r} = \lbrace
f_{O,0}(x)=rOx, 
f_{O,1}(x)=rOx+e_1
\rbrace_{O\in\mathcal R_5}$$
and the probability vector
$$p_{O,0}=\frac{q}{6},
p_{O,1}=\frac{1-q}{6} \text{ for }  O\in\mathcal R_5.$$

Fix \(0<\delta<1\) and \(0\leq t<\delta/2\). Put
$s_0:=1+2t.$
Choose
$s_0<D<1+\delta.$
Recall that, for \(p=5\),
$\rho_5=\frac{\sqrt5}{3}.$
For \(q\in(0,1)\), put
\[
A_q:=q^2+(1-q)^2,
\qquad
B_q:=\sqrt2\,q(1-q)\rho_5,
\qquad
C_5:=\rho_5^2,
\]
and define
\[
\Theta_{5,q}
:=
\frac{
A_q+C_5+
\sqrt{(A_q-C_5)^2+4B_q^2}
}{2}.
\]
Write
$a(q):=-\log\Theta_{5,q},
\qquad
b(q):=-2\log q.$
As \(q\to1\), the formula for \(\Theta_{5,q}\) gives
\[
a(q)=2(1-q)+O((1-q)^2),
\qquad
b(q)=2(1-q)+O((1-q)^2).
\]
Thus
\[
\frac{b(q)}{a(q)}\longrightarrow1
\qquad\text{and}\qquad
a(q)\longrightarrow0
\qquad\text{as }q\to1.
\]
We may therefore choose \(q\in(0,1)\), sufficiently close to \(1\), so that
\begin{equation}
\label{eq:rough-parameter-choice}
D\frac{b(q)}{a(q)}<1+\delta
\end{equation}
and
\begin{equation}
\label{eq:rough-spectral-choice}
\frac{s_0(D+1)}{D-s_0}\frac{a(q)}{D}
<
-\log\rho_5 .
\end{equation}
Set
$r:=\exp\left(-\frac{a(q)}{D}\right)$ and
$\nu:=\nu_{r,q}.$
Then
$\frac{-\log\Theta_{5,q}}{-\log r}=D.$
Repeating the proof of Lemma~\ref{lem:mask-estimate}, but with
$m_R(\omega)
=
q+(1-q)e^{-2\pi iR\omega_1},$
gives a one-step contraction for the spherical average, with limiting
quadratic form
\[
\begin{pmatrix}
A_q & B_q\\
B_q & C_5
\end{pmatrix},
\]
whose largest eigenvalue is \(\Theta_{5,q}\). Iterating this estimate and
using the spherical-average energy identity
\eqref{eq energy id} (or
Lemma~\ref{lem:peres-schlag-line}) gives
\[
\dim_2\nu_{r,q}
\geq
\min\left\{
3,\frac{-\log\Theta_{5,q}}{-\log r}
\right\}.
\]
Since
$\frac{-\log\Theta_{5,q}}{-\log r}
=
D<3,$
we obtain
\begin{equation}
\label{eq:rough-example-lower-dimension}
\dim_2\nu_{r,q}\geq D>s_0.
\end{equation}

We next prove the upper bound on \(\dim_2\nu\). Let \(K=\operatorname{spt}\nu\),
and choose \(M>0\) such that
$K\subset B(0,M).$
For every \(O\in\mathcal R_5\), the map
$f_{O,0}(x)=rOx$
fixes the origin. Hence, for every word
$u=((O_1,0),\ldots,(O_n,0)),$
whose translations are all zero, one has
$f_u(K)\subset B(0,Mr^n).$
The total weight of all such words is \(q^n\). Therefore
\[
\nu\bigl(B(0,Mr^n)\bigr)\geq q^n.
\]
Consequently, for every \(s>0\),
\[
I_s(\nu)
\geq
(2Mr^n)^{-s}\nu\bigl(B(0,Mr^n)\bigr)^2
\geq
(2M)^{-s}\left(q^2r^{-s}\right)^n.
\]
If
$s>\frac{-2\log q}{-\log r},$
then \(q^2r^{-s}>1\), and letting \(n\to\infty\) gives
\(I_s(\nu)=\infty\). Hence
\[
\dim_2\nu
\leq
\frac{-2\log q}{-\log r}
=
D\frac{b(q)}{a(q)}
<
1+\delta,
\]
by \eqref{eq:rough-parameter-choice}. Thus
$1<\dim_2\nu<1+\delta.$
In particular, since \(\nu\) is compactly supported and \(\dim_2\nu<3\), we have
$\nu\notin L^2(\mathbb R^3).$

It remains to prove the projection statement. Since
\[
p_{O,0}+p_{O,1}=\frac16
\qquad
\text{for every }O\in\mathcal R_5,
\]
the rotational averaging operator associated to \(\nu\) is
\[
PF(g)=\frac16\sum_{O\in\mathcal R_5}F(O^{-1}g),
\qquad
F\in L^2(\operatorname{SO}(3)).
\]
Thus, by \eqref{eq: ramanujan},
$\|P\|_{L^2_0(\operatorname{SO}(3))}
=
\rho_5.$
Using \eqref{eq:rough-example-lower-dimension}, the fact that
$x\mapsto \frac{x+1}{x-s_0}$
is decreasing on \((s_0,\infty)\), and \eqref{eq:rough-spectral-choice}, we get
\[
\frac{s_0(\dim_2\nu+1)}{\dim_2\nu-s_0}(-\log r)
\leq
\frac{s_0(D+1)}{D-s_0}\frac{a(q)}{D}
<
-\log\rho_5
=
-\log\|P\|_{L^2_0(\operatorname{SO}(3))}.
\]
Thus Theorem~\ref{thm: baby case} applies with \(s=s_0\). Therefore, for every
\(V\in G_{3,1}\),
\[
\pi_V\nu\ll\operatorname{Leb}_V
\qquad\text{and}\qquad
\pi_V\nu\in H^{(s_0-1)/2}(V)=H^t(V).
\]
This completes the proof.
\subsection{Proof of Corollary \ref{cor:explicit-ambient-ac}}
\label{subsec: proof of corollary ac}

Let
$r:=\frac{1977}{2000},$
and let \(\nu\) be the uniform self-similar measure corresponding to the IFS
$\left\{
x\mapsto rOx,\,
x\mapsto rOx+e_1
\right\}_{O\in\mathcal R_5}.$
For every \(0<\varepsilon<1/10\), the measure \(\nu\) is the
\(\mathbb R^3\)-marginal of \(\mu_{r,\varepsilon}\) from
Lemma~\ref{lem:mask-estimate}. Hence, taking
\(\varrho=1\) and \(z=0\) in that lemma and iterating, we obtain, for every
$0<\beta<\beta_0,
\,
\beta_0:=
\frac{
\log\left(\frac{36}{19+\sqrt{91}}\right)
}{
\log\left(\frac{2000}{1977}\right)
},$
\begin{equation}
\label{eq:spherical-average-ambient-example}
\mathcal S_\nu(R)
:=
\int_{S^2}
\left|\mathcal F_{R\omega}(\nu)\right|^2\,d\sigma(\omega)
\lesssim_\beta
R^{-\beta},
\qquad R\geq1.
\end{equation}
Numerically,
$\beta_0>20.$

We first prove the Sobolev regularity. Let \(0\leq t<3.9\). Choose
\(\beta\) such that
$2t+3<\beta<\beta_0.$
Using polar coordinates and \eqref{eq:spherical-average-ambient-example}, we get
\[
\begin{aligned}
\int_{\mathbb R^3}
(1+|\xi|^2)^t
|\mathcal F_\xi(\nu)|^2\,d\xi
&\lesssim_t
1+
\int_1^\infty
R^{2t+2}\mathcal S_\nu(R)\,dR
\\
&\lesssim_{\beta,t}
1+
\int_1^\infty
R^{2t+2-\beta}\,dR
<\infty.
\end{aligned}
\]
Therefore, by Plancherel's theorem, for every \(0\leq t<3.9\),
\[
\nu\ll\operatorname{Leb}_{\mathbb R^3}
\qquad\text{and}\qquad
\nu\in H^t(\mathbb R^3).
\]

Taking \(t>3/2\), we see that the density is continuous and vanishes at
infinity. Taking \(t>2+3/2\), the Sobolev embedding theorem gives
$\nu\in C^2_0(\mathbb R^3).$
Since \(\nu\) is a probability measure with a continuous density, there is a
point where this density is positive. Hence it is positive on some non-empty
open ball, and therefore
$\operatorname{Int}_{\mathbb R^3}(\operatorname{spt}\nu)\neq\emptyset.$

Moreover, the same Sobolev regularity also implies that \(\nu\) has full
Fourier dimension. Indeed, let \(f=d\nu/dx\). Since \(\nu\) is compactly
supported and \(f\in H^2(\mathbb R^3)\), the second weak derivatives of \(f\)
belong to \(L^1(\mathbb R^3)\). Hence, for \(|\xi|\geq1\), choosing
\(1\leq j\leq3\) with \(|\xi_j|\geq |\xi|/\sqrt3\), we have
\[
|\mathcal F_\xi(\nu)|
=
|\widehat f(\xi)|
\leq
(2\pi|\xi_j|)^{-2}
\left\|\partial_j^2 f\right\|_{L^1}
\lesssim
|\xi|^{-2}.
\]
It follows that \(\dim_F\nu=3\). Since also
\(\nu\ll\operatorname{Leb}_{\mathbb R^3}\), we have \(\dim_H\nu=3\). Thus
\(\nu\) is Salem.

\subsection{Proof of Remark~\ref{rmk: lyp}}
\label{subsec:lyapunov-variant}

We repeat the proof of Theorem~\ref{thm:main}, replacing the deterministic
lower bound on the stopping time by a large-deviation estimate. Keep the
notation of that proof, and put
\[
\Lambda_V
:=
-\log\left(
\sum_{j=1}^m
\widetilde p_j
\left\|
P_{\rho_j,V}
\right\|_{L^2_0(\mathcal O_V,m_V)}
\right).
\]

Choose \(\alpha>0\) such that
\begin{equation}
\label{eq:lyap-alpha-main}
\alpha(S-b)>b
\end{equation}
and
\begin{equation}
\label{eq:lyap-alpha-error}
b(q-1)+\bigl(\sigma+b(q-1)\bigr)\alpha
<
\Lambda_V\frac{1-\varepsilon}{\chi},
\end{equation}
and, when \(R>0\),
\begin{equation}
\label{eq:lyap-alpha-deviation}
b(q-1)+\bigl(\sigma+b(q-1)\bigr)\alpha
<
\frac{2\varepsilon^2\chi}{(1-\varepsilon)R^2}.
\end{equation}
Such an \(\alpha\) exists, since
$\inf_{\alpha>b/(S-b)}
\left[
b(q-1)+\bigl(\sigma+b(q-1)\bigr)\alpha
\right]
=
C_{q,b,S,\sigma},$
and by the assumptions in Remark~\ref{rmk: lyp}.

We now repeat the scale decomposition in the proof of
Theorem~\ref{thm:main}, using the stopping levels \(L_n=n\). The bounded
overshoot
$0\leq\varepsilon_{\eta,n}\leq\lambda_{\max}$
only changes the implicit constants. The main-term estimate is unchanged:
its additional exponential factor is
$\exp\left(
(q-1)n\bigl[b-\alpha(S-b)\bigr]
\right),$
which is summable by \eqref{eq:lyap-alpha-main}.

It remains to estimate the mixing error. Set
\[
B_j
:=
\left\|
P_{\rho_j,V}
\right\|_{L^2_0(\mathcal O_V,m_V)},
\qquad
B
:=
\sum_{j=1}^m\widetilde p_jB_j
=
e^{-\Lambda_V},
\]
and, for \(\omega\in\Omega\) and \(\ell\in\mathbb N\), put
$Y_\ell(\omega)
:=
\prod_{a=1}^{\ell}B_{\omega_a}.$
Up to a multiplicative constant, the error term is bounded by
\begin{equation}
\label{eq:lyap-error-series}
\sum_{n\geq1}
\exp\left(
n\left[
b(q-1)+\bigl(\sigma+b(q-1)\bigr)\alpha
\right]
\right)
\mathbb E_Q\left(Y_{\tau_n}\right).
\end{equation}

Let
$A_n
:=
\left\{
\tau_n\leq\frac{(1-\varepsilon)n}{\chi}
\right\}.$
By Lemma~\ref{lem:stopping-model},
$Q(A_n)
\leq
\exp\left(
-\frac{2\varepsilon^2\chi}{(1-\varepsilon)R^2}\,n
\right),$
with the convention that \(Q(A_n)=0\) when \(R=0\).

Set
$M_n
:=
\left\lfloor
\frac{(1-\varepsilon)n}{\chi}
\right\rfloor.$
On \(A_n^c\), we have \(\tau_n>M_n\), and hence, since
\(0\leq B_j\leq1\),
$Y_{\tau_n}\leq Y_{M_n}.$
Therefore,
\[
\mathbb E_Q\left(Y_{\tau_n}\right)
\leq
Q(A_n)+\mathbb E_Q\left(Y_{M_n}\right).
\]
By independence,
\[
\mathbb E_Q\left(Y_{M_n}\right)
=
B^{M_n}
\lesssim
\exp\left(
-\Lambda_V\frac{1-\varepsilon}{\chi}\,n
\right).
\]
Consequently,
\begin{equation}
\label{eq:lyap-stopped-bound}
\mathbb E_Q\left(Y_{\tau_n}\right)
\lesssim
\exp\left(
-\Lambda_V\frac{1-\varepsilon}{\chi}\,n
\right)
+
\exp\left(
-\frac{2\varepsilon^2\chi}{(1-\varepsilon)R^2}\,n
\right),
\end{equation}
where the second term is omitted when \(R=0\).

Substituting \eqref{eq:lyap-stopped-bound} into
\eqref{eq:lyap-error-series}, we obtain at most two geometric series.
They converge by
\eqref{eq:lyap-alpha-error} and
\eqref{eq:lyap-alpha-deviation}. The remainder of the proof of
Theorem~\ref{thm:main} is unchanged and yields
$\pi_V\nu\in B^{(b-k)/q'}_{q,q}(V).$
In particular,
$\pi_V\nu\ll\operatorname{Leb}_V.$
This proves Remark~\ref{rmk: lyp}.

\bibliographystyle{plain}
\bibliography{bib}

\end{document}